\documentclass[french,11pt]{article}   
\usepackage{lmodern}            
\usepackage[T1]{fontenc}         
\usepackage[utf8]{inputenc}      
\usepackage{bbold}
\usepackage{xspace}
\usepackage{biblatex} 
\usepackage{amsmath,amssymb}
\usepackage{amsthm}
\usepackage{newtxmath}
\usepackage{stmaryrd}
\usepackage{dsfont}
\usepackage[margin=3.2cm]{geometry}
\usepackage{hyperref}
\usepackage{xcolor}
\usepackage{enumitem}
\hypersetup{
    colorlinks,
    linkcolor={black},
    citecolor={black},
    urlcolor={black}}
\usepackage{graphicx}
\graphicspath{ {./images/} }
\input{insbox}

\newtheorem*{remark}{Remark}
\newtheorem{theorem}{Theorem}[section]
\newtheorem{corollary}[theorem]{Corollary}
\newtheorem{lemme}[theorem]{Lemma}
\newtheorem{proposition}[theorem]{Proposition}

\theoremstyle{definition}
\newtheorem{definition}{Definition}[section]

\newcommand{\R}{\mathbb{R}}
\newcommand{\N}{\mathbb{N}}

\newcommand{\E}{\mathbb{E}}
\newcommand{\p}{\mathbb{P}}

\newcommand{\1}{\mathbf{1}}

\newcommand{\T}{\mathcal{T}}
\newcommand{\D}{\mathcal{D}}
\newcommand{\s}{\mathcal{S}}
\newcommand{\Q}{\mathbb{Q}}
\newcommand{\BP}{\overline{\mathcal{BP}}}

\title{The bigeodesic Brownian plane}
\author{Mathieu Mourichoux}
\date{}

\begin{document}

\maketitle

\begin{abstract}
    We introduce and study a random non-compact space called the bigeodesic Brownian plane, and prove that it is the tangent plane in distribution of the Brownian sphere at a point of its simple geodesic from the root (for the local Gromov-Hausdorff-Prokhorov-Uniform topology). We also show that it  is the local limit of the Brownian plane rerooted further and further on its unique infinite geodesic. Furthermore, we discuss various properties of this space, such as its topology, the behavior of its geodesic rays, and its invariance in distribution under several natural transformations.
\end{abstract}

\tableofcontents

\section{Introduction}

This work is concerned with models of random geometry that arise as scaling limits of random maps. In particular, we consider the Brownian sphere $(\s,D)$, which appears as the scaling limit (for the Gromov-Hausdorff topology) of quadrangulations of the sphere with $n$ faces chosen uniformly at random (\cite{convergence,uniqueness}), and more generally of several models of random planar maps (see \cite{ConvergenceBJM,ConvergenceSimple,AbrahamConvergence} for instance). The Brownian sphere also comes with a volume measure $\mu$, a root $\rho$ and distinguished point $x_*$ (these points can be seen as distributed according to the measure $\mu$). The construction of the Brownian sphere is based on the Brownian snake excursion measure $\N_0$, which is a $\sigma-$ finite measure that  will be introduced in Section \ref{section 2}. Moreover, this measure also gives a label to every point $x\in\s$, and the distinguished point $x_*$ carries the minimal label $W_*$. In this paper, we are primarily interested in the study of the Brownian sphere around a point of one of its geodesics.\\

In the last decade, much work has been done to understand the behavior of the Brownian sphere around points that are ``typical'' (i.e. $\mu$-a.e. point). For instance, in his paper \cite{geodesic1}, Le Gall completely describes the geodesics to a distinguished point $x_*$. In particular, we can easily identify a unique geodesic $\Gamma=(\Gamma_t)_{0\leq t\leq D(\rho,x_*)}$ between the root $\rho$ of the Brownian sphere and the distinguished point $x_*$ (let us mention that $D(\rho,x_*)=-W_*$). He also shows a confluence phenomenon for these geodesics. Then, Curien and Le Gall introduced the Brownian plane $(\mathcal{BP},D_\infty)$ in \cite{brownianplane}, which is a non-compact random surface. They proved that it is the tangent plane in distribution of the Brownian sphere at a typical point (for instance around $\rho$ and $x_*$), and that there is a unique infinite geodesic $\gamma$ starting from the root. Afterwards, several properties of this space were studied, such as the volume of its hulls and its isoperimetric properties (see \cite{Hullprocess2016,Isoperimetric}). It is often more convenient to work with the Brownian plane and to use coupling methods to transfer results to the Brownian sphere.\\
Our work may be seen as a counterpart of theirs, but for some ``atypical'' points (i.e. points that belong to a set of $\mu$-measure zero). Indeed, it was proved in \cite{geodesic2} that the set of points that are in the interior of geodesics has Hausdorff dimension 1, whereas the Brownian sphere has Hausdorff dimension 4 (see \cite{dimension}). More precisely, we identify the tangent plane in distribution of the Brownian sphere at a point of $\Gamma$ at fixed distance of $\rho$. It is a random non-compact rooted space that we call the bigeodesic Brownian plane, and we use the notation $(\BP,\overline{D}_\infty,\overline{\rho}_\infty)$ (we will give a precise definition of this space in Subsection \ref{presentation}). We underline that $(\BP,\overline{D}_\infty,\overline{\rho}_\infty)$ has a totally explicit construction, based on labelled trees. The main result of this paper is the following (we use the notation $B_r(X,x)$ to denote the ball of radius $r$ around $x$ in a metric space $(X,d)$, and we omit $x$ if it is the root of $X$).

\begin{theorem}\label{convergence}
    For every $\delta>0$ and $b>0$, we can find $\varepsilon>0$ such that we can construct $(\s,D,\rho)$ (under $\N_0(\cdot\,|\,W_*<-b)$) and $(\overline{\mathcal{BP}},\overline{D}_\infty,\overline{\rho}_\infty)$ on the same probability space such that $B_\varepsilon(\s,\Gamma_b)$ and $B_\varepsilon(\overline{\mathcal{BP}})$ are isometric with probability at least $1-\delta$. Moreover, we have the convergence 
    \begin{equation}        (\s,\varepsilon^{-1}D,\Gamma_b)\xrightarrow[\varepsilon\rightarrow 0]{(d)}(\overline{\mathcal{BP}},\overline{D}_\infty,\overline{\rho}_\infty)
    \end{equation}
    in distribution for the local Gromov-Hausdorff topology. 
\end{theorem}

(Conditioning on the event $\{W_*<-b\}$ ensures that $D(\rho,x_*)>b$, so that $\Gamma_b$ is well-defined). Note that the limiting space does not depend on $b$. We also show that this convergence result also holds under various conditionings of the Brownian sphere (see Theorems \ref{convergence conditionnée} and \ref{conditionnement volume}). Consequently, the convergence also holds if we pick a random point of the interior of $\Gamma$ (in a way that only depends on the length of $\Gamma$). However, note that some points exhibit a different local behavior (for instance, confluence points of geodesics). The following result makes a link between the spaces $\mathcal{BP}$ and $\BP$ (recall that $\gamma=(\gamma_t)_{t\geq0}$ is the unique infinite geodesic starting from the root in the Brownian plane). 

\begin{theorem}\label{convergence2}
For every $\delta>0$ and $r>0$, we can find $b_0>0$ such that for every $b\geq b_0$ we can construct $(\overline{\mathcal{BP}},\overline{D}_\infty,\overline{\rho}_\infty)$ and $(\mathcal{BP},D_\infty,\rho_\infty)$ such that the balls $B_r(\overline{\mathcal{BP}})$ and $B_r(\mathcal{BP},\gamma_b)$ are isometric with probability at least $1-\delta$. Furthermore, we have 
\begin{equation*}    (\mathcal{BP},D_\infty,\gamma_b)\xrightarrow[b\rightarrow\infty]{(d)}(\overline{\mathcal{BP}},\overline{D}_\infty,\overline{\rho}_\infty)
\end{equation*}
in distribution for the local Gromov-Hausdorff topology.      
\end{theorem}

Let us mention that these results still hold if we replace the Gromov-Hausdorff topology by the Gromov-Hausdorff-Prokhorov-Uniform topology (these topologies will be introduced in Section \ref{section 2}). We emphasize that the limiting space $(\overline{\mathcal{BP}},\overline{D}_\infty,\overline{\rho}_\infty)$ has a totally explicit construction, which can be used to derive some of its properties. For instance, the random space $\BP$ has a nice geodesic structure which is different from other Brownian surfaces, as shown in the following theorem.

\begin{theorem}\label{bigeodesic}
    Almost surely, there is a unique infinite bigeodesic $\Gamma_\infty$ in $\BP$.
\end{theorem}

Moreover, we prove several invariance properties of $\BP$. We hope that these invariance properties (together with the coupling result of Theorem \ref{convergence}) could be used to prove new results about geodesics in the Brownian sphere.

Let us discuss some results that are related to this work. The paper \cite{dieuleveut} deals with the local behavior of the Uniform Infinite Planar Quadrangulation (which  is a discrete version of the Brownian plane) around an infinite geodesic. In particular, Dieuleveut introduced the $\overline{\text{UIPQ}}$, which is the local limit of the UIPQ along its infinite geodesic ray. Therefore, Theorem \ref{convergence2} is the continuous counterpart of this result. We expect that $\BP$ is the scaling limit of the $\overline{\text{UIPQ}}$, which will be the subject of another paper. Finally, similar results were obtained by Basu, Bhatia and Ganguly in \cite{basu2021environment} for LQG-surfaces, which is a model of random continuum geometry that contains Brownian surfaces as a special instances. Therefore, according to the link between $\gamma-$Liouville Quantum Gravity and the Brownian sphere proved in a series of papers \cite{MillerSchefield0,MillerSchefield1,MillerSchefield2,MillerSchefield3}, the convergence result in Theorem \ref{convergence2} is close to their result for the parameter $\gamma=\sqrt{\frac{8}{3}}$ (in their paper, they place themselves at a random point of the geodesic that goes to infinity). However, as mentioned previously, our construction of the limiting space allows us to prove a stronger version of convergence (i.e. the coupling result) and to derive many of its properties.

This paper is organized as follows. Section \ref{section 2} gives some preliminaries about the Brownian snake excursion measure $\N_0$, the Brownian sphere and the Gromov-Hausdorff topology. In section \ref{section 3}, we give several spine decompositions of the random tree $\T$ under $\N_0(\cdot\,|\,W_*<-b)$. Section \ref{Section4} is devoted to the proof of a key lemma, and Section \ref{section 5} contains the proof of the main convergence result. Finally, in Section \ref{section 6}, we study various properties of the space $\BP$.\\

\subsection*{Acknowledgements}

I would like to thank warmly Grégory Miermont, for his supervision throughout this work and for precious conversations. I also thank Erich Baur, for his initial input to this project, and Manan Bhatia for comments on a previous draft of this paper. Finally, I wish to thank Lou Le Bihan and Simon Renouf for their essential help in typing this article.

\section{Preliminaries}\label{section 2}

\subsection{Snake trajectories}\label{intervalle}
Here, we recall the construction and some basic properties of the Brownian snake (see \cite{serpent} for more details). A finite path is a continuous function $w:[0,\zeta]\longrightarrow \R$, where $\zeta(w)\geq 0$ is called the lifetime of $w$, and we set $\hat{w}=w(\zeta_{(w)})$. We write $\mathcal{W}$ for the set of all finite paths in $\R$, and for every $x\in\R$, $\mathcal{W}_x:=\{w\in\mathcal{W}:w(0)=x\}$. The set $\mathcal{W}$ is a Polish space when equipped with the distance
\begin{equation*}
    d(w,w')=|\zeta(w)-\zeta(w')|+\sup_{t\geq 0}|w(t\wedge\zeta(w))-w'(t\wedge\zeta(w'))|.
\end{equation*}
Finally, we identify the point $x\in\R$ with the element of $\mathcal{W}_x$ with zero lifetime. 
\begin{definition}
Fix $x\in\R$. A snake trajectory starting from $x\in\R$ is a continuous mapping $s\rightarrow\omega_s$ from $\R_+$ to $\mathcal{W}_x$ which satisfies the following conditions : 
\begin{itemize}
    \item $\omega_0=x$ and the quantity $\sigma(\omega)=\sup\{s\geq 0:\omega_s\neq x\}$ is finite, and is called the duration of the snake trajectory $\omega$. 
    \item For every $0\leq s\leq s'$, we have $\omega_s(t)=\omega_{s'}(t)$ for every $t\in[0,\min_{s\leq r\leq s'}\zeta(\omega_r)]$. 
\end{itemize}
\end{definition}

We will denote by $\s_x$ the set of snake trajectories starting from $x\in\R$, and $\s=\bigcup_{x\in\R}\s_x$ the set of all snake trajectories. We will use the notation $W_s(\omega)=\omega_s$ and $\zeta_s(\omega)=\zeta(\omega_s)$. Note that a snake trajectory $\omega$ is completely determined by its lifetime function $s\rightarrow\zeta_s(\omega)$ and its tip function $s\rightarrow\widehat{W}_s(\omega)$ (see \cite{Refserpent} for a proof). We also write $W_*(\omega)=\inf_{t\geq0}\widehat{W}_t(\omega)$.\\
Given a snake trajectory $\omega\in\s$, its lifetime function $\zeta(\omega)$ encodes a compact $\R$-tree, which will be denoted by $\T_{\omega}$. More precisely, if we introduce a pseudo-distance on $[0,\sigma(\omega)]$ by letting 
\begin{equation*}
    d_{(\omega)}(s,s')=\zeta_s(\omega)+\zeta_{s'}(\omega)-2\min_{s\wedge s'\leq r\leq s\vee s'}\zeta_r(\omega),
\end{equation*}
$\T_{\omega}$ is the quotient space $[0,\sigma(\omega)]/\{d_{(\omega)}=0\}$ equipped with the distance induced by $d_{(\omega)}$. We write $p_\T$ for canonical projection, and root the tree $\T_{\omega}$ at $\rho_\T:=p_\T(0)=p_\T(\sigma(\omega))$. The tree $\T_{\omega}$ also comes with a volume measure, which is the pushforward of the Lebesgue measure on $[0,\sigma(\omega)]$ under the projection $p_\T$. Finally, note that because of the snake property, $W_s(\omega)=W_{s'}(\omega)$ if $p_\T(s)=p_\T(s')$. Similarly, the mapping $s\rightarrow\widehat{W}_s(\omega)$ can be viewed as a function on the tree $\T_\omega$. In this article, for $s\in[0,\sigma(\omega)]$ (or $u\in\T_\omega$), we will often use the notation $\hat{W}_s(\omega)=Z_s=Z_u$ if $u=p_\T(s)$. \\
We also define intervals on the tree $\T_\omega$ as follows.  For every $s,t\in[0,\sigma]$ with $t<s$, we use the convention that $[s,t]=[s,\sigma]\cup[0,t]$. For every $x,y\in\T_\omega$, there is a smallest interval $[s,t]$ such that $p_\T(s)=x$ and $p_\T(t)=y$, and we define 
\[[x,y]:=\{p_\T(r):r\in[s,t]\}.\]

\subsection{The Brownian snake excursion measure}\label{snake}

For every $x\in\R$, we define a $\sigma$-finite measure on $\s_x$, called the Brownian snake excursion measure and denoted as $\N_x$, as follows. Under $\N_x$ :
\begin{enumerate}
    \item The lifetime function $(\zeta_s)_{s\geq 0}$ is distributed according to the Itô measure of positive excursions of linear Brownian motion, normalized so that the density of $\sigma$ under $\N_x$ is $t\rightarrow(2\sqrt{2\pi t^3})^{-1}$.
    \item Conditionally on $(\zeta_s)_{s\geq 0}$, the tip function $(\widehat{W}_s)_{s\geq0}$ is a Gaussian process with mean $x$ and covariance function : 
    \begin{equation*}
        K(s,s')=\min_{s\wedge s'\leq r\leq s\vee s'}\zeta_r.
    \end{equation*}
\end{enumerate}
The measure $\N_x$ is also an excursion measure away from $x$ for the Brownian snake, which is a Markov process in $\mathcal{W}_x$. For every $t>0$, we can define the conditional probability measure $\N_x^{(t)}=\N_x(\cdot\,|\,\sigma=t)$, which can also be constructed by replacing the Itô measure used to define $\N_x$ by a Brownian excursion with duration $t$. \\
For every $y<x$, we have
\begin{equation}\label{inf}
    \N_x(W_*<y)=\frac{3}{2(x-y)^2}.
\end{equation}
(see \cite{serpent} for a proof). Therefore, we can define the conditional probability measure $\N_x(\cdot\,|\,W_*<y)$. Moreover, one can prove that under $\N_x$ or $\N_x^{(t)}$, a.e, there exists a unique $s_*\in[0,\sigma]$ such that $\widehat{W}_{s_*}=W_*$ (see e.g. Proposition 2.5 in \cite{Conditionnedbrowniantrees}). \\
Finally, these measures satisfy a scaling property. For every $\lambda>0$ and $\omega\in\s_x$, we define $\Theta_\lambda(\omega)\in\s_{x\sqrt{\lambda}}$ by $\Theta_\lambda(\omega)=\omega'$ with
\begin{equation}\label{scaling}
    \omega'_s(t):=\sqrt{\lambda}\,\omega_{s/\lambda^2}(t/\lambda),\quad\text{for $s\geq 0$ and $0\leq t\leq\zeta_s':=\lambda\zeta_{s\lambda^2}$}. 
\end{equation}
Then, the pushforward of $\N_x$ by $\Theta_\lambda$ is $\lambda\N_{x\sqrt{\lambda}}$, and for every $t>0$, the pushforward of $\N_x^{(t)}$ by $\Theta_\lambda$ is $\N_{x\sqrt{\lambda}}^{(\lambda^2t)}$.\\
We can also classify the values of the Brownian snake excursion according to some subtrees branching of a spinal path. More precisely, on the event $\{\sigma>T\}$, consider the intervals $(\alpha_i,\beta_i)_{i\in I}$ that are the connected components of the open set $\{t\in[T,\sigma],\,\zeta_t>\min_{T\leq r\leq t}\zeta_r\}$. For every $i\in I$, we define $W^i\in C(\R_+,\mathcal{W})$ in the following way for every $s\geq0$ : 
\[W^i_s(t):=W_{(\alpha_i+s)\wedge\beta_i}(\zeta_{\alpha_i}+t),\quad0\leq t\leq\zeta_{(\alpha_i+s)\wedge\beta_i}-\zeta_{\alpha_i}.\]
The following lemma is proved in \cite[Lemma V.5]{serpent}. 
\begin{lemme}\label{excursion}
    Under $\N_0(\cdot\,|\,\sigma>T)$ and given $W_T$, the point measure 
    \[\sum_{i\in I}\delta_{(\zeta_{\alpha_i},W^i)}(dt,d\omega)\]
    is a Poisson point measure on $\R_+\times C(\R_+,\mathcal{W})$ with intensity 
    \[2\1_{[0,\sigma]}(t)dt\N_{\hat{W}_T(t)}(d\omega).\]
\end{lemme}
\begin{remark}
    In \cite{serpent}, this lemma concerned the Brownian snake driven by a reflected Brownian motion. However, the previous statement can be obtained from the original one by typical arguments from excursion theory.
\end{remark}
\subsection{The Brownian sphere}\label{sphere}

Fix a snake trajectory $\omega\in\s_0$ with duration $\sigma$.  We introduce, for every $x,y\in\T_\omega$
\[D^\circ_{(\omega)}(x,y)=Z_x+Z_y-2\max\bigg(\min_{r\in[x,y]}Z_r,\min_{r\in[y,x]}Z_r\bigg)\]
and  
\[D_{(\omega)}(x,y)=\inf\bigg\{\sum_{i=1}^p D^\circ_{(\omega)}(x_i,x_{i-1})\bigg\}\]
where the infimum is taken over all integers $p\geq 1$ and sequences $x_0,...,x_p\in\T_\omega$ such that $x_0=x$ and $x_p=y$. Note that $D_{(\omega)}\leq D^\circ_{(\omega)}$.\\
Observe that $D^\circ_{(\omega)}(x,y)\geq|Z_x-Z_y|$, which translates into a simple (but very useful) bound: 
\begin{equation}\label{Bound}
    D_{(\omega)}(x,y)\geq|Z_x-Z_y|.
\end{equation}
The mapping $(x,y)\mapsto D_{(\omega)}(x,y)$ defines a pseudo-distance on $\T_\omega$. This allows us to introduce a quotient space $\T_\omega/\{D_{(\omega)}=0\}$, which is equipped with the distance naturally induced by $D_{(\omega)}$. \\
We can now apply the previous construction with a random snake trajectory. 
\begin{definition}
    The standard Brownian sphere is defined under the probability measure $\N_0^{(1)}$ as the random metric space $\s=\T/\{D=0\}$ equipped with the distance $D$, and a volume measure $\mu$ which is the pushforward of the volume measure on $\T$ under the canonical projection $p_\s$. 
\end{definition}

We also introduce the free Brownian sphere, which is defined in the same way replacing $\N_0^{(1)}$ by $\N_0$; even though this is not a random variable anymore, it is often more convenient to work with this object. Note that we can also see the standard Brownian sphere (or the free Brownian sphere) as a quotient of $[0,1]$ (or $[0,\sigma]$). We will sometimes use this point of view, and we write $\mathbf{p}:[0,1]\rightarrow\s$ for the canonical projection. We also set $\rho=\mathbf{p}(0)$. \\
As mentioned earlier, almost surely, there exists a unique $s_*\in[0,1]$ such that $Z_{s_*}=\inf_{t\in[0,1]}Z_t=W_*$, and we write $x_*=\mathbf{p}(s_*)$. Note that the bound \eqref{Bound} together with the inequality $D\leq D^\circ$ implies that almost surely, for every $s\in\s$, 
\[D(s,x_*)=Z_s-Z_{s_*}.\]
In particular, we have 
\[D(\rho,x_*)=-W_*.\]
The following proposition, proved in \cite{TopologicalStructure}, completely characterizes the points of $\T$ that are identified in the Brownian sphere.
\begin{proposition}\label{Identification}
    Almost surely, for every $x,y\in\T$, we have 
    \[D(x,y)=0\Longrightarrow D^\circ(x,y)=0.\]
\end{proposition}
In this work, we will be interested in geodesics in the Brownian sphere. We recall a few results about these objects. For every $b\geq 0$, we set 
\[T_b=\inf\{t\geq 0,Z_t=-b\}\quad\text{and}\quad\widehat{T}_b=\sup\{t\geq 0,Z_t=-b\}\]
(by convention, we let $\inf\{\varnothing\}=+\infty$). The following result was proved in \cite[Theorem 7.4]{geodesic1}. 

\begin{proposition}
    Almost surely, there exists a unique geodesic path $\Gamma=(\Gamma_b)_{0\leq b\leq-W_*}$ in $\s$ between $\rho=\mathbf{p}(0)$ and $x_*$, and for every $b\in[0,-Z_{s_*}]$, we have  
    \[\Gamma_b=\mathbf{p}(T_b)=\mathbf{p}(\widehat{T}_b).\]
\end{proposition}

\subsection{Gromov-Hausdorff convergence}

In this subsection, we recall some facts about geometry and convergence of metric spaces. \\ A pointed metric space $(E,d,\rho)$ is a metric space with a distinguished point $\rho\in E$. For every $r>0$, we denote the closed ball of radius $r$ centered at $\alpha\in E$ by $B_r(E,\alpha)$. If $\alpha=\rho$, we simply write $B_r(E)$. \\
If $X$ and $Y$ are two compact subsets of a metric space $(E,d)$, we define the Hausdorff distance between $X$ and $Y$ by 
\[d_\mathrm{H}^E(X,Y)=\inf\{\varepsilon>0:X\subset Y^\varepsilon \text{ and }Y\subset X^\varepsilon\}\]
where $X^\varepsilon=\{x\in E:d(x,X)\leq \varepsilon\}$. If $(E,d,\rho)$ and $(E',d',\rho')$ are two pointed compact metric spaces, the Gromov-Hausdorff distance between those spaces is 
\[d_\mathrm{GH}(E,E')=\inf\{d_\mathrm{H}^F(\phi(E),\phi'(E'))\vee \delta(\phi(\rho),\phi'(\rho'))\}\]
where the infimum is taken over all possible isometric embeddings $\phi,\phi'$ of $E,E'$ into a common metric space $(F,\delta)$. If $\mathbb{K}$ denotes the set of all isometric classes of pointed compact metric spaces, then $d_{\mathrm{GH}}$ induces a distance on $\mathbb{K}$, and $(\mathbb{K},d_{\mathrm{GH}})$ is a Polish space. \\ 
We also need to introduce the local Gromov-Hausdorff topology for non-compact spaces, but to avoid technicalities, we will restrict ourselves to boundedly compact length space. \\
Let $(E,d)$ be a metric space, and $\gamma$ a continuous curve of $E$ defined on $[0,T]$. We define the length of $\gamma$ by 
\[\text{length}(\gamma)=\sup_{0=t_0<...<t_k=T}\sum_{i=1}^k d(\gamma(t_i),\gamma(t_{i-1}))\]
where the supremum is taken over all subdivisions of $[0,T]$. The space $(E,d)$ is called a length space if, for every $x,y\in E$, the distance $d(x,y)$ is the infimum of the lengths of continuous paths between $x$ and $y$; it is called boundedly compact if the closed balls of $(E,d)$ are compact. Note that, by the Hopf-Rinow Theorem, it is equivalent to the property that $(E,d)$ is complete and locally compact. \\
Let $((E_n,d_n,\rho_n))_{n\geq1}$ and $(E,d,\rho)$ be pointed boundedly compact length spaces. We say that $(E_n,d_n,\rho_n)$ converges towards $(E,d,\rho)$ for the local Gromov-Hausdorff topology if, for every $r\geq 0$ we have
\[d_\mathrm{GH}(B_r(E_n),B_r(E))\xrightarrow[n\rightarrow\infty]{}0.\]
This notion of convergence is compatible with the distance 
\[d_\mathrm{LGH}(E,E')=\sum_{k\geq1}2^{-k}(d_{\mathrm{GH}}(B_k(E),B_k(E'))\wedge 1).\]
As previously, the set $\mathbb{K}_{\mathrm{bcl}}$ of all isometry classes of pointed boundedly compact length spaces, equipped with the distance $d_{\mathrm{LGH}}$, is a Polish space. \\
Consequently, if $(X_n)_{n\geq1}$ and $X$ are random variables taking values in $\mathbb{K}_{\mathrm{bcl}}$, then $X_n$ converges in distribution towards $X$ in $\mathbb{K}_{\mathrm{bcl}}$ if and only if, for every $r\geq0$, $B_r(X_n)$ converges in distribution to $B_r(X)$ in $\mathbb{K}$.\\
We also present the Gromov-Hausdorff-Prokhorov-Uniform topology, which generalizes the Gromov-Hausdorff topology (see \cite{scalingUIHPQ} for a more detailed introduction to this topology and its properties). As previously, we begin with the compact case. \\
Consider $(E,d)$ a compact metric space, equipped with two Borel measures $\mu$ and $\nu$. We define the Prokhorov distance between $\mu$ and $\nu$ by the following formula
\begin{equation*}
    d_\mathrm{P}^E(\mu,\nu)=\inf\{\varepsilon>0:\forall A\in\mathcal{B}(E),\mu(A)\leq\nu(A^\varepsilon)+\varepsilon\,\text{and}\,\nu(A)\leq\mu(A^\varepsilon)+\varepsilon\}
\end{equation*}
(one can check that this defines a distance).\\
Then, let $C_0(\R,E)$ be the set of continuous curves $\eta:\R\rightarrow E$ such that for each $\varepsilon>0$, there exists $T>0$ such that $d(\eta(t),\eta(T))\leq\varepsilon$ and $d(\eta(-t),\eta(-T))\leq\varepsilon$ when $t\geq T$. If $\eta$ is a continuous curve defined on a compact interval $[a,b]$, we identify it with the element of $C_0(\R,E)$ which agrees with $\eta$ on $[a,b]$, and that satisfy $\eta(t)=\eta(a)$ if $t\leq a$ and $\eta (t)=\eta(b)$ if $t\geq b$. We equip $C_0(\R,E)$ with the $d-$uniform distance :
\begin{equation*}
    d_\mathrm{U}^E(\eta,\eta')=\sup_{t\in\R}d(\eta(t),\eta'(t)).
\end{equation*}
Now, let $K_\mathrm{{GHPU}}$ be the set of $4-$tuples $\mathcal{X}=(E,d,\mu,\eta)$, where $(E,d)$ is a compact metric space, $\mu$ a finite Borel measure on $E$, and $\eta\in C_0(\R,E)$ (note that the space $E$ has a natural root, which is $\eta(0)$).\\
For any elements $\mathcal{X}=(E,d,\mu,\eta)$ and $\mathcal{X'}=(E',d',\mu',\eta')$ of $K_{\mathrm{GHPU}}$, we define
\begin{equation*}
    d_{\mathrm{GHPU}}(\mathcal{X},\mathcal{X'})=\inf \{d_{\mathrm{H}}^F(\phi(E),\phi'(E'))+d_\mathrm{P}^F((\phi)_*\mu,(\phi')_*\mu')+d_\mathrm{U}^F(\phi\circ\eta,\phi'\circ\eta')\}
\end{equation*}
where the infimum is taken over all possible isometric embeddings $\phi,\phi'$ of $E,E'$ into a common metric space $(F,\delta)$.
\begin{proposition}
    Let $\mathbb{K}_\mathrm{{GHPU}}$ be the set of equivalence classes of $K_\mathrm{{GHPU}}$ under the equivalence relation where $(E,d,\mu,\eta)\sim(E',d',\mu',\eta')$ if there exists an isometry $f:(E,d)\rightarrow(E',d')$ such that $f_*\mu=\mu'$ and $f\circ \eta=\eta'$. Then $d_\mathrm{{GHPU}}$ induces a distance on $\mathbb{K}_\mathrm{{GHPU}}$, and $(\mathbb{K}_\mathrm{{GHPU}},d_\mathrm{{GHPU}})$ is a Polish space.
\end{proposition}
Just like we did for the Gromov-Hausdorff topology, it is possible to extend the GHPU topology for non-compact spaces by truncation, but we refer to \cite{scalingUIHPQ} for more details.

\section{The Brownian sphere rooted along its geodesic}\label{section 3}

In this section, we give a new representation of the Brownian sphere when rooted along its simple geodesic from $\rho$. This representation relies on a spinal decomposition, in the same spirit as in \cite{bessel}. From now on, we fix $b>0$, that will remain unchanged though this paper.

\subsection{Structure of the subtrees along a trajectory}

On the event $\{W_*<-b)$ and for every $s\geq0$, we set : 
\begin{equation*}    \hat{\zeta}_s=\zeta_{(T_b+s)\wedge\sigma},\quad\check{\zeta}_s=\zeta_{(T_b-s)\vee 0}. 
\end{equation*}
Then we let $(\hat{a}_i,\hat{b}_i)_{j\in J}$ be the excursions of $\hat{\zeta}$ above its past minimum, that is the connected components of :
\begin{equation*}
    \bigg\{s\geq 0:\hat{\zeta}_s>\min_{0\leq r\leq s}\hat{\zeta}_r\bigg\}.
\end{equation*}
Similarly, we let $(\check{a}_j,\check{b}_j)_{i\in I}$ be the excursions of $\check{\zeta}$ above its past minimum. Each of those intervals represents a subtree $\T_i$ of $\T$ branching off the trajectory $W_{T_b}$. Then, for every $j\in J$, we define $W^i\in\mathcal{C}(\R_+,\mathcal{W})$ by :
\begin{equation*}
    W^{j}_s(t)=W_{T_b+(\hat{a}_j+s)\wedge\hat{b}_j}(\hat{\zeta}_{\hat{a}_j}+t),\quad 0\leq t\leq \hat{\zeta}_{(\hat{a}_j+s)\wedge\hat{b}_j}-\hat{\zeta}_{\hat{a}_j},\quad s\geq0.
\end{equation*}
Similarly, for $i\in I$, we set : 
\begin{equation*}
    W^{j}_s(t)=W_{T_b-(\check{a}_i+s)\wedge\check{b}_i}(\check{\zeta}_{\check{a}_i}+t),\quad 0\leq t\leq \check{\zeta}_{(\check{a}_i+s)\wedge\check{b}_i}-\check{\zeta}_{\check{a}_i},\quad s\geq0.
\end{equation*}
Note that we browse the trajectory $W_{T_b}$ backward. Finally, we introduce point measures on $\R_+\times\mathcal{C}(\R_+,\mathcal{W})$ defined by :
\begin{equation}\label{Poisson point}
    \widehat{\mathcal{N}}=\sum_{j\in J}\delta_{(\hat{\zeta}_{\hat{a}_j},W^j)},\quad \check{\mathcal{N}}=\sum_{i\in I}\delta_{(\check{\zeta}_{\check{a}_i},W^{I})}.
\end{equation}
\begin{theorem} \label{poisson}
    Under $\N_0(\cdot\,|\,W_*<-b)$ and conditionally on the trajectory $W_{T_b}$, the point measures $\widehat{\mathcal{N}}$ and $\check{\mathcal{N}}$ are independent Poisson point measures with respective intensities : 
     \begin{equation*}
        2\mathbf{1}_{[0,\zeta_{T_b}]}(t)dt\N_{W_{T_b}(t)}(d\omega)
    \end{equation*}
    and
     \begin{equation*}
        2\mathbf{1}_{[0,\zeta_{T_b}]}(t)\mathbf{1}_{\{\omega_*>-b\}}dt\N_{W_{T_b}(t)}(d\omega).
    \end{equation*}  
\end{theorem}
\begin{proof}
First, observe that the conditional distribution of $\widehat{\mathcal{N}}$ is just a consequence of Lemma \ref{excursion} and of the strong Markov property applied at time $T_b$. Therefore, the difficult part of the statement is the description of the law of $\check{\mathcal{N}}$. Recall that
\begin{equation*}
    \widehat{T}_b=\sup\{s\geq 0:\widehat{W}_s=-b\}.
\end{equation*}
Note that $\{W_*<-b\}=\{T_b<\infty\}=\{\widehat{T}_b<\infty\}$.
Then, we can introduce two point measures associated to the trajectory $W_{\widehat{T}_b}$
\begin{equation*}
    \widehat{\mathcal{N}}'=\sum_{i\in I'}\delta_{(\hat{\zeta}'_{\hat{a}'_i},W^i)},\quad \check{\mathcal{N}}'=\sum_{j\in Jl}\delta_{(\check{\zeta}'_{\check{a}'_j},W^{j})},
\end{equation*}
like we did for $W_{T_b}$ ($\hat{a}'_i,\hat{\zeta}',$ etc… are defined just like $\hat{a}_i,\hat{\zeta},$etc… just by replacing $W_{T_b}$ by $W_{\widehat{T}_b}$). By time-reversal invariance of the measure $\N_0$, the trajectories $W_{T_b}$ and $W_{\widehat{T}_b}$ have the same distribution, as for the measures $\check{\mathcal{N}}$ and $\widehat{\mathcal{N}}'$ . In what follows, we will work with the measure $\widehat{\mathcal{N}}'$, which is more convenient to use the Markov property. \\
For every $n\in N$, we set :
\begin{equation*}
    [\widehat{T}_b]_n=\frac{\lceil n\widehat{T}_b\rceil}{n}.
\end{equation*}
    Clearly, we have $\lim_{n\rightarrow\infty}[\widehat{T}_b]_n=\widehat{T}_b$ a.s. We also write $\widehat{\mathcal{N}}'_{(n)}$ to describe the point measure of the subtrees branching off the trajectory $W_{[\widehat{T}_b]_n}$, and $\widehat{\mathcal{N}}'_t$ to refer to the same quantity with respect to $W_t$.  
Fix $0<\delta<b$ ($\delta$ will become close to $b$), and for every $w\in\mathcal{W}$, set $\tau_\delta(w)=\inf\{t\geq 0:w(t)=-\delta\}$. Then, we define :
\begin{equation*}
    \mathcal{\hat{N}}'_{[\delta]}=\displaystyle\sum_{\substack{i\in I' \\ \zeta_{\widehat{T}_b+\hat{a}'_i}\leq\tau_\delta(W_{\widehat{T}_b}) \phantom{-}}} \delta_{(\zeta_{\widehat{T}_b+\hat{a}'_i},W^i)}.
\end{equation*}
The measure $\widehat{\mathcal{N}}'_{[\delta]}$ represents the subtrees branching off the trajectory $W_{\widehat{T}_b}$ before it hits $-\delta$; such a truncation is needed to avoid the pathological part near the endpoint of $W_{\widehat{T}_b}$. We also let $W_{\leq t}$ be the process $(W_{s\wedge t})_{s\geq 0}$. Consider a bounded, measurable and positive function $g$ on the space of point measures on  $\R_+\times\mathcal{C}(\R_+,\mathcal{W})$, and $f$ a bounded continuous function on $\mathcal{C}(\R_+,\mathcal{W})$. Then, we have :
\begin{align*}
    \N_0(f(W_{\leq \widehat{T}_b})\mathbf{1}_{\{\widehat{T}_b<\infty\}}g(\widehat{\mathcal{N}}'_{[\delta]}))=& \lim_{n\rightarrow \infty } \N_0(f(W_{\leq[\widehat{T}_b]_n})\mathbf{1}_{\{\widehat{T}_b<\infty\}}g(\widehat{\mathcal{N}}_{(n),{[\delta]}}')) \\
    =& \lim_{n\rightarrow\infty}\sum_{k\in\N}\N_0(f(W_{\leq k/n})\mathbf{1}_{\{\widehat{T}_b<\infty\}}\\
    & \quad\quad\quad\quad\quad\mathbf{1}_{\{\inf_{t\in[(k-1)/n,k/n]}\hat{W}_t<-b,\inf_{s\geq k/n} \hat{W}_s>-b\} }g(\mathcal{\widehat{\mathcal{N}}}_{k/n,{[\delta]}}')).
\end{align*}
The first equality holds because a.s, under $N_0(\cdot\,|\,W_*<-b)$, for $n$ large enough, the trajectories $W_{\widehat{T}_b}$ and $W_{[\widehat{T}_b]_n}$ coincide up to a time which is greater than or equal to $\tau_\delta(W_{\widehat{T}_b})$, and the point measures $\widehat{\mathcal{N}}'_{[\delta]}$ and $\widehat{\mathcal{N}}'_{(n),{[\delta]}}$ also coincide. Therefore, for $\varepsilon>0,\,t>0$ and $0<\delta<b$, we have to evaluate : 
\begin{equation*}\label{zoubiii}
    \N_0(f(W_{\leq t})\mathbf{1}_{\{\widehat{T}_b<\infty\}}\mathbf{1}_{\{\inf_{s\in[t-\varepsilon,t]}\hat{W_{s}}<-b\}}\mathbf{1}_{\{\inf_{s\geq t} \hat{W}_s>-b\}}g(\mathcal{\hat{N}}_{t,{[\delta]}}')).
\end{equation*}
By the Markov property applied at time $t$, and by Lemma \ref{excursion}, this quantity is equal to:
\begin{equation*} 
\N_0\bigg(f(W_{\leq t})\mathbf{1}_{\{\widehat{T}_b<\infty\}}\mathbf{1}_{\{\inf_{t-\varepsilon\leq s\leq t}\hat{W}_{s}<-b\}}\Pi_{W_t}\bigg[\1\{\mathcal{M}(\{(t,\omega):\omega_*\leq-b\})=0\}g(\mathcal{M}_{\leq\tau_\delta(W_t)})\bigg]\bigg)   
\end{equation*}
where for every $w\in\mathcal{W}_0$, under the probability measure $\Pi_w$, $\mathcal{M}(dt,d\omega)$ is a Poisson point measure on $\R_+\times C(\R_+,\mathcal{W})$ with intensity :
\begin{equation*}
 2\mathbf{1}_{[0,\zeta_{(w)}]}(t)dt\N_{w(t)}(d\omega)   
\end{equation*}
and $\mathcal{M}_{\leq\tau_\delta(W_t)}$ stands for the restriction of $\mathcal{M}$ to $[0,\tau_\delta(W_t)]\times C(\R_+,\mathcal{W})$. Using the conditional independence of $\mathcal{M}_{\leq \tau_\delta(W_t)}$ and $\mathcal{M}|_{[\tau_\delta(W_t),\zeta_t]\times\s}$ (given $W_t$), we have :
\begin{align*}
  &\Pi_{W_t}\big[\1\{\mathcal{M}(\{(t,\omega):\omega_*\leq-b\})=0\}g(\mathcal{M}_{\leq\tau_\delta(W_t)})\big]\\  
  &=  \Pi_{W_t}\big[\mathcal{M}(\{(t,\omega):\omega_*\leq-b\})=0\big]\Pi_{W_t}\big[(g(\mathcal{M}_{\leq\tau_\delta(W_t)})\,|\,\mathcal{M}(\{(t,\omega):\omega_*\leq-b\})=0)\big]\\
  &=\Pi_{W_t}\big[\mathcal{M}(\{(t,\omega):\omega_*\leq-b\})=0)\big]\Pi_{W_t^{(\delta)}}\big[g(\mathcal{M})\,|\,\mathcal{M}(\{(t,\omega):
  \omega_*\leq-b\})=0)\big]
\end{align*}
where $W_t^{(\delta)}$ stands for the restriction of $W_t$ to $[0,\tau_\delta(W_t)]$. Consequently, if for every $w\in\mathcal{W}_0$ such that $\tau_\delta(w)<\infty$, we set $H(w)=\Tilde{H}((w(t))_{0\leq t\leq \tau_\delta(w)})$, where
\begin{equation*}
    \Tilde{H}(w)=\Pi_{w}\big[g(\mathcal{M})\,|\,\mathcal{M}(\{(t,\omega):\omega_*\leq-b\})=0\big],
\end{equation*}
we proved that the quantity \eqref{zoubiii} is equal to : 
\begin{equation*}
    \N_0\bigg(f(W_{\leq t})\mathbf{1}_{\{\widehat{T}_b<\infty\}}\mathbf{1}_{\{\inf_{s\in[t-\varepsilon]}\hat{W}_{s}<-b\}}H(W_t)\Pi_{W_t}\big[\mathcal{M}(\{(t,\omega):\omega_*\leq-b\})=0\big]\bigg).
\end{equation*}
By applying the Markov property one more time at time $t$, this quantity is also equal to :
\begin{equation*}
   \N_0\bigg(f(W_{\leq t})\mathbf{1}_{\{\widehat{T}_b<\infty\}}\mathbf{1}_{\{\inf_{s\in[t-\varepsilon,t]}\hat{W}_{s}<-b\}}H(W_t)\mathbf{1}_{\{\inf_{s\geq t} \hat{W}_s>-b\}}\bigg). 
\end{equation*}
Coming back to our previous computations, we obtain that :
\begin{align*}
  &\N_0(f(W_{\leq \widehat{T}_b})\mathbf{1}_{\{\widehat{T}_b<\infty\}}g(\widehat{\mathcal{N}}'_{[\delta]}))\\
    &= \lim_{n\rightarrow\infty}\sum_{k\in\N}\N_0(f(W_{\leq k/n})\mathbf{1}_{\{\widehat{T}_b<\infty\}}\mathbf{1}_{\{\inf_{s\in[(k-1)/n,l/n]}\hat{W}_{s}<-b\}}H(W_{k/n})\mathbf{1}_{\{\inf_{s\geq k/n} \hat{W}_s>-b\}})\\ 
    &=\lim_{n\rightarrow\infty}\N_0(f(W_{\leq [\widehat{T}_b]_n})\mathbf{1}_{\{\widehat{T}_b<\infty\}}H(W_{{[\widehat{T}_b]}_n}))\\
    &=\N_0(f(W_{\leq \widehat{T}_b})\mathbf{1}_{\{\widehat{T}_b<\infty\}}H(W_{\widehat{T}_b})).
\end{align*}
(the last convergence holds because $H$ only depends on the trajectory $w$ up to time $\tau_\delta(w)$, and for $n$ large enough, the considered trajectories coincide up to this time). This means that the conditional distribution of $\widehat{\mathcal{N}}'_{[\delta]}$ given $\{\widehat{T}_b<\infty\}$ and $W_{\leq \widehat{T}_b}$ is the law of a Poisson point measure with intensity 
\begin{equation*}
   2\mathbf{1}_{[0,\tau_\delta(W_{\widehat{T}_b})]}(t)\mathbf{1}_{\{\omega_*>-b\}}dt\N_{W_{\widehat{T}_b}(t)}(d\omega).
\end{equation*}
Because $0<\delta<b$ is arbitrary, it follows that the conditional distribution of $\widehat{\mathcal{N}}'$ given $\{\widehat{T}_b<\infty\}$ and $W_{\leq \widehat{T}_b}$ is a Poisson point measure with intensity :
\begin{equation*}
   2\mathbf{1}_{[0,\zeta_{\widehat{T}_b}]}(t)\mathbf{1}_{\{\omega_*>-b\}}dt\N_{W_{\widehat{T}_b}(t)}(d\omega).
\end{equation*}
Moreover, note that the conditional distribution only depends on $W_{\widehat{T}_b}$, which completes the description of the law of $\check{\mathcal{N}}$. Finally, note that $\check{\mathcal{N}}$ is a measurable function of $W_{\leq T_b}$ and that, by the strong Markov property, $\widehat{\mathcal{N}}$ is independent of $W_{\leq T_b}$ given $W_{T_b}$. This means that, given $W_{T_b}$, the two measures are independent, which concludes the proof. 
\end{proof} 

\subsection{Spine representations of the Brownian sphere}\label{construction}

In order to construct the Brownian sphere rooted along its geodesic, we start by recalling a result of \cite{bessel}, which completes the description of the tree $\T$ (under $\N_0(\cdot\,|\,W_*<-b)$) given by Theorem \ref{poisson}.
  \begin{proposition}\label{Bessel}
        Under the probability measure $\N_0(\cdot\,|\,W_*<-b)$, the law of the random path $W_{T_b}$ is the law of the process $(X_t^{(-3)}-b)_{0\leq t\leq S^{(-3)}}$, where $X^{(-3)}$ is a Bessel process of dimension $-3$ starting at $b$, and $S^{(-3)}=\inf\{t>0,X_t^{(-3)}=0\}$.
    \end{proposition}
We also recall that by William's time reversal theorem, the law of the process $(X^{(-3)}_{S^{(-3)}-t})_{0\leq t\leq S^{(-3)}}$ is the law of a Bessel process of dimension $7$ starting at $0$ and stopped at it last hitting time of $b$ (which we denote by $S_b$). We also recall the scaling property of Bessel processes: if $(X_t)_{t\geq0}$ is a Bessel process (of any dimension), then for every $\varepsilon>0$, the processes $(X_t)_{t\geq 0}$ and $(\frac{1}{\varepsilon}X_{\varepsilon^2t})_{t\geq 0}$ have the same law.
\newline
Now, we present our construction. Consider $(R_t)_{t\geq0}$ a Bessel process of dimension $7$, and let $S_b$ be last hitting time of $b$ by $R$ (notice that the scaling property of Bessel processes implies that $S_b$ and $b^2S_1$ have the same law). Then, given $R$, consider two independent Poisson point measures $\mathcal{N}$ and $\mathcal{\widehat{N}}$ of respective intensities
\begin{equation*}
        2\mathbf{1}_{[0,S_b]}(t)\mathbf{1}_{\{\omega_*>0\}}dt\,\N_{R_t}(d\omega)
    \end{equation*}
    and
    \begin{equation*}
        2\mathbf{1}_{[0,S_b]}(t)dt\,\N_{R_t}(d\omega).
    \end{equation*}
    We write $I$ and $J$ for sets indexing the atoms of these Poisson point measures. Note that we can associate to every $i\in I\cup J$ a random labeled tree $\T_i$.\\
    We introduce a random compact metric space $\T_b$ obtained from the spine decomposition 
    
    \begin{equation*}
        [0,S_b]\cup\bigcup_{i\in I\cup J}\T_i
    \end{equation*}
    by identifying, for every $i\in I\cup J$, the root $\rho_i$ of the tree $\T_i$ with the point $t_i$ of $[0,S_b]$. The resulting space is a compact random tree (the fact that it is compact is not trivial, but is left as an exercise for the reader). More precisely, the distance $d$ is defined as follows. First, the restriction of $d$ to a subtree $\T_i$ is just the distance $d_i$. If $u,v\in[0,S_b]$, we set $d(u,v)=|u-v|$. If $u\in\T_i$ and $v\in[0,S_b]$, we take $d(u,v)=d(u,\rho_i)+d(\rho_i,v)$. Finally, if $u\in\T_i$ and $v\in\T_j$ with $i\neq j$, we let $d(u,v)=d(u,\rho_i)+d(\rho_i,\rho_j)+d(\rho_j,v)$. Note that $\T_b$ comes with a volume measure, which is just the sum of the volume measures on the trees $\T_i,i\in I\cup J$. 
    \newline
    We can then define an exploration $\mathcal{E}$ of the tree $\T_b$, given by concatenating the explorations of the trees $\T_i$. For $s\in [0,2S_b]$, set 
    \begin{equation*}
        \beta_s=\sum_{i\in I}\1_{\{t_i\leq s\}}\sigma(\omega_i)+\sum_{j\in J}\1_{\{2S_b-t_j\leq s\}}\sigma(\omega_j),\quad\beta_{s-}=\sum_{i\in I}\1_{\{t_i< s\}}\sigma(\omega_i)+\sum_{j\in J}\1_{\{2S_b-t_j<s\}}\sigma(\omega_j).
    \end{equation*}
    Then, if we set $\sigma=\sum_{i\in I\cup J}\sigma(\omega_i)$,for every $t\in[0,\sigma]$ we define $\mathcal{E}_t\in\T_b$ as follows. Observe that there is a unique $s\in[0,2S_b]$ such that $\beta_{s-}\leq t\leq \beta_s$. If $s\in[0,S_b]$,
    \begin{itemize}[label=\textbullet]
        \item either we have $s=t_i$ for some $i\in I$ and we set $\mathcal{E}_t=p_i(t-\beta_{t_i-})$,
        \item or there is no such $i$ and we set $\mathcal{E}_t=s$.
    \end{itemize}
    If, on the other hand, $s\in]S_b,2S_b]$, then  
     \begin{itemize}[label=\textbullet]
        \item Either we have $2S_b-s=t_j$ for some $j\in J$ and we set $\mathcal{E}_t=p_j(\beta_{t_j}-t)$
        \item Or there is no such $j$ and we set $\mathcal{E}_t=2S_b-s$.
    \end{itemize}
    The exploration process allows us to define intervals on the tree $\T_b$, as we did in section \ref{intervalle}. First, we make the convention that if $s>t$, the interval $[s,t]$ is defined by $[s,\sigma]\cup[0,t]$. Then, for every $u\neq v\in\T_b $, there exists a smallest interval $[s,t]$ with $s,t\in[0,\sigma]$ such that $\mathcal{E}_s=u$ and $\mathcal{E}_t=v$, and we define 
    \[[u,v]=\{\mathcal{E}_r,r\in[s,t]\}.\]
    Note that we usually have $[u,v]\neq[v,u]$.\\
    \begin{figure}
        \centering
        \includegraphics[scale=0.5]{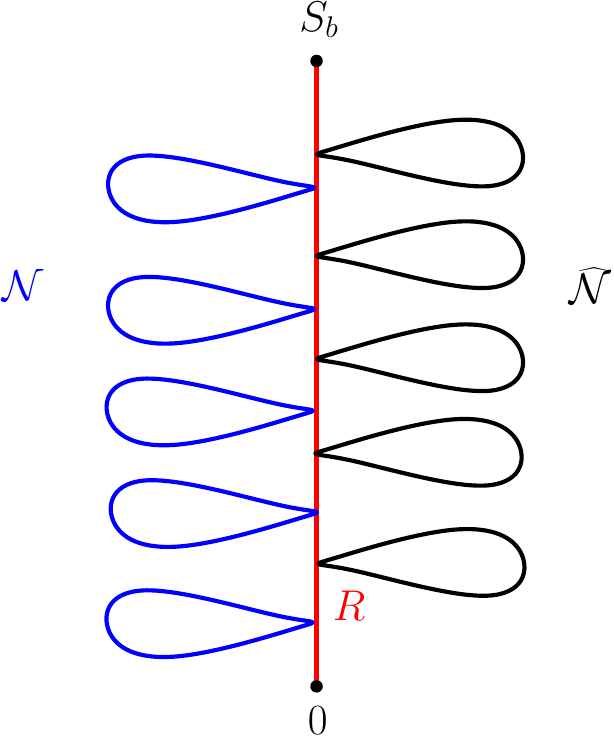}
        \caption{Illustration of the labelled tree $\T_b$. Each bubble represents a subtree $\T_i$, the blue ones being conditioned to carry positive labels. Labels on the red segment evolve as a Bessel process of dimension $7$.}
    \end{figure}
    Next, we assign labels to the points of $\T_b$. For $t\in[0,S_b]$, we set $Z_t=R_t$ and for $u\in\T_i$, we take $Z_u=\widehat{W}_u(\omega_i)$. The following proposition is just a consequence of Theorem \ref{poisson} and Proposition \ref{Bessel}.
    \begin{proposition}\label{tree}
        We have the following equality in distribution 
        \[\left(\T_b,(Z_u-b)_{u\in \T_b},S_b\right)\overset{(d)}{=}\left(\T,(Z_u)_{u\in\T},\rho_\T\right)\]
        where the right hand-member is distributed under $\N_0(\cdot\,|\,W_*<-b)$.
    \end{proposition} 
    \begin{remark}
        The previous proposition does not really make sense at this point, since we have not introduced any topology on the set of labelled trees. However, its meaning should be clear, and can be seen as a reformulation of Theorem \ref{poisson} and Proposition \ref{Bessel}.
    \end{remark}
    Note that the element $p_\T(T_b)$ of $\T$ is naturally identified with $0\in\T_b$. 
    The labels allow us to define a pseudo-distance on $\T_b$ as follows. For $u,v\in\T_b$, we set:
    \begin{equation*}
        D^\circ(u,v)=Z_u+Z_v-2\max\left(\inf_{w\in[u,v]}Z_w,\inf_{w\in[v,u]}Z_w\right)
    \end{equation*}
    and then 
    \begin{equation*}
        D(u,v)=\inf\bigg\{\sum_{i=1}^pD^\circ(u_{i-1},u_i)\bigg\}
    \end{equation*}
    where the infimum is taken over all the choices of integer $p\geq1$ and of the finite sequences $u_0,...,u_p$ such that $u_0=u$ and $u_p=v$. We write $\s_b=\T/\{D=0\}$ for the quotient space, which naturally comes with the distance induced by $D$, and is rooted at the projection of $0\in\T$, denoted by $\rho_b$. The following proposition is just a consequence of Proposition \ref{tree}, and due to the fact that shifting labels do not change the distance.
    
    \begin{proposition}
       The pointed metric space $(\s_b,D,\rho_b)$ has law of the Brownian sphere under $\N_0(\cdot\,|\,W_*<-b)$ rooted at $\Gamma_b$.  
    \end{proposition}
    
\subsection{The three-arms decomposition}\label{three arms}

In this section, we give another construction of the labelled tree $\T$ under $\N_0(\cdot\,|\,W_*<-b)$, which enjoys some independence properties.\\

Consider the Poisson point measures defined in \eqref{Poisson point}, and let $j_b\in J$ be the unique index of $J$ such that 
   \[(W^{j_b})_*<-b\quad\text{ and }\quad \text{for every $j\in J$ such that $t_i<t_{j_b}$, $(W^i)_*>-b$.} \]
   In words, $W^{j_b}$ encodes the labelled subtree containing $p_\T(\widehat{T}_b)$. 
    Note that the element $u\in\T_b$ corresponding to $t_{j_b}$ is the last common ancestor of $p_\T(T_b)$ and $p_\T(\widehat{T}_b)$ in $\T$. Consequently, the trajectories $W_{T_b}$ and $W_{\widehat{T}_b}$ coincide up to time $t_{j_b}$. In what follows, we will write $H=t_{j_b}$.
    We can then introduce three trajectories $Y,X,\widehat{X}$ defined in the following way :
    \begin{itemize}
        \item $Y=W_{T_b}|_{[0,H]}$
        \item $X=W_{T_b}(H+\cdot)$
        \item $\widehat{X}=W_{\widehat{T}_b}(H+\cdot)$.
    \end{itemize}
    These trajectories also come with point measures $\mathcal{N}_0,\widehat{\mathcal{N}_0},\mathcal{N}_1,\mathcal{N}_2,\widehat{\mathcal{N}}_1,\widehat{\mathcal{N}}_2$ (as shown in Figure \ref{3 Arms}). Note that these trajectories and their associated point measures fully describe the tree $\T$. 
     \begin{figure}
        \centering
        \includegraphics[scale=0.4]{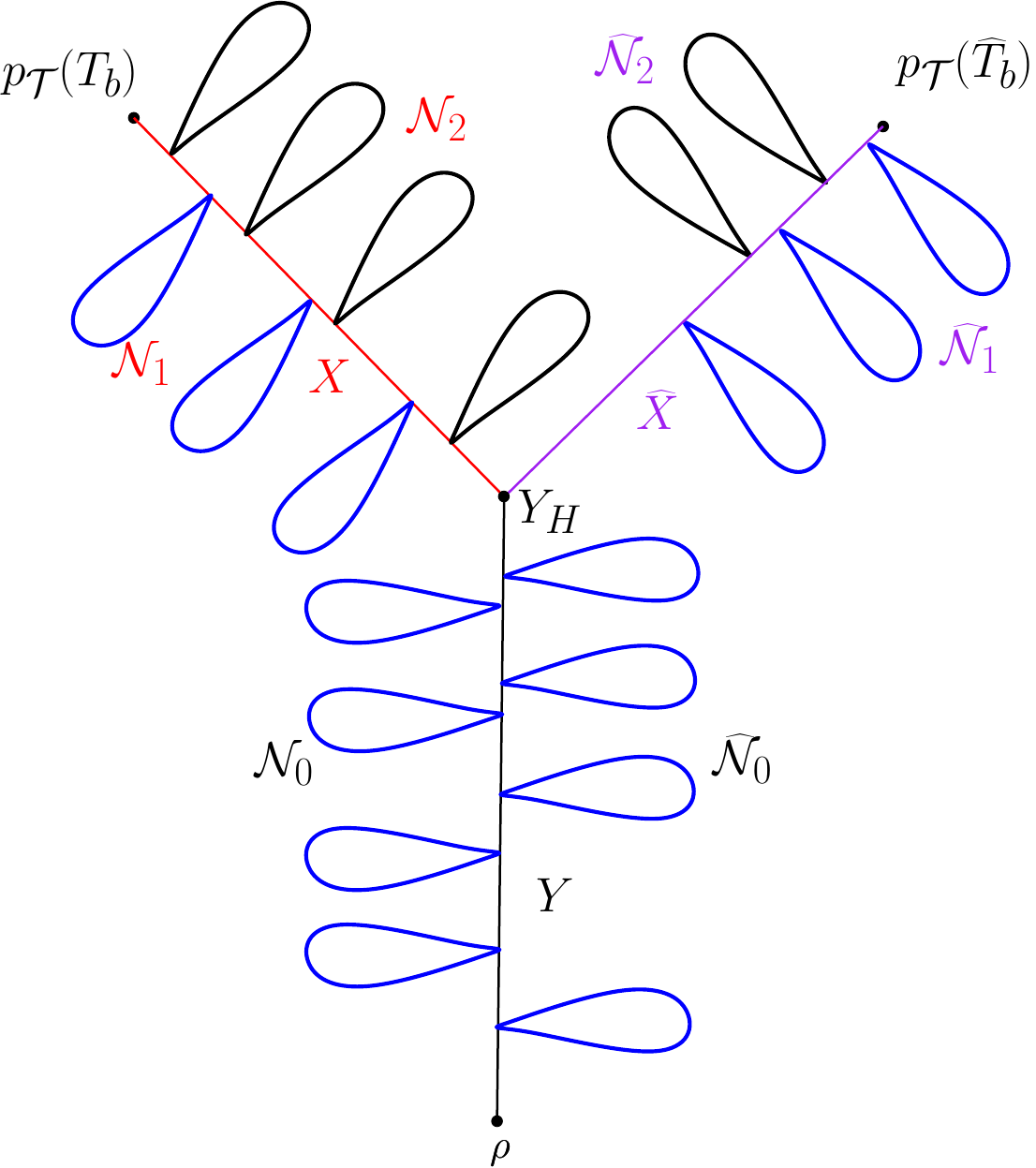}
        \caption{Decomposition of the labelled-tree $\T$. Each bubble represents a subtree; the blue ones are conditioned to have labels strictly above $-b$.}
        \label{3 Arms}
    \end{figure}
    The following proposition identifies the distribution of $H$. 
    \begin{proposition}
        Given $W_{T_b}$, the random variable $H$ satisfies 
        \[\p(H>t\,|\,W_{T_b})=\exp\left(-2\int_0^t\N_{W_{T_b}(s)}(W_*<-b)ds\right)\quad\text{ for every $t\in[0,\zeta_{W_{T_b}}]$}.\]
    \end{proposition}
    \begin{proof}
        The formula follows from the construction given \ref{construction} ; we have 
        \begin{align*}
        \p(H>t\,|\,W_{T_b})&=\Pi_{W_{T_b}}\left(\mathcal{M}(\{(s,\omega):s<t,\,\omega_*<-b\})=0\right)\\
        &=\exp\left(-2\int_0^t\N_{W_{T_b}(s)}(W_*<-b)ds\right).
        \end{align*}      
    \end{proof}
    \begin{proposition}\label{decomposition}
        Given $Y$, the trajectories $X+b$ and $\widehat{X}+b$ are independent Bessel processes of dimension $-3$ starting from $Y_H+b$ and stopped when they reach $0$. Furthermore, conditionally on $(Y,X,\widehat{X})$, the associated point measures are independent Poisson point processes, and their intensity are the following :
        \begin{itemize}[label=\textbullet]
            \item $\mathcal{N}_0,\widehat{\mathcal{N}_0}$ have the same intensity 
            \[2\1_{[0,H]}(t)\1_{\{W_*>-b\}}dt\,\N_{Y(t)}(dW)\]
            \item $\mathcal{N}_1,\mathcal{N}_2$ have respective intensities
            \[2\1_{[0,\zeta(X)]}(t)\1_{\{W_*>-b\}}dt\,\N_{X(t)}(dW)\]
            and 
            \[2\1_{[0,\zeta(X)]}(t)dt\,\N_{X(t)}(dW)\]
            \item  $\widehat{\mathcal{N}_1},\widehat{\mathcal{N}_2}$ have respective intensities 
            \[2\1_{[0,\zeta(\widehat{X})]}(t)\1_{\{W_*>-b\}}dt\,\N_{\widehat{X}(t)}(dW)\]
            and
            \[2\1_{[0,\zeta(\widehat{X})]}(t)dt\,\N_{\widehat{X}(t)}(dW).\]
        \end{itemize}
    \end{proposition}
    \begin{proof}
        We start with the derivation of the law of the trajectories. The proof relies on Palm's formula, that we recall here : if $\mathcal{M}$ is a Poisson point measure of intensity $m$, then for every measurable function $F\geq 0$, 
    \begin{equation*}
        \E\bigg[\int_E F(x,\mathcal{M})\mathcal{M}(dx)\bigg]=\int_E \E[F(x,\mathcal{M}+\delta_x)]m(dx).
    \end{equation*}    
    Recall the notation $\Pi_w$ introduced in the proof of Theorem \ref{poisson}. In what follows, $(R_t+b)_{0\leq t\leq T_0}$ stands for a Bessel process of dimension $-3$ starting at $b$ and stopped when it reaches $0$. Under $\N_0(\cdot\,|\,W_*<-b)$, we have, for $f,g,h$ non-negative measurable functions, 
        \begin{align*}
            &\E[f(Y)g(X)h(W^{j_b})]\\
            &=\E\bigg[f(Y)g(X)\Pi_R\bigg(\sum_{i\in I}h(W^i)\1_{t_i=t_{j_b}}\bigg)\bigg]\\
            &=\E\bigg[2\int_0^{T_0}dt f(R|_{[0,t]})g(R|_{[t,T_0]})\int_\s h(W)\1_{\{W_*<-b\}}\Pi_R\bigg(\mathcal{M}(\{(s,\omega):s<t,\omega_*<-b\})=0\bigg)\N_{R_t}(dW)\bigg]
            \end{align*}
            where we used Theorem \ref{poisson}. Using properties of Poisson point measures and the Markov property for the Bessel process $X$, we obtain 
            \begin{align*}
             &\E[f(Y)g(X)h(W^{j_b})]\\
            &=\E\bigg[2\int_0^{T_0}dt f(R|_{[0,t]})g(R|_{[t,T_0]})\int_\s h(W)\1_{\{W_*<-b\}}\exp\bigg(-2\int_0^t\N_{R_s}(W_*<-b)ds\bigg)\N_{R_t}(dW)\bigg]\\
            &=2\int_0^\infty dt\E\bigg[\1_{\{T_0>t\}}f(R|_{[0,t]})g(R|_{[t,T_0]})\int_\s h(W)\1_{\{W_*<-b\}}\exp\bigg(-2\int_0^t\N_{R_s}(W_*<-b)ds\bigg)\N_{R_t}(dW)\bigg]\\
            &=2\int_0^\infty dt\E\bigg[\1_{\{T_0>t\}}f(R|_{[0,t]})\E_{R_t}[g(R'|_{[0,T'_0]})]\exp\bigg(-2\int_0^t\N_{R_s}(W_*<-b)ds\bigg)\int_\s h(W)\1_{\{W_*<-b\}}\N_{R_t}(dW)\bigg]\\
            &=\E\big[f(R|_{[0,H]})\E_{R_H}[g(R'|_{[0,T'_0]})]\N_{R_H}(h(W)\,|\,W_*<-b)\big]
        \end{align*}
        where under $\E_x[\cdot]$, $R'+b$ has the law of a Bessel process of dimension $-3$ starting from $x$ and stopped when it reaches $0$.\\
        Therefore, given $Y$, $W^{j_b}$ is independent of $X$, and is distributed according to the measure $\N_{Y_H}(\cdot\,|\,W_*<-b)$ (note that the distribution of $W^{j_b}$ only depends on $Y_H$). Now, notice that the trajectory $\widehat{X}$ corresponds to the path $W_{\widehat{T}_b}$ in $W^{j_b}$. Therefore, given $Y$, the conditional distribution of $\widehat{X}$ is a consequence of Proposition \ref{Bessel}. Similarly, the distributions of $\widehat{\mathcal{N}_1}$ and $\widehat{\mathcal{N}}_2$ are given by Theorem \ref{poisson}. The law of $X$ follows from the Markov property applied at time $t$ in the previous calculation. \\
        For a Poisson point measure $\mathcal{N}$ on $[0,\zeta]\times\s$ and $0<t<\zeta$, we set 
        \[\mathcal{N}_{\leq t}=\mathcal{N}|_{[0,t]\times\s}\quad\text{ and }\mathcal{N}_{\geq t}=\mathcal{N}|_{[t,\zeta]\times\s}\]
        (note that $\mathcal{N}_{\leq t}$ and $\mathcal{N}_{\geq t}$ are independent). Another application of the Palm formula shows that, for $f,g,h,h_1,h_2$, non-negative measurable functions, we have 
          \begin{align*}
            &\E[f(Y)g(X)h(W^{j_b})h_1(\widehat{\mathcal{N}}_0)h_2(\mathcal{N}_2))]\\
            &=\E\bigg[2\int_0^{T_0}dtf(R|_{[0,t]})\E_{R_t}[g(R'|_{[0,T'_0]})]\int_\s h(W)\1_{\{W_*<-b\}}\Phi(t,R)\N_{R_t}(d W)\bigg]
        \end{align*}
        where 
        \begin{equation*}
            \Phi(t,R)=\Pi_R\big[h_1(\mathcal{M}_{\leq t})h_2(\mathcal{M}_{\geq t})\1\{\mathcal{M}(\{(s,\omega):s<t,\omega_*<-b\})=0)\}\big].
        \end{equation*}
        Using independence properties of Poisson point measures, we obtain 
        \begin{equation*}
            \Phi(t,R)=\exp\left(-2\int_0^t\N_{R_s}(W_*<-b)ds\right)\Pi_Y\big[h_1(\mathcal{M})\,\big|\,\mathcal{M}(\{(s,\omega):\omega_*<-b\})=0\big]\Pi_{X}\big[h_2(\mathcal{M})\big].
        \end{equation*}
        Hence, we have 
        \begin{align*}
        &\E[f(Y)g(X)h(W^{j_b})h_1(\widehat{\mathcal{N}}_0)h_2(\mathcal{N}_2))]\\
         &=\E\bigg[f(Y)g(X)\N_{Y_H}(h(W)\,|\,W_*<-b)\Pi_Y\big[h_1(\mathcal{M})\,\big|\,\mathcal{M}(\{(s,\omega):s<t,w_*<-b\})=0\big]\Pi_{X}\big[h_2(\mathcal{M}))\big]\bigg].
        \end{align*}
        Therefore, $\widehat{\mathcal{N}}_0$ and $\mathcal{N}_2$ have the distributions described by Proposition \ref{decomposition} (the fact that they are independent of the other point measures is a consequence of the independence between $Y,\,X$ and $\widehat{X}$).\\
        Finally, notice that given $X$, $\mathcal{N}$ is independent of $H$ (because the Poisson point measures of Theorem \ref{poisson} are independent). Hence, $(\mathcal{N}_0,\mathcal{N}_1)$ has the law of $(\mathcal{N}_{\leq H},\mathcal{N}_{\geq H})$. Furthermore, by properties of Poisson point measures, $\mathcal{N}_0$ (respectively $\mathcal{N}_1$) only depends on $Y$ (respectively $X$). \\
        Finally, the independence of the trajectories and the independence coming from Theorem \ref{poisson} shows that given $(Y,X,\widehat{X})$, the point measures $\mathcal{N}_0,\widehat{\mathcal{N}}_0,\mathcal{N}_1,\mathcal{N}_2,\widehat{\mathcal{N}}_1,\widehat{\mathcal{N}}_2$ are independent, which concludes the proof. 
    \end{proof}
   The previous proposition allows us to give a new representation of the Brownian sphere under the probability measure $\N_0(\cdot\,|\,W_*<-b)$. In this section, we abuse notations by using the same letters as in Proposition \ref{decomposition} to designate the processes and the Poisson point measures.\\
    Consider $(Y_t)_{t\in[0,T_0]}$ a Bessel process of dimension $-3$ starting from $b$ and stopped when it reaches $0$, and given $Y$, consider a random variable $H$ on $[0,T_0]$, with distribution 
    \begin{equation}\label{temps branchement}
        \p(H>t)=\exp\left(-2\int_0^t\N_{Y_s}(W_*<0)ds\right)\quad\text{ for every $t\in[0,T_0]$}.
    \end{equation}
    Then, consider $X$ and $\widehat{X}$ two Bessel processes of dimension $7$  starting from $0$, independent of $Y$ and stopped at their last hitting time of $Y_H$. Finally, given ($Y,X,\widehat{X}$), consider $\mathcal{N}_0,\widehat{\mathcal{N}}_0,\mathcal{N}_1,\mathcal{N}_2,\widehat{\mathcal{N}}_1,\widehat{\mathcal{N}}_2$ independent Poisson point measures with following intensities:
        \begin{itemize}[label=\textbullet]
            \item $\mathcal{N}_0,\widehat{\mathcal{N}_0}$ have the same intensity 
            \[2\1_{[0,H]}(t)\1_{\{W_*>0\}}dt\,\N_{Y(t)}(dW)\]
            \item $\mathcal{N}_1,\mathcal{N}_2$ have respective intensity
            \[2\1_{[0,S_{Y_H}]}(t)\1_{\{W_*>0\}}dt\,\N_{X(t)}(dW)\]
            and 
            \[2\1_{[0,S_{Y_H}]}(t)dt\,\N_{X(t)}(dW)\]
            where $S_{Y_H}$ is the last hitting time of $Y_H$ by $X$. 
            \item  $\widehat{\mathcal{N}}_1,\widehat{\mathcal{N}_2}$ have respective intensity 
            \[2\1_{[0,\widehat{S}_{Y_H}]}(t)\1_{\{W_*>0\}}dt\,\N_{\widehat{X}(t)}(dW)\]
            and
            \[2\1_{[0,\widehat{S}_{Y_H}]}(t)dt\,\N_{\widehat{X}(t)}(dW)\]
            where $\widehat{S}_{Y_H}$ is the last hitting time of $Y_H$ by $\widehat{X}$.
        \end{itemize}    
    We write $I_0,J_0,I_1,I_2,J_1,J_2$ for sets indexing the atoms of these Poisson point measures. Note that every atom $i$ encodes a random labelled tree $\T_i$. \\
    Each tuple $(Y,\mathcal{N}_0,\widehat{\mathcal{N}}_0)$, $(X,\mathcal{N}_1,\mathcal{N}_2)$ and $(\widehat{X},\widehat{\mathcal{N}}_1,\widehat{\mathcal{N}}_2)$ encodes a random labelled tree in the following way (we give the construction for $(Y,\mathcal{N}_0,\widehat{\mathcal{N}}_0)$).\\
    We introduce a random compact metric space $\T_{(0)}$ obtained from the spine decomposition 
    
    \begin{equation*}
        [0,H]\cup\bigcup_{i\in I_0\cup J_0}\T_i
    \end{equation*}
    by identifying, for every $i\in I_0\cup J_0$, the root $\rho_i$ of the tree $\T_i$ with the point $t_i$ of $[0,L]$, which gives us a compact random tree. Then, the distance $d$ is defined as follows. First, the restriction of $d$ to a subtree $\T_i$ is just the distance $d_i$. If $u,v\in[0,L]$, we set $d(u,v)=|u-v|$. If $u\in\T_i$ and $v\in[0,L]$, we take $d(u,v)=d(u,\rho_i)+d(\rho_i,v)$. Finally, if $u\in\T_i$ and $v\in\T_j$ with $i\neq j$, $d(u,v)=d(u,\rho_i)+d(\rho_i,\rho_j)+d(\rho_j,v)$. Note that $\T$ comes with a volume measure, which is just the sum of the volume measures on the trees $\T_i,i\in I_0\cup J_0$. 
    We can also define an exploration of the tree $\T_{(0)}$, as we did previously in subsection \ref{construction}. \\
   
    Similarly, $(X,\mathcal{N}_1,\mathcal{N}_2)$ and $(\widehat{X},\widehat{\mathcal{N}}_1,\widehat{\mathcal{N}_2})$ encode two random trees $\T_{(1)}$ and $\T_{(2)}$ (which have the same law). Then, we can define a new random tree $\T_b$ (we abuse notations, by keeping the one introduced in the previous section) obtained by gluing $\T_{(0)}$, $\T_{(1)}$ and $\T_{(2)}$. More precisely, set 
    \[\T_b=\T_{(0)}\cup\T_{(1)}\cup\T_{(2)}\]
    the space obtained by identifying the points $H$ of $\T_{(0)}$, $S_{Y_H}$ of $\T_{(1)}$ and $\widehat{S}_{Y_H}$ of $\T_{(2)}$, in such a way that the exploration process $\mathcal{E}$ of $\T_b$ proceeds as described in figure \ref{3 Arms}. The distance on $\T_b$ is defined as follows. 
    First, the restriction of $d$ to $\T_{(0)}$ (respectively $\T_{(1)}$, $\T_{(2)}$) is $d_0$ (respectively $d_1$, $d_2$). If $u\in\T_{(0)}$, and $v\in\T_{(1)}$, we set $d(u,v)=d_0(u,H)+d_1(S_{Y_H},v)$. We proceed in the same way for the other cases. \\
   
    Next, we assign labels to the points of $\T_b$. If $t\in[0,L]$ (respectively $t\in[0,T_0]$, $t\in[0,T'_0]$), we set $Z_t=Y(t)$ (respectively $Z_t=X(t)$, $Z_t=\widehat{X}(t)$). 
    Otherwise, if $u\in\T_i$, we take $Z_u=\hat{W}_u(\omega_i)$. \\
    For every $s\in[W_*,b]$, we set 
    \begin{equation*}
        \eta_s=\inf\{t\geq0,Z_{\mathcal{E}_t}=s\}\quad\text{and}\quad\widehat{\eta}_s=\sup\{t\geq0,Z_{\mathcal{E}_t}=s\}.
    \end{equation*}
    Note that $\eta_s=T_{b-s}$ and $\widehat{\eta}_s=\widehat{T}_{b-s}$. Therefore, we will abuse notations by using $p_\T(T_s)$ and $p_\T(\widehat{T}_s)$ to refer to $\mathcal{E}_{\eta_{b-s}}$ and $\mathcal{E}_{\widehat{\eta}_{b-s}}$. In particular, $p_\T(T_b)$ and $p_\T(\widehat{T}_b)$ stand for the roots of $\T_{(1)}$ and $\T_{(2)}$. Therefore, with these notations, we have $Z_{p_\T(T_s)}=b-s$.\\
    Those labels allow us to define a pseudo-distance on $\T_b$ as follows. For $u,v\in\T_b$, we set:
    \begin{equation*}
        D^\circ(u,v)=Z_u+Z_v-2\max\left( \inf_{w\in[u,v]}Z_w,\inf_{w\in[v,u]}Z_w\right)
    \end{equation*}
    and then 
    \begin{equation*}
        D(u,v)=\inf\bigg\{\sum_{i=1}^pD^\circ(u_{i-1},u_i)\bigg\}
    \end{equation*}
    where the infimum is taken over all the choices of integer $p\geq1$ and of the finite sequences $u_0,...,u_p$ such that $u_0=u$ and $u_p=v$. We write $\s_b=\T_b/\{D=0\}$ for the quotient space, which naturally comes with the distance induced by $D$, and is rooted at the projection of $0\in\T_b$, denoted by $\rho$. We abuse notations, by denoting by $p_\s$ the canonical projection onto $\s_b$. The following proposition is just a consequence of Proposition \ref{decomposition}, and Williams time reversal theorem.
    
\begin{proposition}\label{construction 3 arms}
      We have
      \[\left(\T_b,(Z_u-b)_{u\in \T_b},\rho\right)\overset{(d)}{=}\left(\T,(Z_u)_{u\in\T},\rho_\T\right)\]
      and 
      \[\left(\s_b,D,\rho\right)\overset{(d)}{=}\left(\s,D,\rho\right)\]
     where the right-hand members are distributed under $\N_0(\cdot\,|\,W_*<-b)$.
\end{proposition}
\begin{remark}
    Note that the exploration process of $\T_b$ (with labels shifted by $-b$) allows us to define a snake trajectory associated to $\T_b$, and this snake trajectory has the law of $\N_0(\cdot\,|\,W_*<-b)$. Therefore, we will often consider $\left(\T_b,(Z_u-b)_{u\in \T_b}\right)$ and $\left(\T,(Z_u)_{u\in \T}\right)$ (under $\N_0(\cdot\,|\,W_*<-b)$) as being the same object, and view the processes $(Y,X,\widehat{X},\mathcal{N}_0,\widehat{\mathcal{N}}_0,\mathcal{N}_1,\mathcal{N}_2,\widehat{\mathcal{N}}_1,\widehat{\mathcal{N}}_2)$ as being part of $\T$. 
\end{remark}
    
\subsection{The bigeodesic Brownian plane}\label{presentation}

In this section, we define the limiting object that will appear in our convergence results. The following objects are the infinite-volume equivalent of those introduced in section \ref{construction}. 

For this purpose, we let $(X^\infty_t)_{t\geq0}$ be a Bessel process of dimension $7$  and, given $X^\infty$, we consider two independent Poisson point measures $\mathcal{N}^\infty_1$ and $\mathcal{N}^\infty_2$ of respective intensities
    \begin{equation*}
       2\mathbf{1}_{\{\omega_*>0\}}dt\N_{X^\infty_t}(d\omega)
    \end{equation*}
    and
    \begin{equation*}
      2dt\N_{X^\infty_t}(d\omega).
    \end{equation*}
   We write $I_1^\infty$ (respectively $I_2^\infty$) for sets indexing the atoms of the measure $\mathcal{N}_1^\infty$ (respectively $\mathcal{N}_2^\infty$). Each of these atoms $i\in I_1^\infty\cup I_2^\infty$ encodes a tree $\T_i$, which naturally comes with labels. Then, we construct our random metric space from the spine decomposition 
    \begin{equation*}
        [0,+\infty)\cup\bigcup_{i\in I_1^\infty\cup I_2^\infty}\mathcal{T}_i
    \end{equation*}
    by identifying, for every $i\in I_1^\infty\cup I_2^\infty$, the root $\rho_i$ of the tree $\T_i$ to the point $t_i$ of $[0,+\infty)$, which give us a non-compact random tree $\T_\infty$, that we root at 0. Note that according to the representation of Section \ref{construction}, $\T_\infty$ is the infinite volume version of the tree $\T$ under $\N_0(\cdot\,|\,W_*<-b)$ rooted at $T_b$. The distance $d_\infty$ on $\T_\infty$ is defined just as the distance $d$ on $\T_b$ in section \ref{construction}. We also have two exploration functions $\mathcal{E}^-$ and $\mathcal{E}^+$ defined as in section \ref{construction} (that explore respectively the sides coded by $\mathcal{N}^\infty_1$ and $\mathcal{N}_2^\infty$).\\
    These exploration functions allow us to define intervals on $\T_\infty$, as we did in \ref{construction} with $\T_b$. Noting that $\mathcal{E}^-_0=\mathcal{E}^+_0=0$, we define $(\mathcal{E}^\infty_s)_{s\in\R}$ by 
    \begin{equation*}
        \mathcal{E}^\infty_s=
        \begin{cases}
            \mathcal{E}_s^+ &\text{if }s\geq 0,\\
            \mathcal{E}_{-s}^- &\text{if }s\leq 0.
        \end{cases}
    \end{equation*}
    We make the convention that, for $s>t$, the  ``interval'' $[s,t]$ is defined by $[s,t]=[s,+\infty)\cup(-\infty,t]$. Now, for every $u,v\in\T_\infty$, there exists a unique smallest interval $[s,t]$ such that $\mathcal{E}^\infty_s=u$ and $\mathcal{E}^\infty_t=v$, and we define 
    \begin{equation*}
        [u,v]=\{\mathcal{E}^\infty_r,\,r\in[s,t]\}.
    \end{equation*}
    \newline
    The tree $\T_\infty$ naturally comes with labels from the process $X^\infty$ and the labeled trees $\T_i,i\in I_1^\infty\cup I_2^\infty$. More precisely, for $t\in[0,\infty)$, we set $Z^\infty_t=X^\infty_t$ and for $u\in\T_i$, we set $Z^\infty_u=\hat{\omega}_i(u)$. These labels allow us to define a function $D^\circ_\infty$ on $\T_\infty\times\T_\infty$ with the same formulas as the ones used in Section \ref{construction}:
    \begin{equation*}
        D^\circ_\infty(u,v)=Z^\infty_u+Z^\infty_v-2\max\left(\inf_{w\in[u,v]}Z^\infty_w,\inf_{w\in[v,u]}Z^\infty_w\right).
    \end{equation*}
    Note that this pseudo-distance satisfies the bound \eqref{Bound}.
    \newline
    We also introduce hitting times in the tree $\T_\infty$. For every $s\in\R$, we set 
    \begin{equation*}
    T^\infty_s=\inf\{t\in\R,\;Z^\infty_{\mathcal{E}^\infty_t}=s\}.       
    \end{equation*}   
    
    \begin{lemme}
       Almost surely, for every $s\in\R$, it holds that $T^\infty_s<\+\infty$. 
    \end{lemme}
    \begin{proof}
      We first assume that $s\leq 0$. By properties of Poisson point process, we have 
        \begin{align*}
           \p(T_s=\infty)= \E\left[\Pi_{X^\infty}\left(\mathcal{M}(\{(t,\omega),\,\omega_*<s\})=0\right)\right]&=\E\left[\exp\left(-2\int_0^\infty\N_{X^\infty_t}(W_*<s)dt\right)\right]\\
            &=\E\left[\exp\left(-3\int_0^\infty\frac{1}{(X^\infty_t-s)^2}dt\right)\right]\\
            &= 0
        \end{align*}
        because the integral diverges almost surely (this is a consequence of the scaling property of Bessel processes). \newline
        To prove the result for $s>0$ , it is enough to show that for every $A>0$, we have
        \begin{equation*}
            \lim_{R\rightarrow\infty}\p\bigg(\inf_{i\in I_1^\infty,t_i>R}\omega_{i_*}<A\bigg)=0. 
        \end{equation*} 
        We proceed as in the proof of Lemma $3.3$ of \cite{Hullprocess2016}. We have, for every $R>0$,
        \begin{align*}
          &\p\left(\inf_{i\in I_1^\infty,t_i>R}\omega_{i_*}<A\right)\\  
          &=\p\left(\inf_{t>R}X^\infty_t<A\right)+\E\left[\1\{\inf_{t>R}X^\infty_t>A\}\left(1-\exp\left(-2\int_R^\infty \N_{X^\infty_t}(0\leq W_*\leq A)dt\right)\right)\right]\\
          &=\p\left(\inf_{t>R}X^\infty_t<A\right)+\E\left[\1\{\inf_{t>R}X^\infty_t>A\}\left(1-\exp\left(-3\int_R^\infty (\frac{1}{(X^\infty_t-A)^2}-\frac{1}{(X^\infty_t)^2})dt\right)\right)\right].
        \end{align*}
        The result follows from the convergence of the last integral and from the transience of the Bessel processes of dimension 7. 
    \end{proof}
    For every $s\in \R$, set
    \begin{equation*}
        \tau_s=
            \mathcal{E}^\infty_{T^\infty_s},
    \end{equation*}
    so that (roughly speaking), the set $(\tau_s)_{s\in \R}$ is the counterpart of the set $(p_\T(T_{b+s}))_{s\in\R}$ in the finite tree $\T_b$.\\
    Consider two independent and identically distributed processes $(X^\infty,\mathcal{N}_1^\infty,\mathcal{N}_2^\infty)$ and $(\widehat{X}^\infty,\widehat{\mathcal{N}}_1^\infty,\widehat{\mathcal{N}}_2^\infty$ with the law described at the beginning of Section \ref{presentation}. As explained before, these processes encode two independent labelled trees $\T_\infty$ and $\widehat{\T}_\infty$. We will construct a new space by gluing together $\T_\infty$ and $\widehat{\T}_\infty$ along the sets $(\tau_s)_{s\in \R}$ and $(\widehat{\tau}_s)_{s\in \R}$ (where the notation $\widehat{\tau}$ refers to the space $\widehat{\T}_\infty$) as follows. For every $u,v\in\T_\infty\cup\widehat{\T}_\infty$, we define:     
    \begin{equation*}
        \overline{D}_\infty^\circ(u,v)=
        \begin{cases}
            D^\circ_\infty(u,v)\text{ (or }\widehat{D}^\circ_\infty(u,v))&\quad \text{if }u,v\in \T_\infty\text{ (or }u,v\in\widehat{\T}_\infty)\\    \inf_{s\in\R}\big(D^\circ_\infty(u,\tau_s)+\widehat{D}^\circ_\infty(\widehat{\tau}_s,v)\big)&\quad\text{if }u\in\T_\infty,\,v\in\widehat{\T}_\infty.
        \end{cases}
    \end{equation*}
    Note that, for every $s\in\R$, $\overline{D}^\circ_\infty(\tau_s,\widehat{\tau}_s)=0$. Then, set 
    \begin{equation*}
        \overline{D}_\infty(u,v)=\inf\bigg\{\sum_{i=1}^p\overline{D}_\infty^\circ(u_{i-1},u_i)\bigg\}   
    \end{equation*}
    where the infimum is taken over all the choices of the integer $p\geq1$ and of the finite sequences $u_0,...,u_p\in\T_\infty\cup\widehat{\T}_\infty$ such that $u_0=u$ and $u_p=v$.
 
    \begin{definition}
    The \textit{bigeodesic Brownian plane} is the quotient space $\BP=(\T_\infty\cup\widehat{\T}_\infty)/\{\overline{D}_\infty=0\}$, equipped with the distance $\overline{D}_\infty$ and with the distinguished point $\overline{\rho}_\infty$ which is the equivalence class of $0$. 
    \end{definition}
The space $\BP$ also has a natural projection $p_{\BP} :\T_\infty\cup\widehat{\T}_\infty\longrightarrow\BP$. Furthermore, it comes with a natural volume measure $\overline{\mu}$, which is the pushforward of the volume measure on the trees $\T_\infty\cup\widehat{\T}_\infty$ under the projection $p_{\BP}$. Note that the bound \eqref{Bound} remains valid for this space. \\
Let $x\in\BP$. A geodesic ray from $x$ is a continuous path $\gamma:\R_+\longrightarrow \BP$ such that $\gamma(0)=x$ and $\overline{D}_\infty(\gamma(s),\gamma(t))=|t-s|$ for every $s,t\geq0$. A bigeodesic from $x$ is a continuous path $\gamma:\R\longrightarrow \BP$ such that $\gamma(0)=x$ and $\overline{D}_\infty(\gamma(s),\gamma(t))=|t-s|$ for every $s,t\in\R$. Note that every bigeodesic $\gamma$ from $x$ naturally contains two geodesic rays from $x$, which are $\gamma|_{\R_+}$ and $\gamma(-\,\cdot)|_{\R_+}$.  \\
Note that there are at least two distinct geodesic rays starting from the root $\overline{\rho}_\infty$, which are 
\[\Gamma^+_\infty(s)=p_{\overline{\mathcal{BP}}}(\tau_s),\quad \text{for $s\geq 0$}\]
and
\[\Gamma^-_\infty(s)=p_{\overline{\mathcal{BP}}}(\tau_{-s}),\quad \text{for $s\geq 0$}.\]
Furthermore, these two geodesics form a bigeodesic, which is 
\[\Gamma_\infty(s)=p_{\BP}(\tau_s)\quad\text{ for every }s\in\R\] (these results are just consequences of the bounds \eqref{Bound} and $\overline{D}_\infty\leq\overline{D}^\circ_\infty$). Consequently, by Proposition 15 of \cite{brownianplane}, the random space $\BP$ is not the Brownian plane, more precisely, the two spaces have mutually singular laws.  
    The following result is a consequence of Theorem \ref{convergence}, but could also be obtained directly by using scaling properties of the labelled tree $\T_\infty$. 
    \begin{proposition}
    $\BP$ is scale invariant, meaning that for every $\lambda>0$, the metric space $\big(\BP,\lambda\cdot \overline{D}_\infty,\overline{\rho}_\infty\big)$ has the same distribution than $\big(\BP,\overline{D}_\infty,\overline{\rho}_\infty\big)$.
    \end{proposition}
    
\section{The localization lemma}\label{Section4}

In this section, we prove that the local behavior in the Brownian sphere (around $\Gamma_b$) only depends on local information in the tree $\T_b$, in a quantitative way, meaning that points that are close to $\Gamma_b$ in the Brownian sphere must also be concentrated in small regions of the  tree. This result will be used several times later to justify that these small regions encode all the necessary information to perform a local convergence. This kind of results were already known when $\Gamma_b$ is replaced by $\Gamma_0$ and in the discrete case (see \cite{brownianplane} for instance), but their proofs rely on the so-called cactus-bound, a refinement of \eqref{Bound}, which cannot be applied in our setting.\\

Recall that $I_1^\infty$ (respectively $I_2^\infty,\,J_1^\infty,\,J_2^\infty$) is a set indexing the atoms of the Poisson point measure $\mathcal{N}_1^\infty$ (respectively $\mathcal{N}_2^\infty,\,\widehat{\mathcal{N}}_1^\infty,\,\widehat{\mathcal{N}}_2^\infty$). For every $L>0$, we define the following subset of $\T_\infty$:
\begin{equation*}
    A_{L}^\infty:=[0,S^\infty_L]\cup\bigcup_{i\in I_1^\infty\cup I_2^\infty,t_i<S^\infty_L}\T_i
\end{equation*}
where $S^\infty_L$ is the last hitting time of $L$ by the Bessel process $X^\infty$ used to build $\T_\infty$.
Similarly, for $\widehat{\T}_\infty$, we set
\begin{equation*}
    \widehat{A}_{L}^\infty:=[0,\widehat{S}^\infty_L]\cup\bigcup_{i\in J_1^\infty\cup J_2^\infty,t_i<\widehat{S}^\infty_L}{\T}_i.
\end{equation*}
These sets correspond to the ancestral lines of the trees $\T_\infty$ and $\widehat{\T}_\infty$ up to the last point of their spines with label $L$, together with the subtrees branching off the spine before these points. \\
We can define similar sets for the tree $\T_b$. For every $L>0$, set 
\begin{equation*}
    A_{L}:=\llbracket p_\T(T_b),S_L\rrbracket\cup\bigcup_{i\in I_1\cup I_2,t_i<S_L}\T_i
\end{equation*}
and 
\begin{equation*}
    \widehat{A}_{L}:=\llbracket p_\T(\widehat{T}_b),\widehat{S}_L\rrbracket\cup\bigcup_{i\in J_1\cup J_2,t_i<\widehat{S}_L}{\T}_i.
\end{equation*}
\begin{figure}
    \centering
    \includegraphics[scale=0.5]{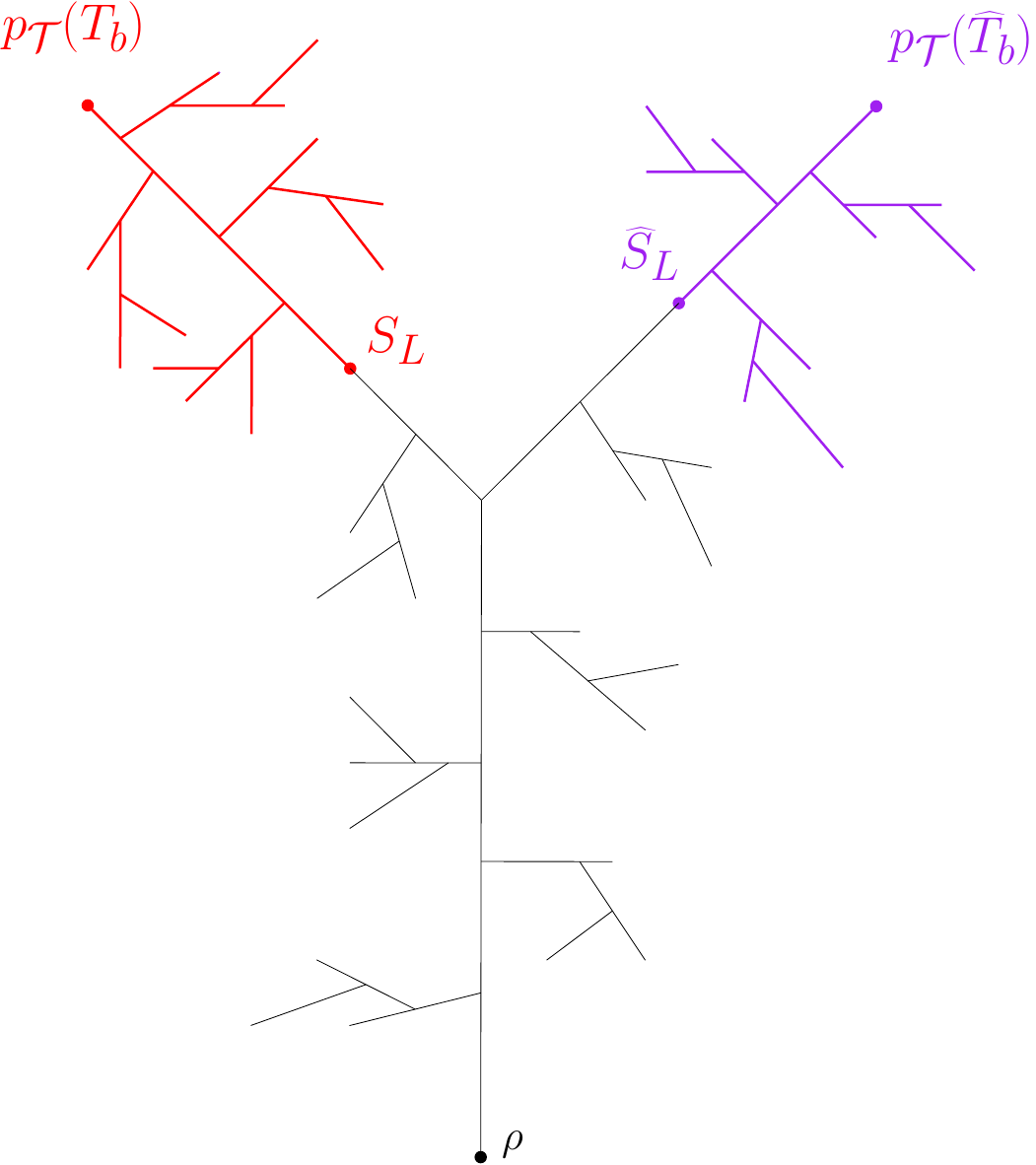}
    \caption{The sets $A_L$ and $\widehat{A}_L$ are respectively represented in red and purple.}
\end{figure}
These sets are well-defined when $Y_H>L$, and also correspond to a portion of the ancestral lines from $p_\T(T_b)$ and $p_\T(\widehat{T}_b)$ to the root $\rho_\T$, together with the subtrees branching off those paths. On the event $\{Y_H<L\}$, we can still give a meaning to these definitions, for instance using the representation of Subsection \ref{construction}, but the two sets are no longer disjoint. These sets also come with labels, which (after a shift of $b$) are given by Bessel processes of dimension 7 for the spine, and by independent Poisson point measures for the subtrees. As explained at the end of Subsection \ref{three arms}, these sets can be seen as being part of $\T$. Note that these spaces can be defined using snake trajectories as follows (we present the construction for $A_{L}$, but by symmetry, it also works for $\widehat{A}_L$ after obvious adaptation). Set
\[a^{(L)}=\inf\{t\geq0,\mathcal{E}_t=S_{L}\}\quad \text{and}\quad b^{(L)}=\sup\{t\geq0,\mathcal{E}_t=S_{L}\}.\]
We can define a new snake trajectory $W^{(L)}=\left(W^{(L)}_t\right)_{\leq t \leq b^{(L)}- a^{(L)}}$ as
\begin{equation}\label{snake traj}   
W^{(L)}_t(s)=W_{a^{(L)}+t}(\zeta_{a^{(L)}}+s),\quad 0\leq t \leq b^{(L)}- a^{(L)}, 0\leq s \leq \zeta_{a^{(L)}+s} - \zeta_{a^{(L)}}. 
\end{equation}

Then, $A_L$ is the labelled tree associated to the snake trajectory $W^{(L)}$. As explained at the end of Subsection \ref{three arms}, we can also recover this snake trajectory from $A_L$. Therefore, these two points of view are equivalent.
This section is devoted to the proof of the following proposition.
\begin{proposition}\label{localisation}
 For every $\delta>0$, there exists $L>0$ and $\varepsilon_0>0$ such that for every $0<\varepsilon<\varepsilon_0$ :
\begin{itemize}[label=\textbullet]
    \item \centering$\p(B_\varepsilon(\s,\Gamma_b)\subset p_\s(A_{L\varepsilon}\cup\widehat{A}_{L\varepsilon}))>1-\delta$
    \item \centering $\p(B_\varepsilon(\BP ,\overline{\rho}_\infty)\subset p_{\BP }(A^\infty_{L\varepsilon}\cup\widehat{A}^\infty_{L\varepsilon}))>1-\delta$
\end{itemize}   
\end{proposition}

We will only prove the first statement, but every step of the proof can be adjusted to prove the second statement (which is in fact a little bit easier to prove). Note that the scaling property of the space $\BP$ implies that the second statement holds for every $\varepsilon>0$.
   The proof will be separated in three parts. First, we will give some lemmas about the sets $A_{L\varepsilon}$ and $\widehat{A}_{L\varepsilon}$. Then, we will see how to interpret these subsets as metric spaces (that are called Brownian slices), and use these spaces to prove Proposition \ref{localisation}.
   \begin{remark}
      In Proposition \ref{localisation}, we could have replaced $A_{L\varepsilon}$ and $\widehat{A}_{L\varepsilon}$ by $B_{M\varepsilon^2}(\T,p_\T(T_b))$ and $B_{M\varepsilon^2}(\T,p_\T(\widehat{T}_b))$ (for some constant $M>0$), but this formulation is less suited for our work. Furthermore, note that these sets can be seen as hulls centered at $p_\T(T_b)$ and $p_\T(\widehat{T}_b)$ relative to $\rho_\T$.
   \end{remark}
   
   \subsection{First estimates}
   
We begin with some results about the behavior of these sets that will be used later.

\begin{proposition}\label{Encadrement}
    For every $\delta>0$, there exists $L>0$ such that for every $\varepsilon$ small enough,
    \[\N_0(Y_H>L\varepsilon,\,\inf_{u\in A_{L\varepsilon}}Z_u<-b-\varepsilon\,|\,W_*<-b)>1-\delta\]
    and
    \[\N_0(Y_H>L\varepsilon,\,\inf_{u\in A_{L\varepsilon}}Z_u>-b-\sqrt{\varepsilon}\,|\,W_*<-b)>1-\delta.\]
\end{proposition}
\begin{proof}
    Using the representation $\ref{three arms}$, we have :
    \begin{align*}
        &\N_0(Y_H>L\varepsilon,\inf_{u\in A_{L\varepsilon}}Z_u<-b-\varepsilon\,|\,W_*<-b)\\
        &=\N_0(Y_H>L\varepsilon\,|\,W_*<-b)-\E\left[\1_{\{Y_H>L\varepsilon\}}\Pi_{X-b}(\mathcal{M}(\{(t,\omega):t<S_{L\varepsilon},\,\omega_*<-b-\varepsilon\})=0)\right]\\
        &=\N_0(Y_H>L\varepsilon\,|\,W_*<-b)-\E\left[\1_{\{Y_H>L\varepsilon\}}\exp\left(-2\int_0^{S_{L\varepsilon}}\N_{X_t-b}(W_*<-b-\varepsilon)dt\right)\right]\\
        &\geq\N_0(Y_H>L\varepsilon\,|\,W_*<-b)-\E\left[\exp\left(-3\int_0^{S_{L\varepsilon}}\frac{1}{(X_t+\varepsilon)^2}dt\right)\right].
    \end{align*}
    Using the scaling properties on Bessel processes, we obtain 
    \[\N_0(Y_H>L\varepsilon,\,\inf_{u\in A_{L\varepsilon}}Z_u<-b-\varepsilon\,|\,W_*<-b)\geq\N_0(Y_H>L\varepsilon\,|\,W_*<-b)-\E\left[\exp\left(-3\int_0^{S_{L}}\frac{1}{(X_t+1)^2}dt\right)\right].\]
   Since the integral diverges to $+\infty$ almost surely as $L\rightarrow\infty$, we can choose $L>0$ large enough so that 
    \[\E\left[\exp\left(-3\int_0^{S_{L}}\frac{1}{(X_t+1)^2}dt\right)\right]<\delta/2.\]
    Then, by \eqref{temps branchement}, for $\varepsilon>0$ small enough, we have
    \[\N_0(Y_H>L\varepsilon,\,|\,W_*<-b)\geq1-\delta/2\]
    which gives the first inequality. Similarly,
    \begin{align*}
        &\N_0(Y_H>L\varepsilon,\,\inf_{u\in A_{L\varepsilon}}Z_u>-b-\sqrt{\varepsilon}\,|\,W_*<-b)\\
        &=\E\left[\1_{\{Y_H>L\varepsilon\}}\Pi_{X-b}(\mathcal{M}(\{(t,\omega):t<S_{L\varepsilon},\,\omega_*<-b-\sqrt{\varepsilon}\})=0)\right]\\       
        &=\E\left[\1_{\{Y_H>L\varepsilon\}}\exp\left(-3\int_0^{S_{L\varepsilon}}\frac{1}{(X_t+\sqrt{\varepsilon})^2}dt\right)\right]\\
        &\geq\p(Y_H>L\varepsilon)\E\left[\exp\left(-\frac{3S_L}{\sqrt{\varepsilon}}\right)\right]\xrightarrow[\varepsilon\rightarrow0]{}1
    \end{align*}
    using the independence between $Y_H$ and $X$, and the scaling property of Bessel processes.
\end{proof}

Then, we need to check that for $L$ large enough, the set $p_\T(T_{b-\varepsilon})$ is included in $A_{L\varepsilon}$ with an arbitrary large probability (and similarly for $p_\T(\widehat{T}_{b-\varepsilon})$ and $\widehat{A}_{L\varepsilon}$). Recall that $S_x$ is the last hitting time of $x$ by $X$. We start with an intermediate result. 

\begin{lemme}\label{inc}

For every $\delta>0$, there exist $L>1$ and $\varepsilon_1>0$ such that for every $\varepsilon<\varepsilon_1$ :
    \begin{equation*}
        \N_0(Y_H>L\varepsilon,\,t_{i_\varepsilon}<S_{L\varepsilon}\,|\,W_*<-b)>1-\delta
    \end{equation*}
    where $i_\varepsilon$ is the unique index $i\in I_1$ such that $p_\T(T_{b-\varepsilon})\in\mathcal{T}_j$ (we set $t_{i_\varepsilon}=\infty$ if no such index exists).
\end{lemme}
\begin{proof}
 Let $\Tilde{X}$ be the concatenation of $X$ and of the process $Y$ reversed in time. By Proposition \ref{tree}, $\Tilde{X}$ is a Bessel process of dimension $7$ starting at $0$ and stopped at its last hitting time of $b$. Using the representation of Section \ref{three arms}, \eqref{inf} and properties of Poisson point measure, we have 
    \begin{align*}        
   \N_0(Y_H>L\varepsilon,\,t_{i_\varepsilon}<S_{L\varepsilon}\,|\,W_*<-b)&=\E\bigg[\1_{\{Y_H>L\varepsilon\}}\Pi_{\Tilde{X}-b}(\mathcal{M}(\{(t,\omega):t>S_{L\varepsilon},\,-b<\omega_*<-b+\varepsilon\}=0)\bigg]\\
    &=\E\bigg[\1_{\{Y_H>L\varepsilon\}}\exp\bigg(-2\int_{S_{L\varepsilon}}^{S_b}\N_{\Tilde{X}_t-b}(-b<W_*<-b+\varepsilon)dt\bigg)\bigg]\\
    &=\E\bigg[\1_{\{Y_H>L\varepsilon\}}\exp\bigg(-3\int_{S_{L\varepsilon}}^{S_b}\left(\frac{1}{(\Tilde{X}_t-\varepsilon)^2}-\frac{1}{\Tilde{X}_t^2}\right)dt\bigg)\bigg].
    \end{align*}
    Now, by the scaling property of Bessel processes, for every $\varepsilon>0$, $\int_{S_{L\varepsilon}}^{S_b}\left(\frac{1}{(\Tilde{X}_t-\varepsilon)^2}-\frac{1}{\Tilde{X}_t^2}\right)dt$ has the same law as $\int_{S_{L}}^{S_{b/\varepsilon}}\left(\frac{1}{(\Tilde{X}_t-1)^2}-\frac{1}{\Tilde{X}_t^2}\right)dt$, which is stochastically dominated by $\int_{S_{L}}^\infty\left(\frac{1}{(\Tilde{X}_t-1)^2}-\frac{1}{\Tilde{X}_t^2}\right)dt$. We denote this random variable by $I_L$ (note that it does not depend on $\varepsilon$). Then, we have:
    \begin{equation}\label{coincoin}
     \N_0(Y_H>L\varepsilon,\,t_{i_\varepsilon}<S_{L\varepsilon}\,|\,W_*<-b)\geq\E[\exp(-3I_L)]-\N_0(Y_H<L\varepsilon\,|\,W_*<-b)\\
    \end{equation}
    Then, because of the convergence of the integral, we can choose $L$ (that does not depend on $\varepsilon$) large enough such that 
    \begin{equation*}
\E[\exp(-3I_L)]>1-\delta/2.
    \end{equation*}
     Finally, for $\varepsilon$ small enough, we have $\N_0(Y_H<L\varepsilon\,|\,W_*<-b)<\delta/2$. For these choices and using \eqref{coincoin}, we obtain the desired result. 
\end{proof}
\begin{remark}
    Note that we always have $S_\varepsilon<t_{i_\varepsilon}$, which implies that $B_\varepsilon(\s,\Gamma_b) 
\nsubseteq p_\s(A_{\varepsilon}\cup\widehat{A}_{\varepsilon}) $. This shows that the order of magnitude in Proposition \ref{localisation} is optimal. 
\end{remark}

\subsection{Brownian slices}

Now, we introduce the notion of Brownian slice and study some of its properties. Fix a snake trajectory $\omega\in\s_0$ with duration $\sigma$. We can define a pseudo-distance $\Tilde{d}$ on $[0,\sigma]$ :
\begin{equation*}
    \Tilde{d}_{(\omega)}(s,t)=Z_s+Z_t-2\inf_{r\in [s\land t,s\lor t]}Z_r.
\end{equation*}
The difference with the distance $d$ of Section \ref{sphere} is that we forbid  ``to go around the root of $\T_\omega$''  when computing the distance. Then, just as we did to construct the Brownian sphere, we can define a pseudo-distance $\Tilde{D}_{(\omega)}^\circ$ on $\T_\omega$:
\begin{equation}\label{pseudo dist slice}
    \Tilde{D}_{(\omega)}^\circ(u,v)=\inf\{\Tilde{d}_{(\omega)}(s,t):s,t\in[0,\sigma],p_{\T_\omega}(s)=u,p_{\T_\omega}(t)=v\}.
\end{equation}
Then, we can define another pseudo-distance on $\T_\omega$ :
\begin{equation*}
    \Tilde{D}_{(\omega)}(u,v)=\inf_{u_0,...,u_p}\sum_{i=1}^p\Tilde{D}_{(\omega)}^\circ(u_i,u_{i-1})
\end{equation*}
here the infimum is taken over every $p\in\N^*$ and sequences in $\T$ such that $u_0=u$ and $u_p=v$.
\begin{definition}
    The \textit{free Brownian slice} is defined under the measure $\N_0$ as the metric space $\Tilde{\s}=\T/\{\Tilde{D}=0\}$, equipped with the distance $\Tilde{D}$. We write $p_{\Tilde{\s}}:\T\mapsto \Tilde{\s}$ for the canonical projection, $\Tilde{\mathbf{p}}$ for the projection $[0,\sigma]\mapsto\Tilde{\s}$ and $\Tilde{\rho}=\Tilde{\mathbf{p}}(0)$.
\end{definition}
This space has already been studied in \cite{uniqueness} to prove the convergence of quadrangulations toward the Brownian sphere, and in \cite{Browniandisk} to prove the convergence of quadrangulation with a boundary toward the Brownian disk (see also \cite{Geodesicstars}). It is also the scaling limit of some models of random planar maps with geodesic boundaries. Moreover, it can be seen as the Brownian sphere cut along its geodesic $\Gamma$ (see \cite[Section 3.2]{uniqueness} for more details).\\
Note that because $d\leq\Tilde{d}$, we have $D\leq\Tilde{D}$. Furthermore, $\Tilde{\s}$ has the same scaling property as the Brownian sphere. Contrary to $\s$, the space $\Tilde{\s}$ has two geodesics between $\Tilde{\rho}$ and its point of minimal label $x_*$: they are given by $(\Tilde{\mathbf{p}}(T_s))_{0\leq s\leq W_*}$ and $(\Tilde{\mathbf{p}}(\widehat{T}_s))_{0\leq s\leq W_*}$, which we denote by $\gamma^L$ and $\gamma^R$. Moreover, these curves intercept only at their endpoints. \\
The proof of \ref{localisation} is based on the following observation (recall that the snake trajectory $W^{(L)}$ was defined in \eqref{snake traj}).
\begin{lemme}\label{law slice}
    Conditionally on $\{Y_H>L\varepsilon\}$, the snake trajectory $W^{(L\varepsilon)}$ has the law of $\N_{L\varepsilon}(\cdot\,|\,W_*<0)$. 
\end{lemme}
This readily follows from the spine decomposition of Section \ref{construction}, together with Proposition \ref{tree} and the definition of $A_{L\varepsilon}$ (that is a Bessel process of dimension $7$ stopped at its last hitting time of $L\varepsilon$, together with two Poisson point measures with the right intensities). 
Therefore, we can associate to $A_{L\varepsilon}$ (respectively $\widehat{A}_{L\varepsilon}$) a Brownian slice $\Tilde{\s}^\varepsilon_1$, with two distinguished curves $\gamma_1^{L,\varepsilon}$ and $\gamma_1^{R,\varepsilon}$ (respectively a Brownian slice $\Tilde{\s}^\varepsilon_2$, with two distinguished curves $\gamma_2^{L,\varepsilon}$ and $\gamma_2^{R,\varepsilon}$). To lighten notations, we will often omit the exponent $\varepsilon$. 
Each of these slices is equipped with a distance ($\Tilde{D}_1$ and $\Tilde{D}_2$), and we have two projections $p_{\Tilde{\s}_1}:A_{L\varepsilon}\rightarrow\Tilde{\s}_1$ and $p_{\Tilde{\s}_2}:\widehat{A}_{L\varepsilon}\rightarrow\Tilde{\s}_2$.\\
Finally, we have two natural applications $p_1:\Tilde{\s}_1\mapsto\s$ and $p_2:\Tilde{\s}_2\mapsto\s$, defined by 
\[\forall x\in\Tilde{\s}_1,\quad p_1(x)=p_\s(u)\quad\text{ for any $u\in A_{L\varepsilon}$ such that $p_{\Tilde{\s}_1}(u)=x$}\] and the definition is similar for $p_2$ (one can easily check that these functions are well-defined). Note that $p_1(\gamma_1^L(L\varepsilon))=p_2(\gamma_2^R(L\varepsilon))=\Gamma_b$. For every $\varepsilon,L>0$, we introduce the event $\mathcal{G}_{L,\varepsilon}$ defined as 
\[\mathcal{G}_{L,\varepsilon}=\{Y_H>L\varepsilon\}\cap\{W_*<-b-\sqrt{\varepsilon}\}\cap\left\{-b-\sqrt{\varepsilon}<\inf_{u\in A_{L\varepsilon}\cup\widehat{A}_{L\varepsilon}}Z_u<-b-\varepsilon\right\}.\] By Proposition \ref{Encadrement}, for every $\delta>0$, there exists $L>0$ such that for every $\varepsilon>0$ small enough, $\p(\mathcal{G}_{L,\varepsilon})>1-\delta$. 
\begin{proposition}
        Under $\mathcal{G}_{L,\varepsilon}$, the canonical projections $p_1$ and $p_2$ are homeomorphisms onto their images.
    \end{proposition}
    Consequently, in what follows, we will identify $\Tilde{\s}_1$ and $\Tilde{\s}_2$ as subsets of $\s$. 
\begin{proof}
        We give the proof for $p_1$. First, note that for every $u,v\in A_{L\varepsilon}$, $D^\circ(u,v)\leq\Tilde{D}_1^\circ(u,v)$, which implies that for every $x,y\in\Tilde{\s}_1$, $D(p_1(x),p_1(y))\leq\Tilde{D}_1(x,y)$. This shows that the projection $p_1$ is continuous. \\
Let $x,y\in\Tilde{\s}_1$ such that $p_1(x)=p_1(y)$, and $u,v\in A_{L\varepsilon}$ such that $p_{\Tilde{\s}_1}(u)=x$ and $p_{\Tilde{\s}_1}(v)=y$. By Proposition \ref{Identification}, we have $D^\circ(u,v)=0$. Note that either $[u,v]\subset A_{L\varepsilon}$ or $[v,u]\subset A_{L\varepsilon}$. Suppose that $[u,v]\subset A_{L\varepsilon}$ and note that under $\mathcal{G}_{L,\varepsilon}$, $x_*\notin p_\s(A_{L\varepsilon}\cup\widehat{A}_{L\varepsilon})$. Then, we have 
\[\min_{w\in[v,u]}Z_w=Z_*<\min_{w\in[u,v]}Z_w\]
which implies that 
\[\Tilde{D}^\circ_1(u,v)=Z_u+Z_v-2\inf_{w\in[u,v]}Z_w=D^\circ(u,v)=0.\]
Hence, $x=y$ and $p_1$ is injective. Since $\Tilde{\s}_1$ is compact, $p_1$ is a homeomorphism onto its image. 
    \end{proof}
    Now, as we did to define \eqref{snake traj}, we fix $a_1,b_1>0$ (respectively $a_2,b_2$) such that $A_{L\varepsilon}=\{\mathcal{E}_t,\,t\in[a_1,b_1]\}$ (respectively $\widehat{A}_{L\varepsilon}=\{\mathcal{E}_t,\,t\in[a_2,b_2]\}$). In other words, $a_1=\inf\{t\geq0,\mathcal{E}_t=S_{L\varepsilon}\}$ and $b_1=\sup\{t\geq0,\mathcal{E}_t=S_{L\varepsilon}\}$.
For every $s\leq L\varepsilon$ and $i\in\{1,2\}$, we set 
\begin{equation*}
    T^i_s=\inf\{t\in[a_i,b_i],Z_{\mathcal{E}_t}=s\}\quad\text{and }\quad\widehat{T}^i_s=\sup\{t\in[a_i,b_i],Z_{\mathcal{E}_t}=s\}.
\end{equation*}
Note that $p_{\Tilde{\s}_i}(\mathcal{E}_{T^i_s})=\gamma^L_i(L\varepsilon-s)$ and $p_{\Tilde{\s}_i}(\mathcal{E}_{\widehat{T}^i_s})=\gamma^R_i(L\varepsilon-s)$. 

\begin{lemme}\label{Identification2}
Under $\{Y_H>L\varepsilon\}$, let $u\in A_{L\varepsilon}$ (respectively $u\in\widehat{A}_{L\varepsilon}$) be such that there exists $v\in\T\backslash A_{L\varepsilon}$ (respectively $v\in\T\backslash\widehat{A}_{L\varepsilon}$) satisfying $D(u,v)=0$. Then $p_\s(u)\in\gamma_1^R\cup\gamma_1^L$ (respectively $p_\s(u)\in\gamma_2^R\cup\gamma_2^L$).
\end{lemme}
\begin{proof}
    Set $\alpha=Z_u$. By Proposition \ref{Identification}, we have $D^\circ(u,v)=0$. But note that this implies $D^\circ(u,\mathcal{E}_{T^i_\alpha})=0$ or $D^\circ(u,\mathcal{E}_{\widehat{T}^i_\alpha})=0$ which means that $p_\s(u)\in\gamma_i^R\cup\gamma_i^L$.
\end{proof}

Our first result is to identify the boundary $\Tilde{\s}_i$, viewed as a subset of $\s$.

\begin{proposition}\label{boundary}
Under $\{Y_H>L\varepsilon\}$, for every $i\in \{1,2\}$, the topological boundary of $\Tilde{\s}_i$ in $\s$ is included in $\gamma_i^R\cup\gamma_i^L$.
\end{proposition} 
\begin{proof}
Without loss of generality, suppose that $i=1$. Consider a continuous path $\eta:[0,1]\rightarrow \s$ such that $\eta(0)\notin\Tilde{\s}_1$ and $\eta(1)\in\Tilde{\s}_1$. We set
\begin{equation*}
        \tau=\inf\{0\leq s\leq 1:\eta(s)\in\Tilde{\s}_1\}
\end{equation*}
   and we want to show that $\eta(\tau)\in\gamma_1^R\cup\gamma_1^L$. If $\eta(\tau)=p_\s(S_{L\varepsilon})$, there is nothing to add. Suppose it is not the case, and for $s\in[0,\tau]$, fix $z_s\in\T\backslash(A_{L\varepsilon}\backslash S_{L\varepsilon})$ (which is a closed set) such that $p_\s(z_s)=\eta(s)$. By compactness, we can suppose that $z_s\rightarrow z$ when $s\rightarrow\tau$, and thus $\eta(\tau)=p_\s(z)$. \\
   However, $\Tilde{\s}_1$ is a closed subset of $\s$, which implies that $\eta(\tau)\in\Tilde{\s}_1$. Thus, we can write $\eta(\tau)=p_\s(u)$, with $u\in A_{L\varepsilon}$. Therefore, by Proposition \ref{Identification}, we have $D^\circ(u,z)=0$. But by Lemma \ref{Identification2}, this implies that $p_\s(z)\in\gamma_1^R\cup\gamma_1^L$, which concludes the proof. 
\end{proof}
        We start by proving that for points of $\Tilde{\s}_i$ which are far from a boundary, the distances $D$ and $\Tilde{D}_i$ coincide.
        \begin{proposition}\label{iso locale}
    The following holds almost surely, on $\mathcal{G}_{L,\varepsilon}$. Fix $\delta>0$, $i\in \{1,2\}$, and $x,y\in\Tilde{\s}_i$ such that $D(x,\gamma_i^L)\wedge D(y,\gamma_i^L)>\delta$ (or $D(x,\gamma_i^R)\wedge D(y,\gamma_i^R)>\delta$) and $D(x,y)<\delta/2$. Then $\Tilde{D}_i(x,y)=D(x,y)$. 
\end{proposition}
\begin{proof} 
    Without loss of generality, suppose that $D(x,\gamma_1^L)\wedge D(y,\gamma_1^L)>\delta$. Fix $u,v\in A_{L\varepsilon}$ such that $x=p_\s(u)$ and $y=p_\s(v)$. First, note that on $\mathcal{G}_{L,\varepsilon}$, for any $u,v\in A_{L\varepsilon}$, it holds that $D^\circ(u,v)=\Tilde{D}_1^\circ(u,v)$. Indeed, if $[u,v]\subset A_{L\varepsilon}$, we have
    \[D^\circ(u,v)=Z_u+Z_v-2\max\left(\inf_{w\in[u,v]}Z_w,\inf_{w\in[v,u]}Z_w\right)=Z_u+Z_v-2\inf_{w\in[u,v]}Z_w=\Tilde{D}^\circ_1(u,v),
    \]
    and the case $[v,u]\subset A_{L\varepsilon}$ is treated similarly. 
    We want to show that, when computing $D(u,v)$, we can restrict ourselves to sequences of points that belong to $A_{L\varepsilon}$. Fix a sequence $u_0,...,u_p\in\T_b$ such that $u_0=u$ and $u_p=v$. Without loss of generality, we can suppose that
        \begin{equation*}
            \sum_{i=1}^pD^\circ(u_i,u_{i-1})\leq\delta.
        \end{equation*}
        We need to show that there exists another sequence $v_0,...,v_q\in A_{L\varepsilon}$ satisfying $v_0=u$ and $v_q=v$ such that
        \begin{equation}\label{2}
            \sum_{i=1}^pD^\circ(u_i,u_{i-1})\geq\sum_{i=1}^qD^\circ(v_i,v_{i-1}).
        \end{equation}
        We construct the sequence $v_0,...,v_q$ from $u_0,...,u_p$ as follows. If $u_0,...,u_p\in A_{L\varepsilon}$, there is nothing to do. Otherwise, we introduce 
        \begin{equation*}
            \tau=\inf\{j>0,u_j\notin A_{L\varepsilon}\}\quad\text{ and }\quad\eta=\inf\{j>\tau,u_j\in A_{L\varepsilon}\}
        \end{equation*}
        (these quantities are well-defined because $u_p\in A_{L\varepsilon}$). 
        Define 
        \begin{equation*}
        \alpha:=\max\bigg(\inf_{w\in[u_\tau,u_{\tau-1}]}Z_w,\inf_{w\in[u_{\tau-1},u_\tau]}Z_w\bigg)\quad\text{and}\quad(W^1)_*=\inf_{w\in A_{L\varepsilon}}Z_w.
        \end{equation*}
         
        \begin{itemize}[label = \textbullet]
              \item Suppose that $\alpha<(W^1)_* $ : it implies that
              \[D^\circ(u_{\tau-1},u_\tau)=D^\circ(u_{\tau-1},p_\T(T^1_{(W^1)_* }))+D^\circ(u_\tau,p_\T(T^1_{(W^1)_* })).\]
            Therefore, we can suppose that $p_\T(T^1_{(W^1)_* })$ is between $u_{\tau-1}$ and $u_\tau$ in the sequence. But this implies that 
            \[\sum_{i=1}^pD^\circ(u_i,u_{i-1})\geq\sum_{i=1}^{\tau-1} D^\circ(u_i,u_{i-1})+D^\circ(u_{\tau-1},p_\T(T^1_{(W^1)_*}))\geq D(x,\gamma_1^L)>\delta\] 
            which contradicts our assumption on the sequence $u_0,...,u_p$ Consequently, this case does not happen.
            \item If $\alpha=\inf_{w\in[u_{\tau-1},u_\tau]}Z_w$ and $\alpha\geq (W^1)_* $, we have
             \begin{equation*}
                D^\circ(u_\tau,u_{\tau-1})=D^\circ(u_\tau,p_\T(\widehat{T}^1_\alpha))+D^\circ(u_{\tau-1},p_\T(\widehat{T}^1_\alpha)).
            \end{equation*}
            As previously, this implies that 
            \[\sum_{i=1}^pD^\circ(u_i,u_{i-1})\geq\sum_{i=1}^\tau D^\circ(u_i,u_{i-1})+D^\circ(u_\tau,p_\T(\widehat{T}^1_\alpha))\geq D(x,\gamma_1^L)>\delta\] 
            which contradicts our assumption on the sequence $u_0,...,u_p$. Consequently, this case does not happen either.          
        \end{itemize}     
    This means that $\alpha=\inf_{w\in[u_\tau,u_{\tau-1}]}Z_w$ and $\alpha\geq (W^1)_* $. Then, we have
            \begin{equation*}
                D^\circ(u_\tau,u_{\tau-1})=D^\circ(u_\tau,p_\T(T^1_\alpha))+D^\circ(u_{\tau-1},p_\T(T^1_\alpha)).
            \end{equation*}
            Therefore, we can suppose that $p_\T(T^1_\alpha)$ is between $u_{\tau-1}$ and $u_\tau$ in the sequence.
        Similarly, if we set 
        \[\beta:=\max\bigg(\inf_{w\in[u_\eta,u_{\eta-1}]}Z_w,\inf_{w\in[u_{\eta-1},u_\eta]}Z_w\bigg),\]
        we can suppose that $p_\T(T^1_\beta)$ is between $u_{\eta-1}$ and $u_\eta$. But note that 
        \[D^\circ(u_\tau,p_\T(T^1_\alpha))+\sum_{i=\tau+1}^{\eta -1}D^\circ(u_i,u_{i-1})+D^\circ(u_{\eta-1},p_\T(T^1_\beta))\geq D^\circ(p_\T(T^1_\alpha),p_\T(T^1_\beta)).\] Therefore, we could erase $u_\tau,...,u_{\eta-1}$ from the sequence and replace it with $(p_\T(T^1_\alpha),p_\T(T^1_\beta))$, and it would decrease the value of the left-hand side of \eqref{2}. By iterating this procedure, we obtain a sequence which takes its values in $A_{L\varepsilon}$, and satisfies \eqref{2}. \\
        Thus, for every sequence $u_0,...,u_p$ such that $u_0=u,u_p=v$ and $u_k\in A_{L\varepsilon}$, we have : 
        \[\sum_{i=1}^pD^\circ(u_i,u_{i-1})=\sum_{i=1}^p\Tilde{D}_1^\circ(u_i,u_{i-1}).\] Taking the infimum over all such sequences, we obtain 
        \[D(x,y)=\Tilde{D}_1(x,y)\] which concludes the proof. 
\end{proof}
We proceed as in section $4.4$ of \cite{Browniandisk} to turn this local result into a global one.
\begin{corollary}\label{length2}
    Fix $x,y\in\Tilde{\s}_i$ and a continuous, injective path $f:[0,1]\rightarrow\Tilde{\s}_i$ such that $f(0)=x,f(1)=y$. Then 
    \begin{equation*}
        \mathrm{length}_D(f)=\mathrm{length}_{\Tilde{D}_i}(f)
    \end{equation*}
\end{corollary}
\begin{proof}
    Without loss of generality, fix $i=1$. Suppose first that $f$ does not visit $p_{\Tilde{\s}_1}(S_{L\varepsilon})$ and $p_{\Tilde{\s}_1}(p_\T(T^1_{(W^1)_* }))$. Then, for every $t\geq0$, $f(t)$ is either not in $\gamma_1^R$ or not in $\gamma_1^L$. Suppose that we are in the first case; by Proposition \ref{iso locale}, we can find a neighborhood $V_r$ of $r$ in $[0,1]$ and $\varepsilon_r>0$ such that for every $r'\in V_r$, $\Tilde{D}_1(f(r),f(r'))<\varepsilon_r/2$ and $\Tilde{D}_1(f(r'),\gamma_1^R)>\varepsilon_r$. By compactness, we can find $\varepsilon>0$ such that, for every $r,r'\in[0,1],|r-r'|<\varepsilon$ implies that 
    \begin{equation*}
        D(f(r),f(r'))=\Tilde{D}_1(f(r),f(r')).
    \end{equation*}
    Hence, for every partition $0=r_0<r_1<...<r_k=1$ such that $|r_{i+1}-r_i|<\varepsilon$ for every $1\leq i\leq k$, we have 
    \begin{equation*}
        \sum_{i=1}^k D(f(r_i),f(r_{i-1}))=\sum_{i=1}^k \Tilde{D}_1(f(r_i),f(r_{i-1})).
    \end{equation*}
    Taking the supremum over all partitions, we obtain 
    \begin{equation}\label{length}
        \text{length}_D(f)=\text{length}_{\Tilde{D}_1}(f).
    \end{equation}
    Now, by left-continuity and additivity of the length function, $\eqref{length}$ remains true if we allow $f$ to visit any point of $\Tilde{\s}_i$ (the reader is referred to the discussion at the end of section $4.4$ of \cite{Browniandisk} for more details), which concludes the proof.
\end{proof}
\subsection{End of the proof of Proposition \ref{localisation}} Now, we just need one last estimate to prove Proposition \ref{localisation}.
\begin{proposition}\label{Distance-frontière}
    For every $\delta>0$, there exists $L,C>0$ such that for every $\varepsilon>0$, 
    \[\N_0\bigg(\Tilde{D}\bigg(B_\varepsilon(\Tilde{\s},\gamma^L(L\varepsilon)),\gamma^R\bigg)>C\varepsilon\,\bigg|\,W_*<-L\varepsilon\bigg)>1-\delta.\]
\end{proposition}
\begin{proof}
    By the scaling property of the Brownian snake, it is enough to prove the result for $\varepsilon=1$. Fix $A>0$. Under $\N_0(\cdot\,|\,W_*<-A)$, we have 
    \[\Tilde{D}(\gamma^L(A),\gamma^R)>0\quad\text{a.s}.\]
    because the curves $\gamma^L$ and $\gamma^R$ are disjoint except at their endpoints. 
    Therefore, there exists $C_1,C_2>0$ such that 
    \[\N_0\bigg(\Tilde{D}\bigg(B_{C_1}(\Tilde{\s},\gamma^L(A)),\gamma^R\bigg)>C_2\,\bigg|\,W_*<-A\bigg)>1-\delta.\]
    By scaling, we also have 
    \begin{equation*}
        \N_0\bigg(\Tilde{D}\bigg(B_{1}(\Tilde{\s},\gamma^L(A/C_1)),\gamma^R\bigg)>C_2/C_1\,\bigg|\,W_*<-A/C_1\bigg)>1-\delta
    \end{equation*}
    which gives us the result by taking $C=C_2/C_1$ and $L=A/C_1$. 
\end{proof}
\begin{proposition}\label{Voisinage}
    Fix $\varepsilon>0$. Then, on the event $\{Y_H>L\varepsilon\}$, $p_{\s}(A_{L\varepsilon}\cup\widehat{A}_{L\varepsilon})$ is a neighborhood of $\Gamma_b$ in $\s$.
\end{proposition}
\begin{proof}
    Suppose that the statement does not hold. Then, we could find a sequence $(u_n)_{n\geq0}\in\T_b$ so that for every $n>0$, $D(u_n,p_\T(T_b))<\frac{1}{n}$ and $u_n\notin A_{L\varepsilon}\cup\widehat {A}_{L\varepsilon}$. By compactness, we could extract a subsequence of $(u_n)_{n\geq0}$ that converges toward an element $u\in\T_b$ that satisfies $u\notin A_{L\varepsilon}\cup\widehat A_{L\varepsilon}$ and $D(u,p_\T(T_b))=0$. However, this is not possible because $p_\T(T_b)$ and $p_\T(\widehat{T}_b)$ are the only elements of $\T$ that satisfy this equality. Therefore, $p_{\s}(A_{L\varepsilon}\cup\widehat A_{L\varepsilon})$ contains a ball centered at $\Gamma_b$, which concludes.
\end{proof}
\begin{proof}[Proof of Proposition \ref{localisation}]
Let $\mathcal{H}_\varepsilon$ the event where :
\begin{itemize}[label=\textbullet]
    \item $Y_H>L\varepsilon$
    \item $p_\T(T_{b+2\varepsilon})\in A_{L\varepsilon}$ and $p_\T(\widehat{T}_{b+2\varepsilon})\in\widehat{A}_{L\varepsilon}$
    \item $p_\T(T_{b-2\varepsilon})\in A_{L\varepsilon}$ and $p_\T(\widehat{T}_{b-2\varepsilon})\in\widehat{A}_{L\varepsilon}$
    \item $\Tilde{D}_1(B_\varepsilon(\Tilde{\s}_1,\gamma_1^L(L\varepsilon)),\gamma_1^R)>\varepsilon$ and $\Tilde{D}_2(B_\varepsilon(\Tilde{\s}_2,\gamma_2^R(L\varepsilon)),\gamma_2^L)>\varepsilon$.
\end{itemize}
By Propositions \ref{Encadrement}, \ref{Distance-frontière} and Lemma \ref{inc}, for every $\delta>0$, there exists $L>0$ such that $\N_0(\mathcal{H}_\varepsilon\,|\,W_*<-b)>1-\delta$ for every $\varepsilon$ small enough. We will show that \[\mathcal{H}_\varepsilon\subset\{B_\varepsilon(\s,\Gamma_b)\subset p_\s(A_{L\varepsilon}\cup\widehat{A}_{L\varepsilon})\},\] which will give the wanted result. \\
From now on, we work on the event $\mathcal{H}_\varepsilon$. Fix $u\in\T_b\backslash(A_{L\varepsilon}\cup\widehat{A}_{L\varepsilon})$, and let $\eta$ be a geodesic between $p_\s(u)$ and $\Gamma_b$. By Lemma \ref{Voisinage}, there exists $t_0>0$ such that for every $t\in[t_0,D(u,T_b)]$, $\eta(t)\in p_\s(A_{L\varepsilon}\cup\widehat{A}_{L\varepsilon})$. By Lemma \ref{boundary}, we have $\eta(t_0)\in\gamma_1^R\cup\gamma_1^L\cup\gamma_2^R\cup\gamma_2^L$. Without loss of generality, suppose that $\eta(t_0)\in\gamma_1^R\cup\gamma_1^L$. We can also suppose that for every $t\in[t_0,D(u,T_b)]$, $\eta(t)\in\Tilde{\s}_1$. Indeed, if we set $t_1=\sup\{t>t_0:\eta(t)\in\Tilde{\s}_2\}$, we must have $\eta(t_1)=\gamma_1^L(s_1)=\gamma_2^R(s_2)$ for some $s_1,s_2>0$, and the concatenation of $\eta|_{[t_0,t_1]}$ and $\gamma_1^L|_{[s_1\wedge L\varepsilon,s_1\vee L\varepsilon]}$ is a geodesic between $\eta(t_0)$ and $\Gamma_b$ which stays in $\Tilde{\s}_1$ (see Figure \ref{loc}).
\begin{figure}
    \centering
    \includegraphics[scale=0.75]{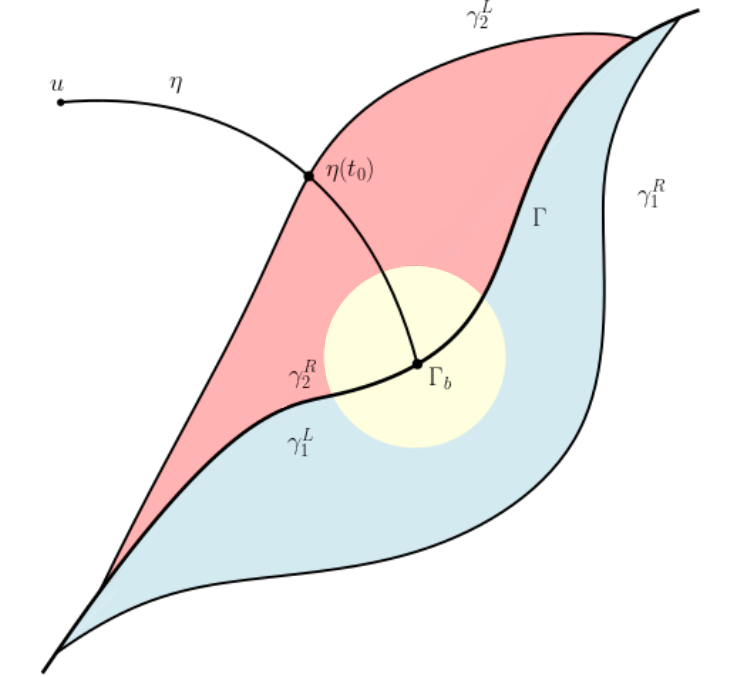}
    \caption{Illustration of the proof of Proposition \ref{localisation}.}
    \label{loc} The red and blue parts represent $\Tilde{\s}_1$ and $\Tilde{\s}_2$, and the yellow part is $B_\varepsilon(\s,\Gamma_b)$.
\end{figure} \\
Let us prove that on $\mathcal{H}_\varepsilon$, $\eta(t_0)\notin\Gamma|_{[b-\varepsilon,b+\varepsilon]}$, which will be enough to conclude the proof. Consider a sequence $(v_n)_{n\geq0}\in\T_b\backslash(A_{L\varepsilon}\cup\widehat{A}_{L\varepsilon})$ such that $\eta((t_0-\frac{1}{n})_+)=p_\s(v_n)$. By compactness, we can suppose that $(v_n)_{n\geq0}$ converges toward some $v\in\T\backslash(A_{L\varepsilon}\cup\widehat{A}_{L\varepsilon})$. There are two possibilities :
\begin{itemize}[label=\textbullet]
    \item If $|Z_v-b|\geq2\varepsilon$, then for every $a\in[-\varepsilon,\varepsilon]$,
    \[D^\circ(v,p_\T(T_{b+a}))\geq|Z_v-(b+a)|\geq\varepsilon>0.\]
    \item If $|Z_v-b|<2\varepsilon$, we have $v\in[p_\T(T_{b-2\varepsilon}),p_\T(\widehat{T}_{b-2\varepsilon})]$. Consequently, for every $a\in[-\varepsilon,\varepsilon]$, we have
    \begin{align*}
        D^\circ(v,p_\T(T_{b+a}))&=-b-a+Z_v-2\max\left(\inf_{w\in[v,p_\T(T_{b+a})]}Z_w,\inf_{w\in[p_\T(T_{b+a}),v]}Z_w\right)\\
        &>-b-a+Z_v-2(-b-2\varepsilon)>\varepsilon>0.
    \end{align*} 
\end{itemize}
Either way, on $\mathcal{H}_\varepsilon$, we have $\eta(t_0)\notin\Gamma|_{[b-\varepsilon,b+\varepsilon]}$. But note that on $\mathcal{H}_\varepsilon$, for every $x\in\gamma_1^R\cup\gamma_1^L$ such that $x\notin\Gamma|_{[b-\varepsilon,b+\varepsilon]}$, we have 
\[\Tilde{D}_1(x,\Gamma_b)>\varepsilon.\]
Hence, by Corollary \ref{length2}, 
\[D(u,\Gamma_b)\geq D(\eta(t_0),\Gamma_b)=\text{length}_D(\eta(t_0+\cdot))=\text{length}_{\Tilde{D}_1}(\eta(t_0+\cdot))\geq\Tilde{D}_1(\eta(t_0),\Gamma_b)>\varepsilon.\]
Therefore, $p_\s(u)\notin B_\varepsilon(\s,\Gamma_b)$, which concludes the proof. 
\end{proof}

\section{Convergence results}\label{section 5}
The purpose of this section is to prove Theorem \ref{convergence}, and some extensions of this theorem to other conditionings. The main idea of the proofs is to obtain some coupling results of the labeled trees that encode the different surfaces, and to use Proposition \ref{localisation} to obtain a coupling of the surfaces. 

\subsection{Proof of Theorem \ref{convergence}}

This subsection is devoted to the proof of Theorem \ref{convergence}. To do so, we will prove the following coupling result. 

\begin{proposition}\label{coupling}
    Fix $\delta>0$ and $r>0$. Then, there exists $\lambda_0$ such that, for every $\lambda>\lambda_0$, we can construct $\lambda\cdot\mathcal{S}$ and $\BP$ on the same probability space such that the equality 
    \begin{equation*}
        B_r(\lambda\cdot\mathcal{S},\Gamma_b)=B_r(\BP)
    \end{equation*}
    holds with probability at least $1-\delta$.
\end{proposition}
By definition of the convergence for the local Gromov-Hausdorff topology, the second statement of Theorem \ref{convergence} follows from Proposition \ref{coupling}, and the first one follows from the scale invariance of $\overline{\mathcal{BP}}$. \\ 
We need to describe ``how to compute distances'' between elements of $p_\s(A_{L\varepsilon})$ and $p_\s(\widehat{A}_{L\varepsilon})$. In particular, we want to show that distances behave as in the space $\BP $. First, we need to introduce some events. Fix $\delta>0$ as in Proposition \ref{coupling}. We define 
\[\mathcal{F}_\varepsilon:=\{B_\varepsilon(\s,\Gamma_b)\subset p_\s(A_{L\varepsilon}\cup\widehat{A}_{L\varepsilon})\}\cap\{B_{2\varepsilon}(\s,\Gamma_b)\subset p_\s(A_{2L\varepsilon}\cup\widehat{A}_{2L\varepsilon})\}\]
and
\[\mathcal{F}_\varepsilon^\infty:=\{B_\varepsilon(\BP,\overline{\rho}_\infty)\subset p_{\BP}(A_{L\varepsilon}^\infty\cup\widehat{A}_{L\varepsilon}^\infty)\}\cap\{B_{2\varepsilon}(\BP,\overline{\rho}_\infty)\subset p_{\BP}(A_{2L\varepsilon}^\infty\cup\widehat{A}_{2L\varepsilon}^\infty)\}\]
where $L>0$ is chosen large enough such that, for every $\varepsilon$ small enough $\p(\mathcal{F}_\varepsilon)>1-\delta/2$ and $\p(\mathcal{F}_\varepsilon^\infty)>1-\delta/2$, which exists by Proposition \ref{localisation}.
\begin{proposition}\label{d}
    Suppose that $\mathcal{F}_\varepsilon$ holds. Then, for every $u,v\in A_{L\varepsilon}$ such that $D(p_\s(u),\Gamma_b)<\varepsilon$, and $D(p_\s(v),\Gamma_b)<\varepsilon$, we have :
    \begin{equation}\label{f1}
        D(p_\s(u),p_\s(y))=\inf_{u_0,...,u_p}\sum_{i=1}^p D^\circ(u_i,u_{i-1})
    \end{equation} where the infimum is over all choices $p\in\N$ and sequences $u_0,...,u_p\in\T_b$ such that $u_0=u,u_p=v$ and $u_k\in A_{2L\varepsilon}$. 
\end{proposition}
\begin{proof}
    First, by Proposition \ref{localisation}, we know that we can restrict the choices of $u_0,...,u_k$ to elements of $A_{2L\varepsilon}\cup \widehat{A}_{2L\varepsilon}$. Then, consider a sequence $u_0,...,u_p$ and $0<r<r'<p$ such that $u_r,...,u_{r'}\in \widehat{A}_{L\varepsilon}$.
    First, note that if we set :
    \begin{equation*}
        \alpha:=\max\bigg(\inf_{u\in[u_r,u_{r-1}]}Z_u,\inf_{u\in[u_{r-1},u_r]}Z_u\bigg),
    \end{equation*}
   then, we have

    \begin{equation*}
        D^\circ(u_{r-1},u_r)=D^\circ(u_{r-1},p_\T(T_\alpha))+D^\circ(p_\T(T_\alpha),p_\T(\widehat{T}_\alpha))+D^\circ(p_\T(\widehat{T}_\alpha),u_r)
    \end{equation*}
    (note that the middle term vanishes). By symmetry, we have a similar identity for $D^\circ(u_{r'},u_{r'+1})$, where $\alpha$ is replaced by 
        \begin{equation*}
        \beta:=\max\bigg(\inf_{u\in[u_{r'},u_{r'+1}]}Z_u,\inf_{u\in[u_{r'+1},u_{r'}]}Z_u\bigg).
    \end{equation*}
    This implies that :
    \begin{align*}
        \sum_{i=r}^{r'+1}D^\circ(u_i,u_{i-1})&=D^\circ(u_{r-1},p_\T(T_\alpha))+D^\circ(p_\T(\widehat{T}_\alpha),u_r)+\sum_{i=r+1}^{r'}D^\circ(u_i,u_{i-1})+D^\circ(p_\T(T_\beta),u_{r'+1})\\
        &\geq D^\circ(p_\T(T_\alpha),p_\T(T_\beta))
    \end{align*}
    where the last inequality follows from \eqref{Bound} and the triangle inequality. 
    
Therefore, we could erase $u_r,...,u_{r'}$ from the sequence and replace it by $p_\T(T_\alpha),p_\T(T_\beta)$. But note that $p_\T(T_\alpha)\in A_{L\varepsilon}$ because $\alpha\in[-\varepsilon,\varepsilon]$ and the same holds for $p_\T(T_\beta)$. Therefore, by iterating the previous procedure, we obtain a sequence which satisfies the requirement of the statement, and this concludes the proof. 
\end{proof}
The previous proposition means that distances between element of $p_\s(A_{L\varepsilon})$ `` do not interact '' with elements of $p_\s(\widehat{A}_{L\varepsilon})$. By symmetry, Proposition \ref{d} also holds (with obvious modifications) if we take $u,v\in\widehat{A}_{L\varepsilon}$. The following proposition explains how to compute distances between elements that are on different sides of the geodesic $\Gamma$. 
\begin{proposition}\label{e}
     Suppose that $\mathcal{F}_\varepsilon$ holds. Then, for every $u\in A_{L\varepsilon}$, $v\in \widehat{A}_{L\varepsilon}$ such that $D(p_\s(u),\Gamma_b)<\varepsilon$ and $D(p_\s(v),\Gamma_b)<\varepsilon$, we have :
    \begin{equation}\label{f2}
        D(p_\s(u),p_\s(v))=\inf_{u_0,...,u_p}\sum_{i=1}^p D^\circ(u_i,u_{i-1})
    \end{equation} where the infimum is over all choices $p\in\N$ and sequences $u_0,...,u_p\in\T_b$ such that : 
    \begin{itemize}
        \item $u_0=u$ and $u_p=v$
        \item there exists $0\leq n\leq p$ such that $u_n=p_\T(T_s)$ for some $s\in\R$
        \item for every $0\leq k<n$, $u_k\in A_{2L\varepsilon}$, and for every $n<k\leq p$, $u_k\in \widehat{A}_{2L\varepsilon}$
    \end{itemize}
\end{proposition}
The previous proposition can be interpreted as follows: to compute the distance between elements of $p_\s(A_{L\varepsilon})$ and $p_\s(\widehat{A}_{L\varepsilon})$, we can restrict ourselves to paths that ``cross exactly once'' the geodesic $\Gamma$.
\begin{proof}
    As previously, we know that we can restrict our choices of $u_0,...,u_p$ to elements of $A_{2L\varepsilon}\cup \widehat{A}_{2L\varepsilon}$. Set $m=\min\{n>0:u_n\in \widehat{A}_{2L\varepsilon}\}$, and  
    \begin{equation*}
        \alpha:=\max\bigg(\inf_{u\in[u_m,u_{m-1}]}Z_u,\inf_{u\in[u_{m-1},u_m]}Z_u\bigg).
    \end{equation*} Note that 
    \begin{equation*}
        D^\circ(u_{m-1},u_m)=D^\circ(u_{m-1},p_\T(T_\alpha))+D^\circ(u_m,p_\T(T_\alpha)).
    \end{equation*}
    Therefore, we can suppose that there exists $k$ such that $u_k=p_\T(T_r)$ for $r\in[-2\varepsilon,2\varepsilon]$, and that for every $n<k$, $u_n\in A_{2L\varepsilon}$.
    \\
    Set $\ell=\inf\{n>k:u_n\in A_{2L\varepsilon}\}$, and suppose that $\ell<\infty$. Also define 
    \[\beta:=\max\bigg(\inf_{u\in[u_\ell,u_{\ell-1}]}Z_u,\inf_{u\in[u_{\ell-1},u_\ell]}Z_u\bigg).\] We have : 

    \begin{align*}
        \sum_{i=k+1}^{\ell}D^\circ(u_i,u_{i-1})&=\sum_{i=k+1}^{\ell-1}D^\circ(u_i,u_{i-1})+D^\circ(u_{\ell-1},p_\T(T_\beta))+D^\circ(p_\T(T_\beta),u_\ell)\\
        &\geq D^\circ(p_\T(T_r),p_\T(T_\beta))+D^\circ(p_\T(T_\beta),u_\ell)
    \end{align*}
    Where the last inequality comes from the bound \eqref{Bound} :
    \begin{equation*}
        \sum_{i=k+1}^{\ell-1}D^\circ(u_i,u_{i-1})+D^\circ(u_{\ell-1},p_\T(T_\beta))\geq\sum_{i=k+1}^{\ell-1}|Z_{u_i}-Z_{u_{i-1}}|+|Z_{u_{\ell-1}}-\beta|\geq|r-\beta|=D^\circ(p_\T(T_r),p_\T(T_\beta)).
    \end{equation*}
    Therefore, we could add $p_\T(T_\beta)$ to the sequence, and erase $u_{k+1},...,u_{\ell-1}$. By iterating this procedure, we obtain a sequence as described in the statement, which concludes the proof. 
\end{proof}
We give a similar result for the space $\overline{\mathcal{BP}}$. We omit the proof, which is similar and easier to those of Propositions \ref{d} and \ref{g}. 
\begin{proposition}\label{g}
    Suppose that $\mathcal{F}^\infty_\varepsilon$ holds.
    \begin{itemize}[label=\textbullet]
        \item  For every $u,v\in A^\infty_{L\varepsilon}$ such that $\overline{D}_\infty(p_{\overline{\mathcal{BP}}}(u),\overline{\rho}_\infty)<\varepsilon$, and $\overline{D}_\infty(p_{\overline{\mathcal{BP}}}(v),\overline{\rho}_\infty)<\varepsilon$, we have :
    \begin{equation}\label{f3}
        \overline{D}_\infty(p_{\overline{\mathcal{BP}}}(u),p_{\overline{\mathcal{BP}}}(v))=\inf_{u_0,...,u_p}\sum_{i=1}^p \overline{D}_\infty^\circ(u_i,u_{i-1})
    \end{equation} where the infimum is over all choices $p\in\N$ and sequences $u_0,...,u_p\in\T_\infty\cup\widehat{\T}_\infty
    $ such that $u_0=u,u_p=v$ and $u_k\in A^\infty_{2L\varepsilon}$.
    \item For every $u\in A^\infty_{L\varepsilon}$ and $v\in \widehat{A}^\infty_{L\varepsilon}$ such that $\overline{D}_\infty(p_{\overline{\mathcal{BP}}}(u),\overline{\rho}_\infty)<\varepsilon$, and $\overline{D}_\infty(p_{\overline{\mathcal{BP}}}(v),\overline{\rho}_\infty)<\varepsilon$, we have :
    \begin{equation}\label{f4}
        \overline{D}_\infty(p_{\overline{\mathcal{BP}}}(u),p_{\overline{\mathcal{BP}}}(v))=\inf_{u_0,...,u_p}\sum_{i=1}^p \overline{D}_\infty^\circ(u_i,u_{i-1})
    \end{equation} where the infimum is over all choices $p\in\N$ and sequences $u_0,...,u_p\in\T_\infty\cup\widehat{\T}_\infty$ such that : 
    \begin{itemize}
        \item $u_0=u$ and $u_p=v$
        \item there exists $0\leq n\leq p$ such that $u_n=T_u$ for some $u\in\R$
        \item for every $0\leq k<n$, $u_k\in A^\infty_{2L\varepsilon}$, and for every $n<k\leq p$, $u_k\in \widehat{A}^\infty_{2L\varepsilon}$
    \end{itemize}
    \end{itemize}
\end{proposition}

We can now state our first coupling result, which tells us that the local limit of the tree $\T$ around $p_\T(T_b)$ and $p_\T(\widehat{T}_b)$ is $(\T_\infty,\widehat{\T}_\infty)$. Recall that each set $A_{L\varepsilon},\widehat{A}_{L\varepsilon},A^\infty_L...$ is identified as a collection $(X,\mathcal{N},\widehat{\mathcal{N}})$ of a trajectory with two point measures. 

\begin{proposition}\label{f}
    For every $\delta>0$, there exists $\varepsilon_0>0$ such that, for every $\varepsilon<\varepsilon_0$, the following holds: we can couple $(A_{L\varepsilon},\widehat{A}_{L\varepsilon})$ and $(A_L^\infty,\widehat{A}_L^\infty)$ such that there exists a bijection $\Phi:A_{L\varepsilon}\cup\widehat{A}_{L\varepsilon}\longrightarrow A_L^\infty\cup\widehat{A}_L^\infty$ satisfying 
    \begin{equation}\label{6}
        D^\circ(u,v)=\varepsilon D^\circ_\infty(\Phi(u),\Phi(v))\quad\text{ for every $u,v\in A_{L\varepsilon}$ or $u,v\in \widehat{A}_{L\varepsilon}$}.
    \end{equation}
    with probability at least $1-\delta$. 
\end{proposition}
\begin{proof}
        We give the proof for $L=1$, the ideas being the same in the general case. First, note that the total variation distance between $(X|_{[0,S_{\varepsilon\wedge Y_H}]},\widehat{X}|_{[0,\widehat{S}_{\varepsilon\wedge Y_H}]})$ and $(X^\infty|_{[0,S^\infty_{\varepsilon}]},\widehat{X}^\infty|_{[0,\widehat{S}^\infty_{\varepsilon}]})$ goes to $0$, when $\varepsilon\rightarrow0$. 
            Indeed, let $F$ be a measurable function bounded by $1$. We have 
            \begin{equation}\label{Corinne}
            |\E[F(X|_{[0,S_{\varepsilon\wedge Y_H}]},\widehat{X}|_{[0,\widehat{S}_{\varepsilon\wedge Y_H}]})-F(X|_{[0,S_{\varepsilon\wedge Y_H}]},\widehat{X}|_{[0,\widehat{S}_{\varepsilon\wedge Y_H}]})\1_{\{Y_H>\varepsilon}\}]|\leq\p(Y_H<\varepsilon).
            \end{equation}
            Then, by independence, 
            \begin{align*}
           \E[F(X|_{[0,S_{\varepsilon\wedge Y_H}]},\widehat{X}|_{[0,\widehat{S}_{\varepsilon\wedge Y_H}]})\1_{\{Y_H>\varepsilon}\}]&=\E[F(X|_{[0,S_{\varepsilon}]},\widehat{X}|_{[0,\widehat{S}_{\varepsilon}]})]\p(Y_H>\varepsilon)\\           &=\E[F(X^\infty|_{[0,S^\infty_{\varepsilon}]},\widehat{X}^\infty|_{[0,\widehat{S}^\infty_{\varepsilon}]})]\p(Y_H>\varepsilon).
           \end{align*}
           Using \eqref{Corinne}, we obtain 
           \begin{equation}\label{35}
               |\E[F(X|_{[0,S_{\varepsilon\wedge Y_H}]},\widehat{X}|_{[0,\widehat{S}_{\varepsilon\wedge Y_H}]})-F(X^\infty|_{[0,S^\infty_{\varepsilon}]},\widehat{X}^\infty|_{[0,\widehat{S}_{\varepsilon}]})]|\leq2\p(Y_H<\varepsilon)\xrightarrow[]{\varepsilon\rightarrow0}0,
           \end{equation} 
           which gives us the result.\\
           
           Next, fix $\delta>0$. By \eqref{35} and the scaling property of Bessel processes, there exists $\varepsilon>0$ such that for every $0<\varepsilon<\varepsilon_0$, we can construct $(X|_{[0,S_{\varepsilon\wedge Y_H}]},\widehat{X}|_{[0,\widehat{S}_{\varepsilon\wedge Y_H}]})$ and $(X^\infty|_{[0,S^\infty_1]},\widehat{X}^\infty|_{[0,\widehat{S}_1]})$ on the same probability space such that with probability at least $1-\delta$, we have 
    \begin{equation}\label{couplage}
      X^\infty(t)=\varepsilon^{-1}X(t\varepsilon^{-2})\quad \text{and}\quad \widehat{X}^\infty(t')=\varepsilon^{-1}\widehat{X}(t'\varepsilon^{-2}),\quad \text{for every }0\leq t\leq S_1,\,0\leq t'\leq\widehat{S}_1. 
\end{equation}
Then, by scaling properties of the Brownian snake, we can couple ($\mathcal{N}_1,\mathcal{N}_2,\widehat{\mathcal{N}}_1,\widehat{\mathcal{N}}_2$) and ($\mathcal{N}^\infty_1,\mathcal{N}^\infty_2,\widehat{\mathcal{N}}^\infty_1,\widehat{\mathcal{N}}^\infty_2$) such that, on the event where \eqref{couplage} holds,
\begin{equation*}
    \mathcal{N}^\infty_1|_{[0,S_1]}=\{(t_i\varepsilon^{-2},\Theta_{\varepsilon^{-2}}(W^i)):(t_i,W^i)\in\mathcal{N}_1|_{[0,S_\varepsilon]}\} 
\end{equation*}
where $\Theta$ is the operator defined in \eqref{scaling} (and similarly for the other measures). \\
Then, on the event where \eqref{couplage} holds, we can naturally define an application $\phi$ which send a point $u\in A_{\varepsilon}$ to the corresponding element $\phi(u)\in A_1^\infty$ in the coupling. Notice that $\phi$ preserves intervals, in the sense that if $[u,v]\subset A_{\varepsilon}$ then $\phi([u,v])=[\phi(u),\phi(v)]\subset A_1^\infty$, and that $Z^\infty_{\phi(u)}=\varepsilon^{-1}(Z_u+b)$. \\
Finally, for $u,v\in A_{\varepsilon}$ such that $[u,v]\subset A_{\varepsilon}$, we have : 
\begin{align*}
D^\circ(u,v)&= Z_u+Z_v-2\inf_{w\in[u,v]}Z_w\\
&=\varepsilon(Z^\infty_{\phi(u)}+Z^\infty_{\phi(v)}-2\inf_{w\in[\phi(u),\phi(v)]}Z^\infty_w)\\
&=\varepsilon \overline{D}_\infty^\circ(\phi(u),\phi(v)). 
\end{align*}
Similarly, we can define an application $\widehat{\phi}$ between $\widehat{A}_{\varepsilon}$ and $\widehat{A}_1^\infty$, which satisfies the same properties, mutatis mutandis.\\
 Then, we can define $\Phi:A_{\varepsilon}\cup\widehat{A}_{\varepsilon}\longrightarrow A_L^\infty\cup\widehat{A}_L^\infty$ whose restriction to $A_{\varepsilon}$ (respectively $\widehat{A}_{\varepsilon}$) is $\phi$ (respectively $\widehat{\phi}$). By our previous results on $\phi$ and $\widehat{\phi}$, it is clear that $\Phi$ satisfies the identity \eqref{6} on the event where the processes are equal, which concludes the proof. 
\end{proof}
\begin{remark}\label{uniform1}
    Note that the proof shows that this result is uniform in $b$, in the following sense. Let $0<a<a'<\infty$. Then, for every $\delta>0$, there exists $\varepsilon_0>0$ such that for every $\varepsilon<\varepsilon_0$ and $b\in[a,a']$, we can couple $(A_{L\varepsilon},\widehat{A}_{L\varepsilon})$ under $\N_0(\cdot\,|\,W_*<-b)$ and $(A^\infty_{L},\widehat{A}^\infty_{L})$ as described in Proposition \ref{f}.
\end{remark}
The following result is a simple consequence of Propositions \ref{d}, \ref{e} and \ref{f}. 
\begin{proposition}\label{isométrie}
    Assume that $\mathcal{F_\varepsilon}$ holds. Fix $u,v\in A_{L\varepsilon}\cup\widehat{A}_{L\varepsilon}$. Then $D(\Gamma_b,p_\s(u))<\varepsilon$ and $D(\Gamma_b,p_\s(v))<\varepsilon$ if and only if $\overline{D}_\infty(\overline{\rho}_\infty,p_{\BP }(\Phi(u)))<1$ and $\overline{D}_\infty(\overline{\rho}_\infty,p_{\BP }(\Phi(v)))<1$. \\
    Furthermore, if those conditions are satisfied, we have : 
    \begin{equation*}
        D(p_\s(u),p_\s(v))=\varepsilon \overline{D}_\infty(p_{\BP }(\Phi(u)),p_{\BP }(\Phi(u))).
    \end{equation*}
\end{proposition}
\begin{proof}
    Fix $u,v\in A_{L\varepsilon}\cup\widehat{A}_{L\varepsilon}$ such that $D(\Gamma_b,p_\s(u))<\varepsilon$ and $D(\Gamma_b,p_\s(v))<\varepsilon$. By Lemmas \ref{d} and \ref{e}, $D(p_\s(u),p_\s(v))$ is given by \eqref{f1} or \eqref{f2}. Therefore, we have to show that it coincides with the formulas of \ref{g} (with a factor $\varepsilon$).\\
    Suppose that $u,v\in A_{L\varepsilon}$ (the other cases can be treated the same way), which implies that $\Phi(u),\Phi(v)\in A^\infty_L$, and consider a sequence $u_0,...,u_p$ satisfying conditions of Lemma \ref{d}. We can then define a new sequence $\Phi(u_0),...,\Phi(u_p)$, and it is immediate that it satisfies the conditions of the first statement of Lemma \ref{g}. But by Lemma \ref{f}, we have 
    \begin{equation*}
        \sum_{i=1}^p D^\circ(u_i,u_{i-1})=\varepsilon\sum_{i=1}^p \overline{D}_\infty^\circ(\Phi(u_i),\Phi(u_{i-1})).
    \end{equation*}
    To conclude, we just need to check that $\overline{D}_\infty(\overline{\rho}_\infty,p_{\BP }(\Phi(u))<\varepsilon$ and $\overline{D}_\infty(\overline{\rho}_\infty,p_{\BP }(\Phi(v))<\varepsilon$ in order to apply Lemma \ref{g}. However, the right-hand side of \eqref{f3} is an upper bound for $\overline{D}_\infty(p_{\BP }(\Phi(u)),p_{\BP }(\Phi(u)))$. Therefore, by taking $u=\overline{\rho}_\infty$ and then $v=\overline{\rho}_\infty$, we obtain that $\overline{D}_\infty(\overline{\rho}_\infty,p_{\BP }(\Phi(u))<\varepsilon$ and $\overline{D}_\infty(\overline{\rho}_\infty,p_{\BP }(\Phi(v))<\varepsilon$, which concludes the proof (the other cases can be treated the same way). 
\end{proof} 
\begin{proof}[Proof of Proposition \ref{coupling}]
    We work under the event described by Proposition \ref{isométrie}. By Proposition \ref{localisation}, any point of $B_1(\varepsilon^{-1}\cdot\s,\Gamma_b)$ is of the form $p_\s(u)$ for some $u\in A_{L\varepsilon}\cup\widehat{A}_{L\varepsilon}$. We can define
    \begin{equation*}
        \mathcal{I}(p_\s(u))=p_{\BP}(\Phi(u)).
    \end{equation*}
     By Lemma \ref{f}, this formula does not depend on the choice of $u$. Furthermore, by Proposition \ref{isométrie}, $\mathcal{I}$ is an isometry, and it maps $B_1(\varepsilon^{-1}\cdot\s,\Gamma_b)$ onto $B_1(\BP ,\overline{\rho}_\infty)$. Finally, we have $\mathcal{I}(\Gamma_b)=\overline{\rho}_\infty$, which concludes the proof. 
\end{proof}
Next, we extend the result of Theorem \ref{convergence} for the GHPU topology. To this end, for every $\varepsilon>0$ and $t\in[-b\varepsilon^{-1},-(W_*+b)\varepsilon^{-1}]$, we define
\begin{equation*}
    \Gamma^\varepsilon(t)=
        \Gamma(b+\varepsilon  t).
\end{equation*}
\begin{proposition}\label{convergence GHPU}
    We have the following convergence
    \begin{equation*}
        \left(\s,\varepsilon^{-1}\cdot D,\varepsilon^{-4}\cdot\mu,\Gamma^\varepsilon\right)\xrightarrow[\varepsilon\rightarrow0]{(d)}\left(\BP,\overline{D},\overline{\mu},\Gamma_\infty\right)
    \end{equation*}
    in distribution for the local GHPU topology.
\end{proposition}
\begin{proof}
    Note that the application $\mathcal{I}$ that appears in the proof of Proposition \ref{coupling} satisfies $\mathcal{I}(\Gamma^\varepsilon)=\Gamma_\infty|_{[-1,1]}$ and $\mathcal{I}_*(\varepsilon^{-4}\mu)|_{B_1(\varepsilon^{-1}\cdot\s,\Gamma_b)}=\overline{\mu}|_{B_1(\BP,\overline{\rho}_\infty)}$. The result is just a consequence of this observation and Proposition \ref{coupling}.
\end{proof}

\subsection{Conditioning on the length of $\Gamma$}

Now, we would like to extend the result of Theorem \ref{convergence} to other conditionings for the measure $\N_0$. Indeed, the proof of Proposition \ref{coupling} only relies on the local structure of the tree $\T_b$ around $p_\T(T_b)$ and $p_\T(\widehat{T}_b)$. Therefore, it is reasonable to think that various conditionings (on the volume or the minimal label) do not change the local behavior of the tree and, a fortiori, the convergence result. In this subsection, we prove that Theorem \ref{convergence} still holds under the probability measure $\N_0(\cdot\,|\,W_*=-c)$. Once again, the proof relies on a coupling of labeled trees.\\
First, let us stress that here, $A_{L\varepsilon}$ is defined for every snake trajectory such that $W_*<-b $, as a collection of processes $(X,\mathcal{N},\mathcal{N}')$, where $X$ encodes the labels of the genealogical line of $T_b$ up to its last hitting time of $L\varepsilon$, and $\mathcal{N},\mathcal{N}'$ are point measures describing the subtrees branching off this path. As discussed at the end of Subsection \ref{three arms}, we can reconstruct the snake trajectory from these three processes. Similarly, we can always define $Y_H$ as $Z_{p_\T(T_b)\wedge p_\T(\widehat{T}_b)}+b$, where $u\wedge v$ stands for the most recent common ancestor of $u$ and $v$ in $\T$. This definition of $Y_H$ is coherent with the one used under the probability measure $\N_0(\cdot\,|\,W_*<-b)$ throughout this article. We begin with an intermediate result, that will be used several times in the forthcoming proofs. 
\begin{proposition}\label{Independence}
    For every $\varepsilon>0$, set $V_\varepsilon=A_{L\varepsilon}\cup\widehat{A}_{L\varepsilon}$. Then, under $\N_0(\cdot\,|\,W_*<-b)$ and conditionally on $\{Y_H>L\varepsilon\}$, the random variables $V_\varepsilon$ and $\T_b\backslash V_\varepsilon$ are independent. 
\end{proposition}
\begin{proof}
    Recall that if $X$ is a Bessel process of dimension $7$ and $S_x$ its last hitting time of $x$, then $(X_t)_{0\leq t\leq S_x}$ and $(X_{S_x+t})_{t\geq0}$ are independent. Therefore, the result follows from Proposition \ref{construction 3 arms}, the definition of $V_\varepsilon$ and independence properties of Poisson point measures. 
\end{proof}
\begin{theorem}\label{conditionnement}
        Fix $0<b<1$. Then, the total variation distance between the law of $V_\varepsilon$ under $\N_0(\cdot\,|\,W_*<-b)$ and $V_\varepsilon$ under $\N_0(\cdot\,|\,W_*=-1)$ goes to $0$ when $\varepsilon\rightarrow 0$. 
    \end{theorem}
    \begin{proof}[Proof of Theorem \ref{conditionnement}]
        Fix $\delta>0$, and let $F$ be a measurable function bounded by $1$. Let us define $H_\varepsilon:=\{\inf_{u\in V_\varepsilon} Z_u>-1\}$ and $G_\varepsilon:=\{Y_H>L\varepsilon\}$.  We have :
    \[|\N_0(F(V_\varepsilon)\1_{H_\varepsilon}\1_{G_\varepsilon}-F(V_\varepsilon)|W_*=-1)|\leq\N_0(\1_{(H_\varepsilon\cup G_\varepsilon)^c}|W_*=-1)\leq\delta/3\] 
         for $\varepsilon$ small enough, by monotone convergence. Note that this holds uniformly in the choice of $F$. Then, we have 
        \begin{align*}
            \N_0(F(V_\varepsilon)\1_{H_\varepsilon}\1_{G_\varepsilon}|W_*=-1)&=\lim_{\eta\rightarrow 0}\N_0(F(V_\varepsilon)\1_{H_\varepsilon}\1_{G_\varepsilon}|-1-\eta\leq W_*\leq-1)\\
            &=\lim_{\eta\rightarrow0}C_\eta^{-1}\N_0(F(V_\varepsilon)\1_{H_\varepsilon}\1_{\{Y_H>L\varepsilon\}}\1_{\{-1-\eta\leq W_*\leq-1\}}|W_*<-b)
        \end{align*} where 
        \[C_\eta=\frac{\N_0(-1-\eta<W_*<-1)}{\N_0(W_*<-b)}=\N_0(-1-\eta<W_*<-1|W_*<-b).\]
    \end{proof}
    By \eqref{inf}, we have 
    \begin{equation*}
     C_\eta=\frac{1-\frac{1}{(1+\eta)^2}}{\frac{1}{b^2}}=b^2-\big(\frac{b}{1+\eta}\big)^2.   
    \end{equation*}
    On the other hand, using Propositions \ref{poisson}, \ref{Bessel} and the representation of section \ref{three arms}, under $\N_0(\cdot \,|\,W_*<-b)$, the event $\{-1-\eta<W_*<-1\}$ occurs if and only if, after a shift of $b$ of the labels, every atom $W^i$ of $\widehat{\mathcal{N}}_1$ and $\widehat{\mathcal{N}}_2$ satisfy $(W^i)_*>-1-\eta$, and at least one satisfies $-1-\eta+<(W^i)_*<-1$. If we denote this event by $\mathcal{K}_\eta$, we have :
    \begin{multline*}
    C_\eta=\N_0(\mathcal{K}_\eta\,|\,W_*<-b)=\E\bigg[\Pi_D\bigg(\mathcal{M}\bigg(\{(s,\omega):\omega_*<-1-\eta\}\bigg)=0\bigg)\\
    \times\Pi_D\bigg(\mathcal{M}\bigg(\{(s,\omega):\omega_*<-1\}\bigg)\geq1\bigg)\,\bigg|\,\mathcal{M}\bigg(\{(s,\omega):\omega_*<-1-\eta\})=0\bigg)\bigg]
    \end{multline*}
    where $D$ is the concatenation of $X-b$ with the process $\widehat{X}-b$ reversed in time. More precisely, for every $x>0$, we consider a probability measure $\p_x$ such that under $\p_{x}(\cdot)$, $D$ is the concatenation of two independent processes $(X,X')$, where $X+b$ is a Bessel process of dimension $7$ starting from $0$ and stopped at its last hitting time of $x$, and $X'+b$ is a Bessel process of dimension $-3$ starting from $x$ and stopped when it reaches $0$. We keep the notation $S_y$ for the last hitting time of $y$ by the process $X$, and we set 
    \[\tau'_y=\inf\{t\geq\zeta_X:D_t+b=y\}\]
    where $\zeta_X$ denotes the lifetime of $X$. Note that $\tau'_y-\zeta_X$ corresponds to the first hitting time of $y$ by $X'$ (in particular, $\tau'_0$ is the lifetime of $D$).
    With this notation, we have : 
    \begin{equation}\label{POI}    C_\eta=\E\bigg[\E_{Y_H}\bigg[\exp\bigg(-4\int_0^{\tau'_0}\N_{D_t}(W_*<-1-\eta)dt\bigg)\bigg(1-\exp\bigg(-4\int_0^{\tau'_0}\N_{D_t}(-1-\eta\leq W_*\leq-1)dt\bigg)\bigg)\bigg]\bigg]
    \end{equation} 
    Note that under the event $H_\varepsilon$, the event $\{-1-\eta<W_*<-1\}$ can be described as the intersection of two events :
    \[\{-1-\eta<W_*<-1\}\cap H_\varepsilon=\left\{\inf_{u\in\T\backslash V_\varepsilon} Z_u\in[-1-\eta,-1]\right\}\cap H_\varepsilon.\]
    Let $J_\varepsilon$ denotes the first event of the right-hand side. By Proposition \ref{Independence}, conditionally on $\{Y_H>L\varepsilon\}$, $J_\varepsilon$ and $H_\varepsilon$ are independent. Hence, we have :
    \begin{multline}\label{coupure}
        \N_0(F(V_\varepsilon)\1_{H_\varepsilon}\1_{\{-1-\eta\leq W_*\leq-1\}}\,|\,W_*<-b,Y_H>L\varepsilon)=\N_0(J_\varepsilon\,|\,W_*<-b,Y_H>L\varepsilon)\\       \times\N_0(F(V_\varepsilon)\1_{H_\varepsilon}\,|\,W_*<-b,Y_H>L\varepsilon).
    \end{multline}
    As previously, we have :
    \begin{multline*}
   \N_0(J_\varepsilon,\,Y_H>L\varepsilon\,|\,W_*<-b)=\E\bigg[\1_{\{Y_H>L\varepsilon\}}\E_{Y_H}\bigg[\exp\bigg(-4\int_{S_{L\varepsilon}}^{\tau'_{L\varepsilon}}\N_{D_t}(W_*<-1-\eta)dt\bigg)\\
   \times\bigg(1-\exp\bigg(-4\int_{S_{L\varepsilon}}^{\tau'_{L\varepsilon}}\N_{D_t}(-1-\eta\leq W_*\leq-1)dt\bigg)\bigg)\bigg]\bigg]
   \end{multline*}
   We will need the following lemma. 
    \begin{lemme}\label{limite}
        We have 
        \begin{equation}\label{gne}
            \lim_{\eta\rightarrow 0}\frac{C_\eta}{\eta}=\E\bigg[\E_{Y_H}\bigg[\exp\bigg(-4\int_0^{\tau'_0}\N_{D_t}(W_*<-1)dt\bigg)\bigg(12\int_0^{\tau'_0}\frac{1}{(1+D_t)^3}dt\bigg)\bigg]\bigg]
        \end{equation}
        and
         \begin{multline*}
    \lim_{\eta\rightarrow0}\frac{\N_0(J_\varepsilon,Y_H>L\varepsilon\,|\,W_*<-b)}{\eta}=\E\bigg[\1_{\{Y_H>L\varepsilon\}}\E_{Y_H}\bigg[\exp\bigg(-4\int_{S_{L\varepsilon}}^{\tau'_{L\varepsilon}}\N_{D_t}(W_*<-1)dt\bigg)\\
    \times\bigg(12\int_{S_{L\varepsilon}}^{\tau'_{L\varepsilon}}\frac{1}{(1+D_t)^3}dt\bigg)\bigg]\bigg].
        \end{multline*}
    \end{lemme}    
Before proving this lemma, let us show how to conclude the proof of \ref{conditionnement}. Using \eqref{coupure} and Lemma \ref{limite}, we obtain that 
        \begin{equation}\label{densite}             \N_0(F(V_\varepsilon)\1_{H_\varepsilon}\1_{J_\varepsilon}\,|W_*=-1)=\frac{\N_0(W_*<-b)}{\N_0(W_*<-b,Y_H>L\varepsilon)}\frac{f(L\varepsilon)}{f(0)}\N_0(F(V_\varepsilon)\1_{H_\varepsilon}\1_{\{Y_H>L\varepsilon\}}\,|\,W_*<-b)
        \end{equation}
        where 
        \begin{equation*}
            f(s)=\E\bigg[\1_{\{Y_H>s\}}\E_{Y_H}\bigg[\exp\bigg(-4\int_{S_s}^{\tau'_s}\N_{D_t}(W_*<-1+b)dt\bigg)\bigg(12\int_{S_s}^{T'_s}\frac{1}{(1+D_t)^3}dt\bigg)\bigg]\bigg].
        \end{equation*}
        The result follows from the fact that $f$ is continuous, which is a consequence of the monotone convergence theorem.
    \begin{proof}[Proof of Lemma \ref{limite}]
       We just give the proof of \eqref{gne}. We will use \eqref{POI}; note that, by continuity and monotone convergence, 
       \[\lim_{\eta\rightarrow0}\exp\bigg(-4\int_0^{\tau'_0}\N_{D_t}(W_*<-1-\eta+b)dt\bigg)=\exp\bigg(-4\int_0^{\tau'_0}\N_{D_t}(W_*<-1)dt\bigg)\quad \text{a.s.}\]
       We also have, using \eqref{inf} :
       \begin{align*}
           &\lim_{\eta\rightarrow0}\eta^{-1}\bigg(1-\exp\bigg(-4\int_0^{\tau'_0}\N_{D_t}(-1-\eta\leq W_*\leq-1)dt\bigg)\bigg)\\
           &=\lim_{\eta\rightarrow 0} 4\eta^{-1}\int_0^{\tau'_0}\N_{D_t}(-1-\eta\leq W_*\leq-1)dt\\
           &=\lim_{\eta\rightarrow0}6\eta^{-1}\int_0^{\tau'_0}\bigg(\frac{1}{(1+D_t)^2}-\frac{1}{(1+\eta+D_t)^2}\bigg)dt\\
           &= 12\int_0^{\tau'_0}\frac{1}{(1+D_t)^3}dt\quad\text{a.s.}
       \end{align*}
       The result will follow if we can use the dominated convergence theorem. To justify this, note that there exists some constant $C,C_1,C_2>0$ such that :
       \begin{align*}
        \exp\bigg(-4\int_0^{\tau'_0}\N_{D_t}(W_*<-1-\eta)dt\bigg)
        \leq\exp(-C_1T'_0)  
       \end{align*}       
       and 
       \begin{align*}
           &\eta^{-1}\bigg(1-\exp\bigg(-4\int_0^{\tau'_0}\N_{D_t}(-1-\eta\leq W_*\leq-1)dt\bigg)\bigg)\\
           &\leq 4\eta^{-1} \int_0^{\tau'_0}\N_{D_t}(-1-\eta\leq W_*\leq-1)dt\leq C_2\tau'_0
       \end{align*}
       
       (we used \eqref{inf} and the mean value inequality). Hence, 
       \begin{align*}
          &\exp\bigg(-4\int_0^{\tau'_0}\N_{D_t}(W_*<-1-\eta)dt\bigg)\bigg(1-\exp\bigg(-4\int_0^{\tau'_0}\N_{D_t}(-1-\eta\leq W_*\leq-1)dt\bigg)\bigg)\\
          &\leq C_2 \tau'_0\exp(-C_1 \tau' _0)\leq C 
       \end{align*}
       for some constant $C,C_1,C_2>0$. Consequently, we can apply the dominated convergence theorem, which concludes the proof. 
    \end{proof}
    \begin{remark}
        A closer look at \eqref{densite} shows that the convergence of Theorem \ref{conditionnement} is uniform in $b\in[1/4,3/4]$, a fact that will be used later.
    \end{remark}
    We can now use the local couplings of the tree to deduce a convergence result for the spaces.
    \begin{theorem}\label{convergence conditionnée}
        For every $\delta>0$ and $0<b<c$, we can find $\varepsilon>0$ such that we can couple $\s$ under $\N_0(\cdot\,|\,W_*=-c)$ and $\overline{\mathcal{BP}}$ so that $B_\varepsilon(\s,\Gamma_b)$ and $B_\varepsilon(\overline{\mathcal{BP}})$ are isometric with probability at least $1-\delta$. Moreover, we have the convergence 
    \begin{equation}   (\s,\varepsilon^{-1}D,\Gamma_b)\xrightarrow[\varepsilon\rightarrow 0]{(d)}(\overline{\mathcal{BP}},\overline{D}_\infty,\overline{\rho}_\infty)
    \end{equation}
    in distribution for the local Gromov-Hausdorff topology. 
    \end{theorem}
    \begin{proof}
        First, by scaling properties of the Brownian snake, it is enough to prove the result for $c=1$ and $0<b<1$. \\
        Set $V_{\varepsilon}^\infty=A_{L\varepsilon}^\infty\cup\widehat{A}_{L\varepsilon}^\infty\subset\T_\infty\cup\widehat{\T}_\infty$. Then, using Theorem \ref{conditionnement}, for every $\delta>0$, we can choose $\varepsilon_0>0$ small enough so that for every $0<\varepsilon<\varepsilon_0$, we can couple the sets $V_{2\varepsilon}$ of $\T_b$, $V_{2\varepsilon}$ of $\T$ under $\N_0(\cdot\,|\,W_*=-1)$ and $V_{2\varepsilon}^\infty$ in a way that these sets are equal with probability at least $1-\delta$. On this event, using Lemma \ref{isométrie}, we obtain (with probability at least $1-\delta$) an isometry between $B_\varepsilon(\s,\Gamma_b)$ under $\N_0(\cdot\,|\,W_*=-1)$ and $B_\varepsilon(\overline{\mathcal{BP}},{\overline{\rho}_\infty})$, which concludes the proof.
    \end{proof}
    \begin{remark}\label{uniform2}
      By adapting the proofs and using the same ideas, one can show that the results of Theorem \ref{conditionnement} and \ref{convergence conditionnée} hold if we replace $\N_0(\cdot\,|\,W_*=-c)$ by $\N_0(\cdot\,|\,W_*\leq-c)$. 
    \end{remark}
    \subsection{Conditioning on the volume}
Our next goal is to show that the convergence result still holds under the measure $\N_0^{(1)}(\cdot\,|\,W_*<-b)$. Even though the ideas are similar to the one used in Proposition \ref{conditionnement}, the proof is a little bit more technical. Indeed, we don't have explicit formulas for the law of the volume (analog to Lemma \ref{limite}). 
   \begin{proposition}\label{conditionnement volume}
    The total variation distance between the law of $V_\varepsilon$ under $\N_0(\cdot\,|\,W_*<-b)$ and $V_\varepsilon$ under $\N_0^{(1)}(\cdot\,|\,W_*<-b)$ goes to $0$ as $\varepsilon\rightarrow0$. 
\end{proposition}
\begin{proof}
    The first part of the proof is similar to the one of Theorem \ref{conditionnement}. Consider $F$ a measurable function bounded by $1$. Note that the volume $\sigma$ of the tree $\T$ can be decomposed into 2 parts $\sigma_\varepsilon+\overline{\sigma}_\varepsilon$ where 
    \begin{itemize}
        \item $\sigma_\varepsilon$ is the volume of $V_\varepsilon$
        \item $\overline{\sigma}_\varepsilon$ is the volume of $\T\backslash V_\varepsilon$.
    \end{itemize} 
   
    Fix $\delta>0$. We also define the following events : 
    \begin{itemize}
        \item $G_\varepsilon:=\{Y_H>L\varepsilon\}$
        \item $H_\varepsilon:=\{\sigma_\varepsilon<1\}$
        \item $J_\varepsilon^{\delta}:=\left\{\min{\left(d_\T(g(S_{L\varepsilon}),p_\T(T_b)\wedge p_\T(\widehat{T}_b)),d_\T(\widehat{g}(\widehat{S}_{L\varepsilon}),p_\T(T_b)\wedge p_\T(\widehat{T}_b))\right)}>\delta'\right\}$.
    \end{itemize}
    where $\delta'$ is chosen small enough such that $J_\varepsilon^{\delta}$ has probability at least $1-\delta$ (these events are also defined for $\varepsilon=0$), and $g(S_{L\varepsilon})$ (respectively $\widehat{g}(\widehat{S}_{L\varepsilon})$) is the last point on the geodesic between $p_\T(T_b)$ (respectively $p_\T(\widehat{T}_b)$) with label $-b+L\varepsilon$.
    Note that the probability of $G_\varepsilon$ and $H_\varepsilon$ go to $1$ as $\varepsilon\rightarrow0$. Therefore, 
    \begin{equation}\label{GG}
    \limsup_{\varepsilon\rightarrow0}\left|\,\N_0^{(1)}(F(V_\varepsilon)\,|\,W_*<-b)-\N_0^{(1)}(F(V_\varepsilon)\1_{H_\varepsilon}\,|\,W_*<-b,\,G_\varepsilon,\,J_\varepsilon^\delta)\,\right|\leq2\delta
    \end{equation}
    uniformly in $F$. Then, we have 
    \begin{align}\label{JJ}
        \N_0^{(1)}(F(V_\varepsilon)\1_{H_\varepsilon}\,|\,W_*<-b,\,G_\varepsilon,\,J_\varepsilon^\delta)&=\lim_{\eta\rightarrow0}\N_0(F(V_\varepsilon)\1_{H_\varepsilon}\,|\,W_*<-b,\,Y_H>L\varepsilon,\,J_\varepsilon^\delta,\,1<\sigma<1+\eta)\nonumber \\        &=\lim_{\eta\rightarrow0}C^\delta_\eta\,\N_0(F(V_\varepsilon)\1_{H_\varepsilon}\1_{\{1<\sigma<1+\eta\}}\,|\,W_*<-b,\,Y_H>L\varepsilon,\,J_\varepsilon^\delta)
    \end{align}
    where 
    \[C^\delta_\eta=\frac{\N_0(W_*<-b,\,Y_H>L\varepsilon,\,J_\varepsilon^\delta)}{\N_0(W_*<-b,\,Y_H>L\varepsilon,\,J_\varepsilon^\delta,\,1<\sigma<1+\eta)}=\N_0(1<\sigma<1+\eta\,|\,W_*<-b,\,Y_H>L\varepsilon,\,J_\varepsilon^\delta)^{-1}\]
\end{proof}
By Proposition \ref{Independence} and conditioning on the value of $\sigma_\varepsilon$, we have
\begin{align}\label{Markov slice}   
&\N_0(F(V_\varepsilon)\1_{\{\sigma_\varepsilon<1\}}\1_{\{1<\sigma<1+\eta\}}|W_*<-b,\,Y_H>L\varepsilon,\,J_\varepsilon^{\delta})\\
&=\N_0\left(F(V_\varepsilon)\1_{\{\sigma_\varepsilon<1\}}\N_0(1-\sigma_\varepsilon<\overline{\sigma}'_\varepsilon<1+\eta-\sigma_\varepsilon\,|\,W_*<-b,\,Y_H>L\varepsilon,\,J_\varepsilon^{\delta})\,|\,W_*<-b,\,Y_H>L\varepsilon\right)\nonumber
\end{align} where $\overline{\sigma}'_\varepsilon$ is independent of $\sigma_\varepsilon$. We will need the following lemma. 
\begin{lemme}\label{densité}
    Under the probability measure $\N_0(\cdot\,|\,W_*<-b,\,Y_H>L\varepsilon,\,J_\varepsilon^{\delta})$, the random variable $\overline{\sigma}_\varepsilon$ has a continuous density with respect to the Lebesgue measure. 
\end{lemme}
\begin{proof}
   For $\varepsilon>0$, we define 
   \[M_{\delta'}=\sum_{(t,\omega)\in I_2,S_{L\varepsilon}<t<S_{L\varepsilon}+\delta'}\sigma(\omega_i),\quad \widehat{M}_{\delta'}=\sum_{(t,\omega)\in J_2,\widehat{S}_{L\varepsilon}<t<\widehat{S}_{L\varepsilon}+\delta'}\sigma(\omega_i),\quad\]
   Under $\N_0(\cdot\,|\,W_*<-b,\,Y_H>L\varepsilon,\,J_\varepsilon^{\delta})$, the random variables $\overline{\sigma}_\varepsilon$ can be decomposed as $\overline{\sigma}_\varepsilon=M_{\delta'}+\widehat{M}_{\delta'}+\overline{M}_\varepsilon$, where these three random variables are independent. Indeed, the Poisson point measures $I_2$ and $J_2$ only depend on the processes $(X,\widehat{X})$ through the labels of their root, which is independent of the volume of the trees.  Moreover, the Laplace transform of $M_{\delta'}$ can be computed explicitly :
   \begin{align*}
   \N_0(\exp(-sM_{\delta'})\,|\,W_*<-b,\,Y_H>L\varepsilon,\,J_\varepsilon^{\delta})&=\exp\left(-2\int_{S_\varepsilon}^{S_\varepsilon+\delta'}\N_0(1-\exp(-s\sigma))dt\right)\\
   &=\exp\left(-2\delta'\sqrt{\frac{s}{2}}\right).
   \end{align*}
   This is also the Laplace transform of a random variable with density \[q_{\delta'}(t)=\frac{\delta'\exp(-\delta'^2/(2t))}{\sqrt{2\pi}t^{3/2}},\] which is the density of the first hitting time of $\delta'$ by a Brownian motion. Since $M_{\delta'}$ has a continuous density, so does $\overline{\sigma}_\varepsilon$, which concludes the proof.
\end{proof} 
In what follows, we write :
\begin{itemize}[label=\textbullet]
    \item $p_\varepsilon^\delta$ for the density of $\sigma$ under $\N_0(\cdot\,|\,W_*<-b,\,Y_H>L\varepsilon,\,J_\varepsilon^{\delta})$
    \item $\overline{p}_\varepsilon^\delta$ for the density of $\overline{\sigma}_\varepsilon$ under $\N_0(\cdot\,|\,W_*<-b,\,Y_H>L\varepsilon,\,J_\varepsilon^{\delta})$
    \item $p_0^\delta$ for the density of $\sigma$ under $\N_0(\cdot\,|\,W_*<-b,\,J_0^\delta)$
\end{itemize} (using the same ideas as in Lemma \ref{densité}, we can prove that these quantities are all well-defined and continuous). 
\begin{lemme}\label{uniform}
    We have
    \[\lim_{\varepsilon\rightarrow0}p_\varepsilon^\delta=\lim_{\varepsilon\rightarrow0}\overline{p}_\varepsilon^\delta=p_0^\delta\] 
    where the convergence holds uniformly on compacts. 
\end{lemme}
\begin{proof}
    We just give the proof for the sequence $(\overline{p}_\varepsilon^\delta)_{\varepsilon\geq0}$. First, note that $\overline{\sigma}_\varepsilon\xrightarrow[\varepsilon\rightarrow0]{(d)}\sigma$ (under $\N_0(\cdot\,|\,W_*<-b,\,J_0^\delta)$). To obtain the convergence of the densities, one needs to check the sequence $(\overline{p}_\varepsilon^\delta)_{\varepsilon\geq0}$ is uniformly continuous. For every $x,y\geq0$, we have 
    \[|\overline{p}_\varepsilon^\delta(x)-\overline{p}_\varepsilon^\delta(y)|\leq\int_\R|(\exp(ixt)-\exp(iyt))\phi_{\overline{\sigma}_\varepsilon}(t)|dt\leq|x-y|\int_\R |t\phi_{\overline{\sigma}_\varepsilon}(t)|dt\]
  where $\phi_\xi$ stands for the characteristic function of $\xi$. Therefore, we have to show that
    \[\sup_{\varepsilon>0}\int_\R |t\phi_{\overline{\sigma}_\varepsilon}(t)|dt<\infty.\]
    The decomposition given in the proof of Lemma \ref{densité} shows that, for every $\varepsilon>0$, we have $|\phi_{\overline{\sigma}_\varepsilon}|\leq|\phi_{M_{\delta'}}|$. Moreover, we have 
    \[\phi_{M_{\delta'}}(t)=\begin{cases}
    \delta'\exp((\delta')^{-1}(-1+i)\sqrt{t})\quad\text{ if }t\geq0\\
    \delta'\exp((\delta')^{-1}(-1-i)\sqrt{-t})\quad\text{ if }t\leq0
\end{cases}\]
In particular, one has 
\[\int_\R|t\phi_{M_{\delta'}}(t)|dt<\infty\] which concludes the proof.
\end{proof}
\begin{proof}[End of the proof of Proposition \ref{conditionnement volume}]
By Lemma \ref{densité}, \eqref{Markov slice} is equal to 
\[\N_0\left(F(V_\varepsilon)\1_{\{\sigma_\varepsilon<1\}}\frac{\int_{1-\sigma_\varepsilon}^{1+\eta-\sigma_\varepsilon}\overline{p}^\delta_\varepsilon(x)dx}{\int_{1}^{1+\eta}p_\varepsilon^\delta(x)dx}\,\bigg|\,W_*<-b,Y_H>L\varepsilon\right).\]
Taking the limit as $\eta\rightarrow0$ and using \eqref{JJ}, we have 
\[\N_0^{(1)}(F(V_\varepsilon)\1_{H_\varepsilon}\,|\,W_*<-b,\,G_\varepsilon,\,J_\varepsilon^\delta)=\N_0\left(F(V_\varepsilon)\1_{\{\sigma_\varepsilon<1\}}\frac{\overline{p}^\delta_\varepsilon(1-\sigma_\varepsilon)}{p_\varepsilon^\delta(1)}|W_*<-b,Y_H>L\varepsilon\right)\]
Since $\sigma_\varepsilon\xrightarrow[\varepsilon\rightarrow0]{}0$ almost surely, Lemma \ref{uniform} implies that $\lim_{\varepsilon\rightarrow0}\frac{\overline{p}^\delta_\varepsilon(1-\sigma_\varepsilon)}{p_\varepsilon^\delta(1)}=1$ a.s. Therefore, we have
\begin{equation}\label{Presk}
\left|\N_0^{(1)}(F(V_\varepsilon)\1_{H_\varepsilon}\,|\,W_*<-b,\,G_\varepsilon,\,J_\varepsilon^\delta)-\N_0(F(V_\varepsilon)\1_{\{\sigma_\varepsilon<1\}}\,|\,W_*<-b,\,Y_H>L\varepsilon)\right|\xrightarrow[\varepsilon\rightarrow0]{}0
\end{equation}
uniformly in $F$. Finally, \eqref{GG} and \eqref{Presk} give 
\[\limsup_{\varepsilon\rightarrow0}\left|\N_0^{(1)}(F(V_\varepsilon)\,|\,W_*<-b)-\N_0(F(V_\varepsilon)\,|\,W_*<-b)\right|\leq\delta.\]  
This concludes the proof, $\delta$ being arbitrary.
\end{proof}
The following theorem is a consequence of Proposition \ref{conditionnement volume}. As this proof is very similar to that of Theorem \ref{convergence conditionnée}, we will not give it here.
    \begin{theorem}\label{convergence conditionnée volume}
        For every $\delta>0$ and $b>0$, we can find $\varepsilon>0$ such that we can couple $\s$ under $\N_0^{(1)}(\cdot\,|\,W_*<-b)$ and $\overline{\mathcal{BP}}$ so that $B_\varepsilon(\s,\Gamma_b)$ and $B_\varepsilon(\overline{\mathcal{BP}})$ are isometric with probability at least $1-\delta$. Moreover, we have the convergence 
    \begin{equation}   (\s,\varepsilon^{-1}D,\Gamma_b)\xrightarrow[\varepsilon\rightarrow 0]{(d)}(\overline{\mathcal{BP}},\overline{D}_\infty,\overline{\rho}_\infty)
    \end{equation}
    in distribution for the local Gromov-Hausdorff topology. 
    \end{theorem}
    
\section{Properties of the space $\BP$}\label{section 6}

In this section, we establish various properties of the space $\BP$. The proofs mostly rely on coupling arguments and estimates for the Brownian sphere.

\subsection{Topological properties}

It was proved in \cite{brownianplane} that the Brownian plane has almost surely the topology of the plane. Even though we expect $\BP$ to have the same property, we can not apply the method used for the Brownian plane. Instead, we use the classification of non-compact surfaces, as in \cite[Corrolary 3.8]{classification}.
\begin{proposition}\label{Simply connected}
    Almost surely, every ball $B_r(\BP,\overline{\rho}_\infty)$ is included in a subset of $\BP$ which is homeomorphic to the disk.
\end{proposition}
\begin{proof}
    By scaling, it is enough to prove the result for $r=1$.
    Fix $\delta >0$. By Proposition \ref{localisation}, there exists $L>0$ such that 
    \[\p(B_1(\BP ,\overline{\rho}_\infty)\subset p_{\BP }(A^\infty_{L}\cup\widehat{A}^\infty_{L}))>1-\delta.\]
    Therefore, we just have to prove that for every $L>0$, $p_{\BP }(A^\infty_{L}\cup\widehat{A}^\infty_{L})$ is almost surely homeomorphic to the disk.\\
    Note that $A_L^\infty$ is encoded by a Bessel process of dimension 7 $(X^\infty_t)_{0\leq t\leq S_L^\infty}$ starting at 0 and stopped at its hitting time of $L$, together with two Poisson point measures $(\mathcal{N}_1^\infty|_{[0,S_L^\infty]},{\mathcal{N}}_2^\infty |_{[0,{S}^\infty_L]})$ with intensities described in Section \ref{presentation}. According to Proposition \ref{tree} and after shifting the labels by $-L$, the labelled tree $A^\infty_L$ (rooted at 0) is distributed as the tree $\T$ under $\N_0(\cdot\,|\,W_*<-L)$ (rooted at $T_L$).\\
    Then, we can define a new pseudo-distance on $A^\infty_L$ (we abuse notation by calling it $\Tilde{D}^\circ$), which is the restriction of the pseudo distance $D^\circ_\infty$ to $A^\infty_L\times A^\infty_L$
    \begin{equation*}
        \Tilde{D}^\circ(u,v)=Z^\infty_u+Z^\infty_b-2\inf_{w\in[u\wedge v,u\vee v]}Z^\infty_w\quad\quad\text{for every }u,v\in\T_\infty.
    \end{equation*}
    Note that this function is the same as \eqref{pseudo dist slice}. Finally, we construct another pseudo-distance $\Tilde{D}$ on $A^\infty_L$ from $\Tilde{D}^\circ$ as we did in Sections \ref{construction} (also see \ref{three arms} and \ref{presentation}). Note that the distribution of $A_L^\infty$ and the definition of $\Tilde{D}^\circ$ imply that $\Tilde{\s}=\left(A_L^\infty/\{\Tilde{D}=0\},\Tilde{D}\right)$ is distributed as a Brownian slice (under $\N_0(\cdot\,|\,W_*<-L)$). In particular, it has almost surely the topology of a disk. \\
    Note that we have a natural embedding $\phi$ of $\Tilde{\s}$ into $\BP$ given by 
    \[\phi(p_{\Tilde{\s}}(u))=p_{\BP}(u)\quad\text{for every $u\in A^\infty_L$}.\]
    Note that this application is continuous, because $\overline{D}_\infty\leq\Tilde{D}$. Moreover, it is also injective. To see this, note that Proposition \ref{Identification} also holds for the tree $\T_\infty$ (which is a consequence of Proposition \ref{f}). Then, using the fact that $\Tilde{D}^\circ$ is just the restriction of $\D^\circ$, we have
    \[\Tilde{D}(u,v)=0\longleftrightarrow \overline{D}_\infty(u,v)=0,\] 
    which gives the injectivity. Thus, $A_L^\infty$ being compact, $\phi$ realizes a homeomorphism onto its image. In particular, $p_{\BP}(A^\infty_L)$ is homeomorphic to the disk. By symmetry, this also holds for $p_{\BP}(\widehat{A}^\infty_L)$. \\
    To conclude, note that $p_{\BP}(A^\infty_L)\cap p_{\BP}(\widehat{A}^\infty_L)$ is just a connected portion of the geodesic $\Gamma_\infty$. Furthermore, almost surely, 
    \[\inf_{x\in p_{\BP}(A_L^\infty)}Z^\infty_x\neq\inf_{x\in p_{\BP}(\widehat{A}_L^\infty)}Z^\infty_x\]
   which implies that $p_{\BP}(A^\infty_L)\cap p_{\BP}(\widehat{A}^\infty_L)$ has a non-trivial boundary. By Van Kampen's theorem, $p_{\BP}(A^\infty_L\cup\widehat{A}^\infty_L)$ is also homeomorphic to the disk, which concludes the proof.  
\end{proof}
\begin{proposition}
            Almost surely, the space $\BP$ is homeomorphic to the plane. 
        \end{proposition}
\begin{proof}
    First, by Proposition \ref{Simply connected}, $\BP$ is a non-compact orientable surface. Moreover, note that $\BP\backslash p_{\BP}(A^\infty_L\cup\widehat{A}^\infty_L)$ is connected, because $p_{\BP}(\T_\infty\backslash A^\infty_L)$ and $p_{\BP}(\widehat{\T}_\infty\backslash \widehat{A}^\infty_L)$ are connected and have a non-empty intersection. Therefore, according to the classification of non-compact surfaces, almost surely, $\BP$ has the topology of the plane. 
\end{proof}
\subsection{The Brownian plane seen from a point at infinity along its geodesic ray}

Here, we prove Theorem \ref{convergence2}, which establish a connection between the space $\overline{\mathcal{BP}}$ and the Brownian plane $\mathcal{BP}$, introduced in \cite{brownianplane}. The Brownian plane $(\mathcal{BP},D_\infty,\rho_\infty)$ is a non-compact rooted metric space, that appears as the scaling limit of the $\mathrm{UIPQ}$, and as the local limit of the Brownian sphere around a typical point. The connection between $\BP$ and $\mathcal{BP}$ is reminiscent of similar results about the $\mathrm{UIPQ}$ proved in \cite{dieuleveut}. We denote by $\gamma=(\gamma_t)_{t\geq 0}$ the simple geodesic path from the root in the Brownian plane. By Proposition 15 of \cite{brownianplane}, this is almost surely the unique infinite geodesic path from the root.

\begin{proof}[Proof of Theorem \ref{convergence2}]
First, we claim that for every $\delta>0$ and $r>0$, there exists $\varepsilon_0$ such that for every $\varepsilon<\varepsilon_0$, we can construct $\s$ (under $\N_0(\cdot\,|\,W_*<-b)$) and $\mathcal{BP}$ on the same probability space in such a way that the equality 
\[B_r(\varepsilon^{-1}\cdot\s,\rho)=B_r(\mathcal{BP},\rho_\infty)\]
holds with probability at least $1-\delta$. This result was proved in (\cite{Geodesicstars}, Lemma 18) for the Brownian sphere under $\N_0(\cdot\,|\,W_*=-1)$, but the author was working with hulls equipped with the intrinsic distance (instead of balls equipped with the restriction of the distance). However, the proof can easily be adapted to obtain the stated result, by taking larger hulls so that the two distances coincide on smaller balls (we omit the details). Therefore, if we fix $\delta>0$, there exists $\varepsilon>0$ so that we can couple $(\s,\mathcal{BP})$ so that $B_{2\varepsilon}(\mathcal{BP},\rho_\infty)$ and $B_{2\varepsilon}(\s,\rho)$ are isometric with probability at least $1-\delta/2$. Note that on this event, for every 
$0<\eta<\varepsilon$,  $B_{\eta}(\mathcal{BP},\gamma_{\varepsilon})$ and $B_{\eta}(\s,\Gamma_{\varepsilon})$ are also isometric. Therefore, this implies that 
\begin{equation}
    d_{TV}(\mathcal{L}(B_{\eta}(\mathcal{BP},\gamma_{\varepsilon})),\mathcal{L}(B_{\eta}(\s,\Gamma_{\varepsilon})))<\delta/2.
\end{equation}
where $\mathcal{L}(X)$ stands for the law of $X$. Now, by Theorem \ref{coupling}, there exists $\eta_0>0$ such that for every $0<\eta<\eta_0$, 
\begin{equation}
    d_{TV}(\mathcal{L}(B_{\eta}(\BP,\overline{\rho}_\infty)),\mathcal{L}(B_{\eta}(\s,\Gamma_{\varepsilon})))<\delta/2.
\end{equation}
Thus, for $0<\eta<\min(\varepsilon,\eta_0)$,
\begin{equation}
    d_{TV}(\mathcal{L}(B_{\eta}(\mathcal{BP},\gamma_{\varepsilon})),\mathcal{L}(B_{\eta}(\BP,\overline{\rho}_\infty)))<\delta.
\end{equation}
Hence, there exists a coupling of $(\BP,\mathcal{BP})$ such that 
\[\p(B_{\eta}(\mathcal{BP},\gamma_{\varepsilon})=B_{\eta}(\BP,\overline{\rho}_\infty))>1-\delta\]
where the equality is in the sense of isometry between pointed compact metric spaces. For every $r>0$, the scale invariance of these spaces implies that 
\[\p\left(B_{r}(\frac{\eta}{r}\cdot\mathcal{BP},\gamma_{\varepsilon})=B_{r}(\frac{\eta}{r}\cdot\BP,\overline{\rho}_\infty)\right)=\p\left((B_{r}(\mathcal{BP},\gamma_{\frac{r\varepsilon}{\eta}})=B_{r}(\BP,\overline{\rho}_\infty)\right)>1-\delta.\]
The coupling result follows, with $b_0=\frac{r\varepsilon}{\min(\varepsilon,\eta_0)}$. Finally, as in Proposition \ref{convergence GHPU}, we can see that the coupling behaves well with respect to the volume measures and the distinguished geodesics, which implies the convergence in law for the local Gromov-Hausdorff-Prokhorov-Uniform topology, which concludes the proof. 
\end{proof} 

\subsection{Invariance properties}

The purpose of this section is to prove some invariance properties of the space $\BP$, and to use these results to obtain further information about this random space.\\

   Our first result relies on Theorem \ref{convergence conditionnée}, which can be used to perform the local limit on a specific point of the geodesic $\Gamma$. In particular, we can choose $b$ so that $\Gamma_b$ is the middle of the geodesic. This can be used to show that the random space $\BP$ is invariant by reversal of its infinite bigeodesic, which is not clear from its construction. This result will be used later to deduce some properties of geodesics.
    \begin{theorem}\label{Retournement}
        Let $\widehat{\Gamma}_\infty(t)=\Gamma_\infty(-t)$. Then $(\BP ,\Gamma_\infty)$ has the same law as $(\BP ,\widehat{\Gamma}_\infty)$. 
    \end{theorem}
    \begin{proof}
        Consider the Brownian sphere under the probability measure $\N_0(\cdot\,|\,W_*=-1)$. By the invariance of the Brownian sphere under rerooting, we know that $(\s,\Gamma)$ has the same law as $(\s,\widehat{\Gamma})$. In particular, $(B_\varepsilon(\s,\Gamma_{1/2}),\Gamma|_{[1/2-\varepsilon,1/2+\varepsilon]})$ has the same law as $(B_\varepsilon(\s,\Gamma_{1/2}),\widehat{\Gamma}|_{[1/2-\varepsilon,1/2+\varepsilon]})$ (we used the fact that $\Gamma_{1/2}=\widehat{\Gamma}_{1/2}$). In particular, $(B_\varepsilon(\s,\Gamma_{1/2}),\varepsilon^{-1}D,\Gamma|_{[1/2-\varepsilon,1/2+\varepsilon]})$ and $(B_\varepsilon(\s,\Gamma_{1/2}),\varepsilon^{-1}D,\widehat{\Gamma}|_{[1/2-\varepsilon,1/2+\varepsilon]})$ converge toward the same random variable. By Proposition \ref{convergence GHPU}, this means that $(B(\BP ,\overline{\rho}_\infty),\overline{D}_\infty,\Gamma|_{[-1,1]})$ and $(B(\BP ,\overline{\rho}_\infty),\overline{D}_\infty,\widehat{\Gamma}|_{[-1,1]})$ have the same law. The result follows from the scaling invariance of $\BP $.
    \end{proof}
    Next, we present the invariance under rerooting on its geodesic of the space $\BP$, which is a consequence of a similar property for the tree $\T_\infty$.
    \begin{proposition}\label{rerooting}
   For every $s\in \R$,  the random labelled tree  $(\T_\infty,\tau_s)$  has the same distribution as $(\T_\infty,\tau_0)$.
\end{proposition}
\begin{proof}
 We fix $s\in\R$, and work with the measure $\N_0(\cdot\,|\,W_*=-2)$. By Proposition \ref{conditionnement} and Theorem \ref{f}, for every $r>|s|$, we have 
 \begin{equation*}     B_r\bigg(\T,\varepsilon^{-1}d_\T,p_\T(T_1)\bigg)\xrightarrow[\varepsilon\rightarrow 0]{(d)}B_r\bigg(\T_\infty,d_{\T_\infty},\tau_0\bigg).
 \end{equation*}
 This implies that for every $0<r'<r-|s|$, 
 \begin{equation*}     B_{r'}\bigg(\T,\varepsilon^{-1}d_\T,p_\T(T_{1+s\varepsilon})\bigg)\xrightarrow[\varepsilon\rightarrow 0]{(d)}B_{r'}\bigg(\T_\infty,d_{\T_\infty},\tau_s\bigg)
 \end{equation*}
 which means that $(\T,\varepsilon^{-1}d_\T,p_\T(T_{1+s\varepsilon}))$ converges in distribution towards $(\T_\infty,d_{\T_\infty},\tau_s)$ for the local Gromov-Hausdorff topology. On the other hand, according to the remarks \ref{uniform1} and \ref{uniform2}, we also have the convergence in distribution :
 \begin{equation*}    \bigg(\T,\varepsilon^{-1}d_\T,p_\T(T_{1+s\varepsilon})\bigg)\xrightarrow[\varepsilon\rightarrow 0]{(d)}\bigg(\T_\infty,d_{\T_\infty},\tau_0 \bigg).
\end{equation*}
The result is obtained by comparing the two limits. 
\end{proof}
\begin{corollary}\label{translation}
    For every $s\in\R$, the random space $(\overline{\mathcal{BP}},\overline{D}_\infty,\Gamma_\infty(s+\cdot))$ has the same distribution as $(\overline{\mathcal{BP}},\overline{D}_\infty,\Gamma_\infty)$. 
\end{corollary}
\begin{proof}
    This is a consequence of Proposition \ref{rerooting}, and the fact that the random space $(\overline{\mathcal{BP}},\overline{D}_\infty,\Gamma_\infty(s+\cdot))$ can be obtained from $(\T_\infty^{(s)},\widehat{\T}_\infty^{(s)})$ by the same construction that we used to obtain  $(\overline{\mathcal{BP}},\overline{D}_\infty,\Gamma_\infty)$ from $(\T_\infty,\widehat{\T}_\infty)$, where $\T_\infty^{(s)}$ is the tree $\T_\infty$ rooted at $\tau_s$. 
\end{proof}
Before stating our last invariance result, we need to introduce some notations. Set 
\[\overline{\mathcal{HBP}}_1=p_{\BP}(\T_\infty)\quad\text{ and }\overline{\mathcal{HBP}}_2=p_{\BP}(\widehat{\T}_\infty).\]
We omit the details, but we also could have constructed $\overline{\mathcal{HBP}}_1$ (respectively $\overline{\mathcal{HBP}}_2$) from $\T_\infty$ (respectively $\widehat{\T}_\infty$), and then glue these two spaces together along their infinite bigeodesics $\Gamma_\infty^1$ and $\Gamma_\infty^2$. The resulting space has the same distribution as $\BP$ (this is a consequence of the fact that we glued the spaces along a geodesic). 
\begin{corollary}\label{earthquake}
    For every $s\in \R$, let $\BP^{(s)}$ be the random space obtained by gluing $\overline{\mathcal{HBP}}_1$ and $\overline{\mathcal{HBP}}_2$ along the curves $\Gamma_\infty^1(s+\cdot)$ and $\Gamma_\infty^2$. Then $(\BP^{(s)},\overline{D}_\infty^{(s)},\overline{\rho}_\infty^{(s)})$ has the same distribution as $(\BP,\overline{D}_\infty,\overline{\rho}_\infty)$. 
\end{corollary}
\begin{proof}
    Let $\overline{\mathcal{HBP}}^{(s)}$ stands for the space $\overline{\mathcal{HBP}}$ rooted at $\Gamma_\infty(s)$. The result follows from Proposition \ref{rerooting}, and the fact that the pair $(\overline{\mathcal{HBP}}^{(s)}_1,\overline{\mathcal{HBP}}_2)$ can be obtained from $(\T_\infty^{(s)},\widehat{\T}_\infty)$ in the same way that $(\overline{\mathcal{HBP}}_1,\overline{\mathcal{HBP}}_2)$ is obtained from $(\T_\infty,\widehat{\T}_\infty)$. 
\end{proof}
In words, Corollary \ref{earthquake} states that the space $\BP$ is invariant under shifting of one of its half-space alongside the geodesic $\Gamma_\infty$. Note that this operation has an interpretation in terms of hyperbolic geometry, where it corresponds to an elementary left-earthquake (see \cite{earthquake} for more details). 
\subsection{Study of geodesics}
The purpose of this section is to identify all the infinite geodesics in $\BP $. Our results are similar to the discrete results proved in \cite{UIPQgeodesic} and to the results obtained for the Brownian plane. \\
Recall that we can easily identify two geodesic rays from $\overline{\rho}_\infty$, which are $\Gamma^+_\infty$ and $\Gamma^-_\infty$. We will show that these are in fact the only geodesic rays from $\overline{\rho}_\infty$. We start by proving that in the Brownian sphere, there is a confluence property of geodesics to $\Gamma_b$ with the simple geodesic $\Gamma$. This result easily follows from estimates in \cite{uniqueness} and \cite{convergence}, and was probably clear to both authors. 
\begin{lemme}\label{intersection}
Almost surely, under the probability measure $\N_0(\cdot\,|\,W_*<-b)$, for any point $z\in\s$, there exists a geodesic between $z$ and $\Gamma_b$ that coalesces with the geodesic $\Gamma$ before hitting $\Gamma_b$. 
\end{lemme}
\begin{proof}
     We prove the statement under the probability measure $\N_0^{(1)}(\cdot\,|\,W_*<-b)$, which will extend to $\N_0(\cdot\,|\,W_*<-b)$ by scaling.\\
    For every $s\in[0,1]$, we can define the simple geodesic $\Phi_s$ between $\mathbf{p}(s)$ and $x_*$ in the following way. If $s\in[0,s_*]$, for $t\in[0,Z_s-Z_*]$, we set
    \begin{equation*}
        \phi_s(t)=\inf\{s'\geq s\,:\,Z_{s'}=Z_s-t\}
    \end{equation*}
    and
    \begin{equation*}
        \Phi_s(t)=\mathbf{p}(\phi_s(t)).
    \end{equation*}
     Similarly, for $s\in[s_*,1]$, we set 
    \begin{equation*}
        \phi_s(t)=\sup\{s'\leq s\,:\,Z_{s'}=Z_s-t\}
    \end{equation*}
    and
    \begin{equation*}
        \Phi_s(t)=\mathbf{p}(\phi_s(t)).
    \end{equation*}
    Using \eqref{Bound}, we can check that $\Phi_s$ is a geodesic between $\mathbf{p}(s)$ and $x_*$.\\
    For every $n\in\N^*$, set 
    \begin{equation*}
      \alpha_n=\inf_{s\in[0,(T_b-1/n)_+]\cup[(\widehat{T}_b+1/n)\wedge 1,1]}Z_s.  
    \end{equation*}
    We have $\alpha_n>-b$ almost surely, and by definition of a simple geodesic, for every $s\in[0,T_b-1/n]\cup[\widehat{T}_b+1/n,1]$, the simple geodesic $\Phi_s$ coincides with $\Gamma$ before $\Gamma_{(b-\alpha_n)/2}$, meaning that there exists $t$ such that $\Phi_s(t+\cdot)=\Gamma((b-\alpha_n)/2+\cdot)$. Therefore, by cutting the simple geodesic $\Phi_s$, we obtain a geodesic between $\mathbf{p}(s)$ and $\Gamma_b$ that coalesces with $\Gamma$ before hitting its endpoint. This proves the result for every $z\in\s$ of the form $z=\mathbf{p}(s)$ for some $s\in[0,T_b]\cup[\widehat{T}_b,1]$.  \\
    We then need to introduce some notations. For $\delta\in[0,b[$, we set 
    \begin{equation*}
        \eta_\delta(b)=\inf\big\{s>T_b:\mathbf{e}_s=\min_{t\in[T_b,s]}\mathbf{e}_t\text{ and }Z_s=-b+\delta\big\}.
    \end{equation*}
    Informally, $\mathbf{p}(\eta_\delta(b))$ is the first point in the ancestral line of $p_\T(T_b)$ with label $-b+\delta$ that we encounter in the contour exploration. Similarly, we set 
    \begin{equation*}
        \eta'_\delta(b)=\sup\big\{s<T_b:\mathbf{e}_s=\min_{t\in[s,T_b]}\mathbf{e}_t\text{ and }Z_s=-b+\delta\big\}.
    \end{equation*}
    Note that $\mathbf{p}(\eta_\delta(b))=\mathbf{p}(\eta'_\delta(b))$. \\
    Now, suppose that the statement of Lemma \ref{intersection} does not hold for every $z\in\s$ of the form $\mathbf{p}(t)$ with $t\in[T_b,\widehat{T}_b]$ with probability $k>0$. Note that the event $\{D(z,\Gamma(r))<D(z,\Gamma(r+\varepsilon))+\varepsilon\text{ and }D(z,\Gamma(r))<D(z,\Gamma(r-\varepsilon))+\varepsilon\}$ occurs if and only if there is no geodesic from $z$ to $\Gamma_b$ that coalesces with $\Gamma$ before $\Gamma_{b-\varepsilon}$ (or $\Gamma_{b+\varepsilon}$). Hence, by monotone convergence, there exists $\delta,\varepsilon>0$ such that :
    \begin{align}\label{long}
    \p(\{T_{b+\varepsilon}<\infty\}\cap\{\exists z\in\mathbf{p}([\eta_\delta(b),\eta'_\delta(b)]):&D(z,\Gamma(r))<D(z,\Gamma(r+\varepsilon))+\varepsilon\\
    &\text{ and }D(z,\Gamma(r))<D(z,\Gamma(r-\varepsilon))+\varepsilon\})>k/2.\nonumber
    \end{align}
    By \cite[Lemma 5.3]{uniqueness}, for every $0<\delta<b$ and $\kappa>0$, there exists constants $0<\beta<1$ and $C>0$ and an event $\mathcal{H}_{\delta,\kappa}$ such that, for every $0<\varepsilon<1$,
    \begin{align}\label{long1}
     \p(\mathcal{H}_{\delta,\kappa}\cap\{T_{b+\varepsilon}<\infty\}\cap&\{\exists z\in\mathbf{p}([\eta_\delta(b),\eta'_\delta(b)]):\\
     & D(z,\Gamma(r))<D(z,\Gamma(r+\varepsilon))+\varepsilon\text{ and }D(z,\Gamma(r))<D(z,\Gamma(r-\varepsilon))+\varepsilon\})\leq C\varepsilon^\beta.\nonumber
    \end{align}
    Moreover $\lim_{\delta,\kappa\rightarrow0}\p(\mathcal{H}_{\delta,\kappa}^c\cap\{T_b<\infty\})=0$.\\ 
    Then, we can choose $\kappa>0$ (and reduce $\delta$ and $\varepsilon$ if necessary) such that 
    \begin{equation*}
        \p(\mathcal{H}^c_{\delta,\kappa}\cap\{T_{b+\varepsilon}<\infty\})<k/4\quad\text{and}\quad C\varepsilon^\beta<k/4
    \end{equation*}
    (we can reduce $\delta$ if necessary). Thus, using \eqref{long1}, the term in \eqref{long} is bounded above by 
   \[\p(\mathcal{H}^c_{\delta,\kappa}\cap\{T_{b+\varepsilon}<\infty\})+C\varepsilon^\beta<k/2\]
    which gives us a contradiction with \eqref{long}, and concludes the proof. 
\end{proof}
\begin{proposition}\label{intersection2}
Almost surely, conditionally on $\{T_b<\infty\}$, all geodesics to $\Gamma_b$ coalesce with the simple geodesic $\Gamma$ before hitting their endpoint. 
\end{proposition}
\begin{proof}
    We argue by contradiction. Suppose the statement does not hold with probability $c>0$, and on this event, consider a geodesic $\eta:[0,T]\rightarrow\s$ from some element $z\in\s$ to $\Gamma_b$ that does not coalesce with the simple geodesic from $\rho$. By Lemma \ref{intersection}, for every $t\in]0,T[$, there exists a geodesic from $\eta(t)$ to $\Gamma_b$ that coalesces with the simple geodesic. By substitution, this gives an infinite number of geodesics between $z$ and $\Gamma_b$, which almost surely does not happen by Theorem $1.6$ of \cite{geodesic2}. The result follows. 
\end{proof}
In order to prove our result, we need a uniform control of geodesics from $\Gamma_b$ in the Brownian sphere. This kind of results is already known for geodesics from the root $\rho$ in \cite{geodesic1}. 
\begin{figure}
    \centering
    \includegraphics[scale=0.5]{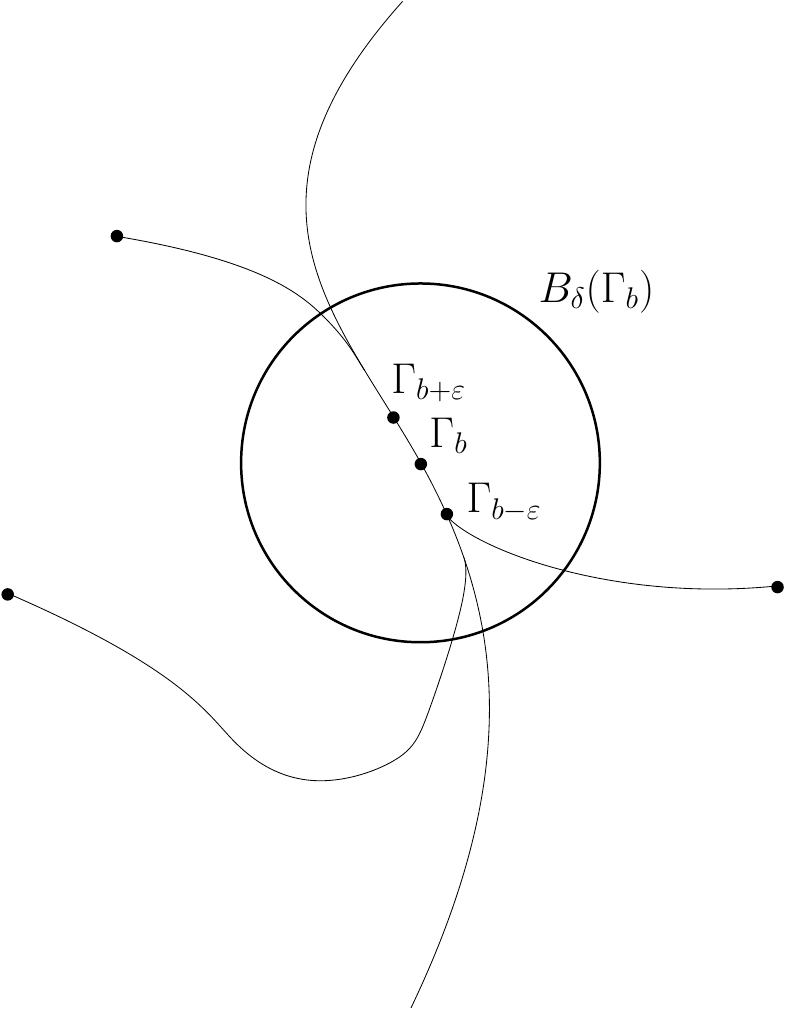}
    \caption{Illustration of the confluence of geodesics described in Proposition \ref{geodesic} ; every curve on this picture represents a geodesic to $\Gamma_b$.}
\end{figure}
\begin{proposition}\label{geodesic}
  Almost surely, for every $\delta>0$, there exists $\varepsilon>0$ such that,  all geodesics from $\Gamma_b$ to a point outside $B_\delta(\Gamma_b)$ coincide with the simple geodesic from $\rho$ on $[0,\varepsilon]$. 
\end{proposition}
\begin{proof}
    We argue by contradiction. If the statement does not hold, then we can find a sequence $(z_n)_{n\in\N}\in\s\cap B_\delta(\Gamma_b)^c$ and a sequence of paths $(\xi_n)_{n\in \N}$ such that for every $n\in\N$, $\xi_n$ is a geodesic between $\Gamma_b$ and $z_n$ such that $\Gamma_{b+1/n}\notin\xi_n([0,D(\Gamma_b,z_n)])$ and $\Gamma_{b-1/n}\notin\xi_n([0,D(\Gamma_b,z_n)])$ (which means that the geodesic $\xi$ leaves $\Gamma$ before time $1/n$). By compactness, we can suppose that $(z_n)_{n\in\N}$ converges toward $z\in\s$ such that $\D(z,\Gamma_b)\geq\delta$. Let 
    \begin{equation*}
      \alpha=\inf\{t\geq 0:\text{there exists a geodesic $\gamma$ from $\Gamma_b$ to $z$ such that $\gamma(t)\neq \Gamma(b+t)$ and $\gamma(t)\neq \Gamma(b-t)$}\}.  
    \end{equation*}
    By Proposition \ref{intersection2} and because there is only a finite number of geodesics between $\Gamma_b$ and $z$, $\alpha>0$ a.s. Now, note that the duration of the geodesic $(\xi_n)_{n\in \N}$ is bounded and that there are $1$-Lipschitz functions. Therefore, we can extract a subsequence that converges uniformly toward a random curve $\xi$. By assumption on the sequence $(z_n)_{n\in\N}$, $\xi$ must be a geodesic between $\Gamma_b$ and $z$. Now, because the subsequence of geodesic converges uniformly, it also converges for the Hausdorff distance. Therefore, by Theorem $1.1$ of \cite{geodesic2}, we have $\xi([\varepsilon,D(\Gamma_b,z)-\varepsilon])\subset\xi_n$ for $n$ large enough. But this implies that, for $n$ large enough, $\Gamma_{b-\alpha/2}$ or $\Gamma_{b+\alpha/2}$ belongs to $\xi_n([0,D(\Gamma_b,z_n)])$, which give us a contradiction if $n\geq\frac{1}{\alpha}$. This concludes the proof. 
\end{proof}
\begin{theorem}\label{rays}
   Almost surely, the only geodesic rays from $\overline{\rho}_\infty$ are $\Gamma^+_\infty$ and $\Gamma^-_\infty$. 
\end{theorem}
\begin{proof}
    We proceed as in Proposition $15$ of \cite{brownianplane}. By Theorem \ref{coupling}, for every $\delta>0$, we can find $\lambda$ large enough so that with probability at least $1-\delta$, the balls $B_1(\lambda\cdot\s,\Gamma_b)$ and $B_1(\BP ,\rho)$ are isometric. By Proposition \ref{geodesic}, we can find $\varepsilon>0$, which is random, such that all geodesic paths from $\Gamma_b$ to a point outside $B_1(\lambda\cdot\s,\Gamma_b)$ coincide with $\Gamma$ on $[0,\varepsilon]$. Let $\mathcal{R}$ denote the set of all geodesic rays from $\overline{\rho}_\infty$ and set 
    \begin{equation*}
        \tau=\inf\{t\geq 0:\exists\omega\in\mathcal{R},\,\omega(t)\neq\Gamma^+_\infty(t)\text{ or }\omega(t)\neq\Gamma^-_\infty(t)\}.
    \end{equation*}
    From the previous considerations, $\tau>0$ a.s. But, by the scaling invariance of $\BP $, $\lambda\cdot\tau$ has the same distribution as $\tau$ for $\lambda>0$. Therefore, $\tau=+\infty$ a.s, which concludes the proof. 
\end{proof}
We can identify the behavior of every geodesic ray in $\BP $, by adapting the previous proof and using the invariance under rerooting alongside $\Gamma_\infty$ of Corollary \ref{translation}. We start by generalizing Corollary \ref{geodesic}. 
\begin{figure}
    \centering
    \includegraphics[scale=0.5]{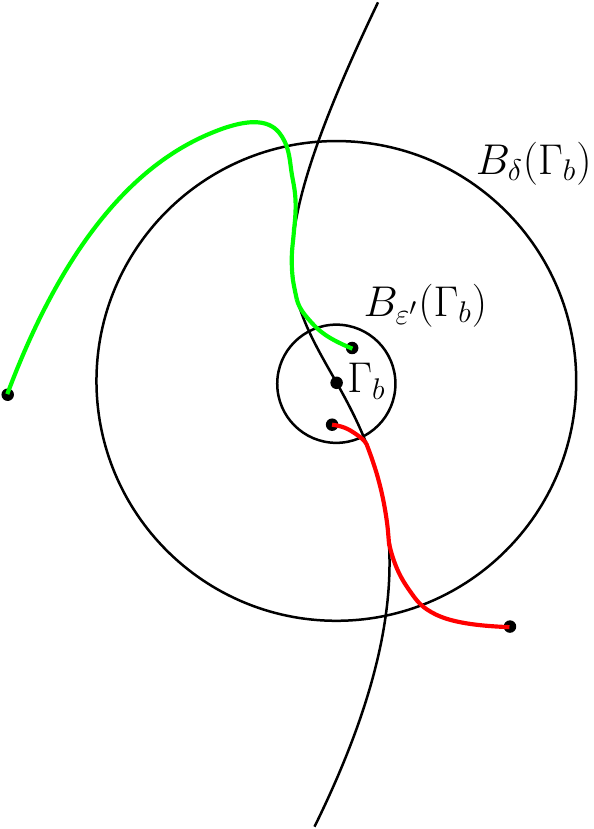}
    \caption{Illustration of Corollary \ref{geodesic2}.}
\end{figure}
\begin{corollary}\label{geodesic2}
     For every $\delta>0$, there exist two random variables $\varepsilon>0$ and $\varepsilon'>0$ such that, almost surely, all geodesics from a point of $B_{\varepsilon'}(\Gamma_b)$ to a point outside $B_\delta(\Gamma_b)$ coincide with $\Gamma$ on an interval of length at least $\varepsilon$. 
\end{corollary}
\begin{proof}
    We argue by contradiction. If the statement does not hold, we can find sequences $(x_k)_{k\in\N}$ and $(y_k)_{k\in\N}$, such that for every $k\in\N$, $D(\Gamma_b,x_k)<\frac{1}{k},D(\Gamma_b,y_k)>\delta$, and there exists a geodesic $\gamma_k$ between $x_k$ and $y_k$ that does not coincide with $\Gamma$ on any interval of length $\frac{1}{k}$. By the Arzela-Ascoli theorem, we can suppose that the family $(\gamma_k)_{k\in\N}$ converges uniformly toward a geodesic $\gamma$ between $\Gamma_b$ and some point $y\in\s$, such that $D(\Gamma_b,y)\geq\delta$, and that $d(\gamma,\gamma_k)<\frac{1}{k}$. \\
    By Proposition \ref{intersection2}, we know that $\gamma$ coincides with $\Gamma$ on some interval $[0,\alpha]$. By Theorem $1.1$ of \cite{geodesic2}, we know that $\gamma([\frac{1}{k},T-\frac{1}{k}])\subseteq\gamma_k$ for $k$ large enough. But this leads to a contradiction for $\alpha-\frac{1}{k}>\frac{1}{k}$, which concludes the proof.
     \end{proof}
 \begin{theorem}\label{coalesce}
    Almost surely, every geodesic ray $\omega$ of $\BP $ coalesces with $\Gamma_\infty$.
\end{theorem}
    \begin{proof}
        Fix $r>0$, and set :
        \begin{equation*}
            \mathcal{R}_r:=\{\omega:\omega\text{ is a geodesic ray starting in }B_r(\BP )\}
        \end{equation*} 
        and 
        \begin{equation*}
            \tau_r:=\sup\{t:\text{every geodesic }\omega\in\mathcal{R}_r\text{ coincides with $\Gamma_\infty$ on some interval of length at least $t$}\}.
        \end{equation*}
        By Corollary \ref{geodesic2} and the coupling result of Theorem \ref{convergence}, we know that $\tau_r>0$ a.s. But by Theorem \ref{rays} and Corollary \ref{translation}, almost surely, for every $q\in\Q$, the only infinite geodesics from $\Gamma_\infty(q)$ are $(\Gamma_\infty(q+t))_{t\geq 0}$ and $(\Gamma_\infty(q-t))_{t\geq 0}$. Therefore, because every $\omega\in\mathcal{R}_r$ passes through one of those points, it must coincide with $\Gamma$ after some time. This being true almost surely for every $r\geq 0$, the result follows.
    \end{proof}
We are now able to prove Theorem \ref{bigeodesic}.
    \begin{proof}[Proof of Theorem \ref{bigeodesic}]   
    Consider a bigeodesic $\gamma$ in $\BP$, and set $\gamma_+:=\gamma|_{\R_+}$. By Theorem \ref{coalesce}, there exists $t\geq 0$ such that $\gamma_+$ coalesces with $\Gamma_\infty$ before time $t$. Without loss of generality, suppose that $\gamma_+|_{[t,\infty[}=\Gamma_\infty|_{[r,\infty[}$ for some $r\in\R$. Then, for every $q\in\Q$ such that $q\geq r$, the path 
    \begin{equation*}
        \gamma_q(s)=
        \begin{cases}
            \Gamma_\infty(q-s)&\quad\text{if }0\leq s\leq q-r\\
            \gamma(t-(s-q+r))&\quad\text{if }s\geq q-r
        \end{cases}
    \end{equation*}
    is a geodesic ray from $\Gamma_q$. By Theorem \ref{rays} and Corollary \ref{translation}, we know that for every $q\in\Q$, $\Gamma_\infty(q)$ has the same role as $\overline{\rho}_\infty$, which implies that the only geodesic rays from $\Gamma_\infty(q)$ are $\Gamma_\infty(q+\cdot)$ and $\Gamma_\infty(q-\cdot)$. Therefore, almost surely, for every $q\in\Q$ such that $q\geq r$, 
    \[\gamma_q(s)=\Gamma_\infty(q-s)\quad\text{for every }s\geq 0.\]
    In particular, 
    \[\gamma(s)=\Gamma_\infty(s+r-t)\quad\text{for every }s\in\R,\]
    which concludes.
    \end{proof}
    \begin{proposition}
        Almost surely, there are at least two geodesic rays starting from every point $x\in\BP $.
    \end{proposition}
    \begin{proof}
        For every point $u\in\T_\infty$ (or $u\in\widehat{\T}_\infty$), we can define the simple geodesic $\gamma_u$, which coincides with the curve $\widehat{\Gamma}_\infty$ at some point, meaning that $\gamma_u(t+\cdot)=\Gamma_\infty(s-\cdot)$ for some $t\geq0,s\in\R$. \\
        However, by Theorem \ref{Retournement}, it means that there also exists a geodesic $\gamma'_u$ which coincides with $\Gamma_\infty$ after some point, meaning that $\gamma'_u(t'+\cdot)=\Gamma_\infty(s'+\cdot)$ for some $t'\geq0,s'\in\R$, which concludes the proof.
    \end{proof}
\printbibliography
\end{document}